\documentclass[reqno,a4paper,11pt]{amsart}
\usepackage{color}
\usepackage{amsmath,amssymb,amsfonts,amsthm,enumerate}
\usepackage{hyperref,mathrsfs,accents}
\usepackage[T1]{fontenc}
\usepackage[active]{srcltx}
\usepackage{epsfig}
\usepackage{graphicx}
\usepackage{cite}
\usepackage{tikz}
\usepackage{pgfplots}
\usetikzlibrary{shapes,arrows,shapes.geometric,patterns,fadings}
\usepackage{subfig}

\catcode`@=11 \@addtoreset{equation}{section} \catcode`@=12

\newtheorem{theorem}{Theorem}[section]
\newtheorem{lemma}[theorem]{Lemma}
\newtheorem{definition}[theorem]{Definition}
\newtheorem{remark}[theorem]{Remark}
\newtheorem{proposition}[theorem]{Proposition}
\newtheorem{corollary}[theorem]{Corollary}
\newtheorem{example}[theorem]{Example}

\newcommand{\N}{\mathbb{N}}
\newcommand{\R}{\mathbb{R}}

\newcommand{\Sph}{\mathbb{S}} 
\newcommand{\Mm}{\mathcal{M}}
\newcommand{\tMm}{\widetilde\Mm}
\newcommand{\MH}{\mathcal{M_{\Haus{}}}}
\newcommand{\PB}{\mathcal{PB}}

\newcommand{\Ff}{\mathfrak{F}}
\newcommand{\Gg}{\mathfrak{G}}
\newcommand{\Hh}{\mathfrak{H}}
\newcommand{\Jj}{\mathfrak{J}}
\newcommand{\Ii}{\mathfrak{I}}

\newcommand{\eps}{\varepsilon}

\newcommand{\weakto}{\rightharpoonup}

\newcommand{\Haus}[1]{{\mathcal H}^{#1}} 
\newcommand{\Leb}[1]{{\mathcal L}^{#1}} 
\newcommand{\Per}{P} 

\newcommand{\mean}[1]{\,-\hskip-1.08em\int_{#1}} 
\newcommand{\DM}{\mathcal{DM}}
\newcommand{\redb}{\partial^{*}} 

\newcommand{\baru}{\bar{u}}
\newcommand{\barp}{\bar{p}}
\newcommand{\tV}{\widetilde{V}}

\newcommand{\Tr}{\mathrm{Tr}}

\newcommand{\de}{\partial}

\newcommand{\restrict}{\mathbin{\vrule height 1.4ex depth 0pt width
0.13ex\vrule height 0.13ex depth 0pt width 1.3ex}\,}

\newcommand{\B}{\mathcal{B}}

\newcommand{\cS}{{\mathcal S}}
\newcommand{\cP}{{\mathcal P}}
\newcommand{\T}{\mathfrak{T}}
\newcommand{\ds}{\displaystyle}

\DeclareMathOperator{\sign}{sign} 

\renewcommand{\Subset}{\subset\!\subset}
\renewcommand{\Supset}{\supset\!\supset}

\newcommand{\Lip}{\mathrm{Lip}}

\AtBeginDocument{
  \let\div\relax
  \DeclareMathOperator{\div}{div}
}
\newcommand{\res}{\mathop{\hbox{\vrule height 7pt width .5pt depth 0pt
\vrule height .5pt width 6pt depth 0pt}}\nolimits}

\title{The prescribed mean curvature measure equation in non-parametric form}

\author[Gian Paolo Leonardi]{Gian Paolo Leonardi}
\address[Gian Paolo Leonardi]{Dipartimento di Matematica, Università di Trento, via Sommarive 14, IT-38123 Povo - Trento (Italy)}
\email{gianpaolo.leonardi@unitn.it}

\author[Giovanni E. Comi]{Giovanni E. Comi}
\address[Giovanni E. Comi]{Dipartimento di Matematica, Università di Bologna, Piazza di Porta San Donato 5, 40126 Bologna (Italy)}
\email{giovannieugenio.comi@unibo.it}%

\thanks{This work has been financially supported by GNAMPA - INdAM. Part of this work was undertaken while the second author visited the University of Trento. He would like to thank this institution for its support and warm hospitality during the visits. The authors are particularly grateful to Lorenzo Brasco and Giorgio Saracco for their support and encouragement during the preparation of the paper. The authors are also grateful to Virginia De Cicco for her comments on a preliminary version of the paper and to the anonymous referee for their detailed report, which helped us improve the final version of the paper.}

\makeatletter
\@namedef{subjclassname@2020}{\textup{2020} Mathematics Subject Classification}
\makeatother


\begin{document}

\begin{abstract}
We introduce a weak formulation of the non-parametric prescribed mean curvature equation with measure data and show the existence and several properties of $BV$ solutions under natural assumptions on the prescribed measure. Our approach does not rely on approximate or viscosity-type solutions and requires the combination of various ingredients, including Anzellotti's pairing theory for divergence-measure fields and its recent developments, a refinement of Anzellotti-Giaquinta approximation, and convex duality theory. We also prove a Gamma-convergence result valid for suitable smooth approximations of the prescribed measure, and a maximum principle for continuous weak solutions. We finally construct some examples of non-uniqueness, showing at the same time the need for the continuity assumption in the maximum principle, as well as an unexpected feature of weak solutions.
\end{abstract}

\maketitle

\tableofcontents

\section{Introduction}

Let $\Omega\subset \R^{n}$ be a bounded open set, and let 
$\mu$ be a signed Radon measure on $\Omega$ with finite total variation (denoted as $\mu \in \Mm(\Omega)$). Our objective is to study the \textit{Prescribed Mean Curvature Measure equation}
\begin{equation} \tag{PMCM}\label{eq:PMCM} \div \left (\frac{\nabla u}{\sqrt{1 + |\nabla u|^{2}}} \right) = \mu \quad \text{on} \ \Omega, \end{equation}
and particularly determine the appropriate weak formulation and the necessary and sufficient assumptions on the prescribed measure $\mu$ for the existence of solutions in $BV(\Omega)$.

\subsection{Context and literature review} 
The left-hand side of \eqref{eq:PMCM} is the \textit{minimal surface operator} applied to $u$. When $u\in C^{2}(\Omega)$, the minimal surface operator evaluates the mean curvature of the graph of $u$ at any point $(x,u(x))$. In the case $\mu = H\, dx$ with $H(x)$ Lipschitz-continuous on $\Omega$, equation \eqref{eq:PMCM} reduces to the classic prescribed mean curvature equation in non-parametric form, which has been studied by several authors in the past (see, e.g., \cite{Gerhardt1974, Gia74, Giu78, FG84, massari1974esistenza}). 

The mathematical theory of capillarity, with its long history starting from the seminal works of Young, Lagrange, Laplace, and Gauss, represents one of the main motivations for studying prescribed mean curvature problems. The classical theory of capillarity focuses on the case where the forces acting on the interface, besides surface tension, are smooth or at least Lipschitz. A notable exception is provided by the theory of minimal surfaces with discontinuous or thin obstacles, see \cite{nitsche1969variational} and \cite{giusti1972non, giusti1973CIME}, as well as the more recent \cite{FocardiSpadaro2018, FocardiSpadaro2020}, with references therein. However, a complete mathematical model of capillarity in the presence of singular forces - despite its theoretical and applicative interest, such as in modeling how some animals and/or mechanical devices can rest or move on the water surface without sinking  \cite{hu2010water} - has not been fully developed yet. 

Partial results in this direction have been obtained by Ziemer \cite{ziemer1995nonhomogeneous}, who proved the existence and local boundedness of $BV$ solutions of \eqref{eq:PMCM} when $\mu$ is a nonnegative measure satisfying $\mu(B_{r}(x)) \le C r^{q(n-1)}$ with a suitably small constant $C$, for any ball contained in $\Omega$ and for some $1<q\le \frac{n}{n-1}$. Another contribution has been given by Dai, Trudinger, and Wang \cite{dai2012mean}, who studied viscosity-type solutions of \eqref{eq:PMCM} when $\mu$ is nonnegative. However, the restrictions on the prescribed measure in \cite{ziemer1995nonhomogeneous} are so strong that most of the relevant and nontrivial features of the problem are lost (as they exclude measures concentrated on $(n-1)$-dimensional sets). Furthermore, the notion of viscosity solution proposed in \cite{dai2012mean}, besides being provided only for nonnegative measures, does not cover the entire class of variational solutions, thus missing some important non-uniqueness phenomena that are instead captured by the theory developed in this paper for the first time. 


A necessary condition for the existence of a classical solution to \eqref{eq:PMCM} can be derived by integrating the equation over smooth subdomains. Assuming that $u\in C^{2}(\Omega)$ satisfies \eqref{eq:PMCM}, we define 
\[
T(u)(x) := \frac{\nabla u(x)}{\sqrt{1+|\nabla u(x)|^{2}}}\,.
\]
Consequently, we have
\[
|\mu(A)| = \left|\int_{A} \div(T(u))(x)\, dx\right| = \left|\int_{\de A} T(u)(x)\cdot \nu_{A}(x)\, d\Haus{n-1}(x)\right| < \Per(A)
\]
for any nonempty $A\Subset \Omega$ with smooth boundary, 
where $\nu_{A}$, $\Haus{n-1}$, and $\Per(A)$ denote the outer normal to $\de A$, the Hausdorff measure of dimension $n-1$, and the perimeter of $A$, respectively. Note that the inequality is strict because $|T(u)(x)| <1$ for all $x\in \Omega$. In the seminal paper by Giusti \cite{Giu78}, it was shown that when $\mu = H\, dx$ with $H$ and $\de\Omega$ both Lipschitz, the necessary condition
\begin{equation} \label{eq:NCclassic} 
\left|\int_{A}H(x)\, dx\right| < \Per(A),\qquad \forall\, A\Subset \Omega\quad \text{nonempty, open and smooth,}
\end{equation}
is also sufficient for the existence of solutions. Recently, an extension of Giusti's result to a larger class of domains called \textit{weakly-regular} has been achieved \cite{LS18a}. The proof method involves distinguishing between two cases: \textit{non-extremal} and \textit{extremal}. In the non-extremal case, there exists a constant $0<L<1$ such that the aforementioned condition \eqref{eq:NCclassic} is replaced by the stronger
\begin{equation}\label{eq:NEclassic} 
\left|\int_{A}H(x)\, dx\right| \le L\, \Per(A),\qquad \forall\, A\Subset \Omega\quad \text{smooth.}
\end{equation}
In the extremal case, \eqref{eq:NCclassic} holds true and the inequality is saturated over the entire $\Omega$:
\[
\left|\int_{\Omega}H(x)\, dx\right| = \Per(\Omega)\,.
\]
In the first, non-extremal case, the variational nature of the problem is exploited, and indeed a solution can be found by minimizing an auxiliary energy functional of the form
\[
\Jj[u] =  \sqrt{1+|D u|^{2}}(B) + \int_{\Omega} u \, H \, dx
\]
among $BV$ functions defined on a ball $B\Supset \Omega$ that coincide with a given function $\phi\in W^{1,1}(B)$ (a \textit{weak Dirichlet datum}) on $B\setminus \Omega$. Here, the term $\sqrt{1+|D u|^{2}}(B)$ denotes the area functional, that is, the generalized area of the graph of $u$ in the cylinder $B \times \R$.
The functional $\Jj$ is lower-semicontinuous with respect to $L^{1}$ convergence since it is the sum of a lower semicontinuous functional and a continuous integral term. The key requirement here is the compactness of sequences of $BV$ functions with equibounded energy, which directly follows from the non-extremality assumption \eqref{eq:NEclassic}.

In the extremal case, the proof involves approximating $\Omega$ from the inside using a sequence of smooth subdomains falling within the non-extremal case. It is then shown that the corresponding sequence of variational solutions admits a locally convergent subsequence, up to vertical translations. Although the extremal case, along with its characterization, is more intricate due to the lack of compactness and the possibility of having solutions with infinite $BV$ norm, it is evident that the non-extremal case serves as the fundamental first step in the whole proof. This paper is thus devoted to the study of the non-extremal case, even though in Section \ref{sec:admissible} we provide a generalized definition also of the extremality condition.

\subsection{Description of the main results} 
The primary goal of the paper is to propose the appropriate weak formulation of \eqref{eq:PMCM}. We will focus on the case when $\Omega$ is weakly regular (or even Lipschitz) and $\mu$ is a \textit{non-extremal measure}, meaning there exists $0<L<1$ such that $|\mu(A)|\le L\, \Per(A)$ for all open sets $A\Subset \Omega$ with smooth boundary. Another equivalent definition of non-extremality will be given in Section \ref{sec:admissible}. 
The existence and properties of solutions to the weak formulation, as well as the variational approximation issue, pose significant technical challenges, as we will explain later. 

To identify a natural class of measures, we observe that the condition of being the distributional divergence of a sub-unitary (or bounded) vector field imposes restrictions on $\mu$, particularly that $|\mu|$ is absolutely continuous with respect to the $(n-1)$-dimensional Hausdorff measure (see \cite{CF1}). Thus we conveniently introduce the class $\MH(\Omega)$ of such measures. In other words, we must exclude measures that are (partially) concentrated on sets of Hausdorff dimension less than $n-1$. Further conditions on $\mu$, such as the property of being \textit{admissible}, that is, the fact that its total variation $|\mu|$ belongs to the dual of the space $BV(\Omega)$ (see Definition \ref{def:muadmissible} and \cite{ComiLeo} for more details), will be discussed later on.

The first step in deriving the weak formulation of \eqref{eq:PMCM} relies on an alternative characterization of the vector field $T(u)$. Following the approach of Scheven and Schmidt \cite{scheven2016bv}, one notes that $T(u)$ is the unique sub-unitary vector field that satisfies the identity
\begin{equation*} 
T(u) \cdot \nabla u = \sqrt{1 + |\nabla u|^{2}} - \sqrt{1 - |T(u)|^{2}}
\end{equation*}
in a pointwise sense. This can be shown through a straightforward computation (assuming $u\in C^{1}$ for simplicity). Furthermore, the right-hand side of the identity can be interpreted as a difference between measures: the area of the graph minus the absolutely continuous measure with an $L^{\infty}$ density given by $\sqrt{1-|T(u)|^{2}}$. Therefore, it is natural to assign a consistent measure-theoretic meaning to the scalar product between $T(u)$ and $\nabla u$. This is achieved through the theory of divergence-measure fields and their pairings with the distributional gradients of weakly differentiable functions, initially proposed in the seminal work of Anzellotti {and then developed by many authors, see \cite{Anzellotti_1983,ACM,CF1,comi2017locally,Phuc_Torres,scheven2016bv,crasta2017anzellotti,crasta2019pairings,comi2022representation,crasta2022variational,Silhavy1,Frid2}. In particular, our recent contribution \cite{ComiLeo} contains an up-to-date list of references and new results that further extend the theory, and at the same time are crucial for this paper. 

It is worth observing that the theory of pairings between essentially bounded divergence-measure fields and scalar functions with bounded variation has already proven useful when studying the boundary behavior of solutions to prescribed mean curvature equations in domains with low regularity \cite{MR4385590, LS18a}, and in relation with super and sub-solutions to the same kind of equations \cite{scheven2016bv, scheven2018dual}. See also \cite{giachetti2023bounded} for another recent application of the theory of pairings in the presence of non-linear forcing terms.

For the reader's convenience, we briefly introduce the so-called $\lambda$-pairing, that will be needed later on. We denote as $\DM^{\infty}(\Omega)$ the class of $L^{\infty}$ vector fields with divergence measure. Let $F\in \DM^{\infty}(\Omega)$, $u\in BV(\Omega)\cap L^{\infty}(\Omega)$, and $\lambda:\Omega\to [0,1]$ be a given Borel function. One can define the $\lambda$-pairing distribution as follows:
\[
(F, Du)_{\lambda} := \div(uF) - u^{\lambda}\, \div F\,,
\]
where $u^{\lambda}$ is defined $\Haus{n-1}$-almost everywhere as the convex combination of $u^{+}$ and $u^{-}$ (the upper and lower approximate limits of $u$) using the coefficients $\lambda$ and $1-\lambda$, respectively. Note that the $\lambda$-pairing is a measure \cite{crasta2019pairings}. The need for this generalization of the classical pairing (i.e., the one obtained when $\lambda \equiv \frac{1}{2}$) is soon explained by considering the counterpart of the auxiliary functional $\Jj[u]$ in the case of a prescribed mean curvature measure:
\[
\Jj_{\mu}[u] := \sqrt{1+|D u|^{2}}(B) + \int_{\Omega} u\, d\mu\,,
\] 
where $B$ is a fixed ball compactly containing $\Omega$.
The problem with this expression is that $u\in BV(B)$ cannot be uniquely determined on null sets, which can be problematic when $\mu$ has a nontrivial singular part with respect to the Lebesgue measure, leading to uncertainty about the meaning of the integral of $u$ with respect to $\mu$. Thus, we need to choose a representative of $u$ defined almost everywhere with respect to $\Haus{n-1}$. This can be achieved by selecting an appropriate Borel function $\lambda$, which becomes an additional variable in the variational formulation:
\[
\Jj_{\mu}[u,\lambda] := \sqrt{1+|D u|^{2}}(B) + \int_{\Omega} u^{\lambda}\, d\mu\,.
\]
Given $\lambda$, we can provide the following characterization of minimizers of $\Jj_{\mu}[\cdot,\lambda]$ in terms of a weak formulation of \eqref{eq:PMCM}.

\begin{definition}\label{def:PMCM}
We say that $u \in BV(\Omega)$ is a \textit{weak solution to the prescribed mean curvature measure equation} if there exist $T \in L^{\infty}(\Omega;\R^{n})$ and a Borel function $\lambda:\Omega\to [0,1]$ such that
\begin{align} 
\label{L_infty_bound_eq} \|T\|_{L^{\infty}(\Omega; \R^{n})} & \le 1,\\
\label{divergence_eq} \div T & = \mu\, \text{ on } \Omega, \\
\label{pairing_eq} (T, Du)_{\lambda} & = \sqrt{1 + |D u|^{2}} - \sqrt{1 - |T|^{2}} \Leb{n} \, \text{ on } \Omega,
\end{align}
where the last two identities involve scalar Radon measures in $\Mm(\Omega)$, and $(T,Du)_{\lambda}$ denotes the $\lambda$-pairing between $T$ and $Du$.
\end{definition}

Through the rest of the paper, we say that $u$ is a weak solution to \eqref{eq:PMCM} if it satisfies the conditions of Definition \ref{def:PMCM}.

We can minimize $\Jj_{\mu}[u,\lambda]$ first with respect to $\lambda$, which corresponds to choosing $\lambda=\lambda_{\mu}$ as the characteristic function of a Borel set satisfying $(1-\lambda)\mu = \mu^{+}$, where $\mu^\pm$ are the positive and negative parts of $\mu$, respectively (see Lemma \ref{lem:min_J_lambda}). In other words, if $\mu = \mu^{+} - \mu^{-}$, then $\lambda_{\mu}$ is such that:
\[
\int_{\Omega} u^{\lambda_{\mu}}\, d\mu = \int_{\Omega} u^{-}\, d\mu^{+} - \int_{\Omega} u^{+}\, d\mu^{-}\,.
\]
This provides the natural candidate $\lambda$-pairing to be used in the weak formulation of \eqref{eq:PMCM}, and we write $\Jj_{\mu}[u,\lambda_\mu] = \Jj_{\mu}[u]$.

One can demonstrate that, if $u$ is a weak solution of \eqref{eq:PMCM}, such that $u\in W^{1,1}(\Omega)$, then the vector field $T$ satisfying \eqref{L_infty_bound_eq}--\eqref{divergence_eq} coincides with $\frac{\nabla u}{\sqrt{1+|\nabla u|^{2}}}$ almost everywhere on $\Omega$ (we will prove this result in Section \ref{sec:further1} for any weak solution in $BV(\Omega)$). Moreover, if $\mu = H\, dx$ with $H$ being Lipschitz (as assumed in \cite{Giu78}), then $u\in C^{2}(\Omega)$ and the weak formulation reduces to the classical one. Additionally, using \eqref{L_infty_bound_eq} and \eqref{divergence_eq} along with the divergence theorem, we obtain a necessary condition for the existence of weak solutions of \eqref{eq:PMCM}:
\begin{equation} \label{eq:intro_nec_cond_A}
|\mu(A)| \le \Per(A)\,,\quad \forall\, A\Subset \Omega\ \text{smooth.}
\end{equation}
However, it should be noted that we cannot enforce strict inequality as in \eqref{eq:NCclassic} because in this general case we cannot guarantee that $|T|<1$ $\Haus{n-1}$-a.e. on $\Omega$ (even when $u\in W^{1,1}(\Omega)$, it is possible to have $|T|=1$ on some Lebesgue null set). More general versions of the necessary condition, which also depend on the (weak) regularity of $\Omega$, are proven in Lemma \ref{necessary_cond}.

To demonstrate the existence of weak solutions of \eqref{eq:PMCM} we employ the direct method of Calculus of Variations, which involves proving coercivity and lower semicontinuity for the functional 
\[
\Jj_{\mu}[u] := \sqrt{1+|Du|^{2}}(B) + \int_{\Omega}u^{-}\, d\mu^{+} - \int_{\Omega} u^{+}\, d\mu^{-}\,,
\]
see Section \ref{sec:coercivity} and Theorem \ref{thm:sci}. This functional is restricted to the class of $BV$ functions on $B$ that satisfy a weak Dirichlet boundary condition determined by a prescribed function $\phi\in W^{1,1}(B)$. Notably, we establish lower semicontinuity for measures of a specific form; that is, admissible non-extremal measures $\mu \in \MH(\Omega)$ which can be decomposed into three parts: an absolutely continuous part with density in $L^q(\Omega)$ for some $q > n$, a compactly supported divergence-measure of a vector field which is continuous outside of a compact $\Haus{n-1}$-negligible set, and a Radon measure concentrated on another compact set, which is Lebesgue-negligible and disjoint from the discontinuity set of the aforementioned vector field (for more details, see Definition \ref{def:tMomega}). These assumptions play essential roles in the proof of the lower semicontinuity of $\Jj_\mu$ (see Theorems \ref{thm:liminf_mu_2} and \ref{thm:liminf_mu_3}). However, it is natural to expect that the result holds for any admissible non-extremal measure. The main challenge in proving lower semicontinuity arises from the fact that while the area term $\sqrt{1+|Du|^{2}}(B)$ is lower semicontinuous with respect to $L^{1}$ convergence, the other integral terms in $\Jj_{\mu}[u]$ are not, except for the part with $q$-summable density, for $q > n$. To overcome this issue, we exploit the non-extremality assumption on $\mu$ to show that any loss of semicontinuity in the integral terms is offset by a continuity drop for the area term. The part of the measure which is the divergence of a continuous vector field (up to an $\Haus{n-1}$-negligible compact set) can be dealt with by exploiting integration by parts formulas and Leibniz rules for such field and $BV$-functions. Finally, in the case of the singular part of the measure $\mu$, we accomplish this trading between the area functional and the integral term by combining several techniques. First, we start by assuming that lower semicontinuity does not hold, and we consider a sequence of functions obtained through contradiction, which is converging in the $L^{1}$ sense to a limit function $u\in BV(\Omega)$. Then, we exploit our $\lambda$-approximation result (Theorem \ref{thm:smooth_lambda_approx}) proved in the companion paper \cite{ComiLeo} to show that we can take a sequence of smooth functions, without loss of generality (see also Lemma \ref{lem:density_energy}). After this step, we apply the uniform truncation to such sequence, via Lemma \ref{lem:apriori1} and Lemma \ref{lem:apriori2}. Then, we employ a finer truncation argument by relying on Lemma \ref{lem:sopraupiueps}, and use the non-extremality of the measure $\mu$ by passing to the super and sublevel sets of the sequence of functions, and by exploiting the coarea and Cavalieri's formulas.

Once a minimizer $\baru$ of the functional $\Jj_{\mu}$ is found, in Section \ref{sec:duality} we prove that it is a weak solution of \eqref{eq:PMCM}. This is accomplished through a duality argument that establishes the existence of a vector field $T\in \DM^\infty(\Omega)$ such that the pair $(\baru,T)$ satisfies equations \eqref{L_infty_bound_eq}, \eqref{divergence_eq} and \eqref{pairing_eq} with $\lambda = \lambda_{\mu}$.

In the rest of the paper, we deal with other aspects of the prescribed mean curvature measure problem. In particular, we establish a Gamma-convergence result for a sequence of functionals $\Jj_{\mu_{j}}$ (Theorem \ref{thm:gammaconv}), where $\mu_{j}$ is a suitable sequence of $C^{\infty}$ functions on $\Omega$ and the measures $\mu_{j}\Leb{n}$ converge to $\mu$ in weak-$\ast$ sense, as $j\to+\infty$. The construction of $\mu_{j}$ from $\mu$ involves a delicate Anzellotti-Giaquinta type regularization of a vector field $F$ such that $\div F = \mu$ (see Proposition \ref{prop:muapprox}, which is proved in \cite{ComiLeo}). In particular, this construction ensures that the measures $\mu_j$ are uniformly non-extremal for $j$ large enough, and this implies the uniform coercivity of $\Jj_{\mu_{j}}$ for $j$, from which we deduce the $\Gamma$-compactness property stated in Proposition \ref{prop:gammacompact}. It is up to now an open question whether the simple weak-$*$ convergence of $\mu_{j}$ to $\mu$ is enough for $\Jj_{\mu_{j}}$ to Gamma-converge to $\Jj_{\mu}$.

Then, we prove the uniqueness of the vector field $T=T(u)$ associated with a weak solution $u$ (see Theorem \ref{thm:Tuformula}), we list some identities satisfied by the singular part of the $\lambda$-pairing, and we obtain a Lipschitz regularity result for $u$ under some additional assumptions on the field $T$, see Section \ref{sec:further1}.

We also prove a maximum principle valid for continuous weak solutions (Theorem \ref{thm:maxprinc}), which implies the uniqueness of such solutions having the same trace at the boundary of $\Omega$ (Corollary \ref{corol:uniqueness}). We stress that the continuity assumption is crucial, see Examples \ref{ex:one_dim_non_uniq_non_Schmidt} and \ref{ex:two_balls}. 

Then, we provide some concrete examples of solutions. First, we discuss the one-dimensional case, for which we prove a convexity result for solutions in the case of a nonnegative measure $\mu$, and we provide examples of discontinuous and non-unique solutions in the case $\mu$ does not have a constant sign. In the $n$-dimensional case for $n \ge 2$, we characterize radial solutions on annuli with a prescribed radial measure supported on a sphere. We utilize this preliminary analysis to construct a one-parameter family of radial weak solutions of \eqref{eq:PMCM}, with the measure $\mu$ nonnegative, radial, and supported on two spheres. These solutions have the same trace at the boundary of $\Omega$. We remark that these non-uniqueness phenomena, Examples \ref{ex:one_dim_non_uniq_non_Schmidt} and \ref{ex:two_balls}, are completely new compared to the classical theory. 

Finally, in the Appendix we provide a technical decay result for the $BV$-energy of local truncates of $BV$ functions, that we were unable to find in the existing literature.

Furthermore, we believe that the theory developed in this paper can be successfully applied to more general classes of operators with linear growth at infinity, such as
\[
{\mathcal L}u = \div \frac{\nabla u}{\sqrt{a^{2} + |\nabla u|^{2}}}
\] 
with $a \in [0,1]$. Notably, the extreme values $0$ and $1$ correspond respectively to the $1$-Laplacian and the minimal surface operator. Future works will focus on exploring interesting topics such as the optimal regularity of weak solutions, extending the semicontinuity result to all admissible non-extremal measures, and analyzing the extremal case.
\medskip

After publishing the initial version of the paper on ArXiv, we were informed by Thomas Schmidt about his simultaneous preprint \cite{schmidt2023isoperimetric}. His preprint discusses semicontinuity and existence results for the prescribed mean curvature problem with measure data in the parametric sense of Massari. The technique used by Schmidt is quite different from ours, even though based on a similar compensation between the perimeter and the measure $\mu$.     

\section{Preliminaries}
\label{sec:prelim}

We denote by $\Leb{n}$ the Lebesgue measure, and by $\Haus{m}$ the $m$-dimensional Hausdorff measure, for $m \in [0, n]$, although we shall focus on the case $m = n-1$. Given $x \in \R^n$ and $r > 0$, we denote by $B_r(x)$ the open ball centered in $x$ with radius $r$, and, if $x = 0$, we simply write $B_r$. Unless otherwise stated, $\Omega \subset \R^n$ is an open set. The space of finite signed Radon measures on $\Omega$ is denoted as $\Mm(\Omega)$, and we set 
\begin{equation}\label{def:MH}
\MH(\Omega) := \{ \mu \in \Mm(\Omega) : |\mu|(B) = 0 \text{ for all Borel sets } B \subset \Omega \text{ such that } \Haus{n-1}(B) = 0 \}\,,
\end{equation}
where $|\mu|$ denotes the total variation of $\mu$. 
Any $\mu \in \Mm(\Omega)$ can be written as $\mu = \mu^{+} - \mu^{-}$, where $\mu^+$ and $\mu^-$ are the positive and the negative part of $\mu$, respectively. Moreover, there exist mutually disjoint Borel sets $\Omega_+,\Omega_-$ such that $\Omega = \Omega_+ \cup \Omega_{-}$, and $\mu_{+} = \mu \restrict \Omega_+$, $\mu_{-} = -\mu \restrict \Omega_-$. We also have $|\mu| = \mu^+ + \mu^-$. In addition, by Lebesgue-Besicovitch differentiation theorem, for $|\mu|$-a.e. $x \in \Omega$ we have
\begin{equation} \label{eq:mu_pm_differentiation}
\frac{d \mu^\pm}{d |\mu|}(x) = \begin{cases} 1 & \text{ if } x \in \Omega^{\pm} \\
0 & \text{ if } x \in \Omega^{\mp}.
\end{cases}
\end{equation}

Given $\mu \in \Mm(\Omega)$, by Radon-Nikodym and Lebesgue's decomposition theorems we can write $\mu = \mu^{ac} + \mu^s$, where $\mu^{ac}= g \Leb{n}$, for some $g \in L^1(\Omega)$, is the absolutely continuous part and $\mu^s$ is the singular part of $\mu$. 

We say that $u \in L^1(\Omega)$ is a function of bounded variation, and we write $u \in BV(\Omega)$, if its distributional gradient $Du$ is a vector-valued Radon measure whose total variation $|Du|$ is a finite measure on $\Omega$. $BV(\Omega)$ is a Banach space once equipped with the norm
$$ \|u\|_{BV(\Omega)} = \|u\|_{L^1(\Omega)} + |Du|(\Omega).$$

Given $u \in BV(\Omega)$, we say that a sequence $(u_j)_{j \in \N}$ converges to $u$ in $BV(\Omega)$-strict if
\begin{equation*}
\lim_{j \to + \infty} \|u - u_j\|_{L^1(\Omega)} + \big| |Du|(\Omega) - |Du_j|(\Omega) \big | = 0.
\end{equation*}

For a given $u \in BV(\Omega)$ we define the measure $\sqrt{1 + |Du|^{2}}$ on any open set $U\subset \Omega$ as 
\begin{equation*} 
\int_{U} \sqrt{1 + |Du|^{2}} := \sup \left \{ \int_{\Omega} \eta + u \div \phi \, dx : (\phi, \eta) \in C^{1}_{c}(U; \R^{n} \times \R), | (\phi, \eta) | \le 1 \right \}, 
\end{equation*}
see \cite{giusti1984minimal}. If we denote the absolutely continuous part of $Du$ by $\nabla u \, \Leb{n}$, and the singular part by $D^{s} u$, we have
\begin{equation} \label{area_factor_decomposition_eq} 
\sqrt{1 + |Du|^{2}} = \sqrt{1 + |\nabla u|^{2}} \Leb{n} + |D^{s} u|. 
\end{equation}
Given a measurable set $E$, the perimeter of $E$ in $\Omega$ is defined by $\Per(E;\Omega) = |D\chi_{E}|(\Omega)$. When $\Omega = \R^{n}$ we simply write $\Per(E)$. We refer to \cite{AFP, maggi2012sets} for the definition and properties of $\redb E$, the reduced boundary of $E$, for which one has $\Per(E;\Omega) = \Haus{n -1}(\redb E \cap \Omega)$.

When $\Omega$ has a bounded and Lipschitz boundary, we denote by ${\rm Tr}_{\partial \Omega}(u)$ the trace of $u$ over $\partial \Omega$ and recall that the trace operator ${\rm Tr}_{\partial \Omega}:BV(\Omega)\to L^{1}(\partial \Omega; \Haus{n-1})$ is linear, surjective and continuous with respect to the topology induced by the strict convergence (see for instance \cite[Theorem 3.88]{AFP}). In the following theorem, we recall the well-known Poincar\'e-Wirtinger and Poincar\'e-trace inequalities.
\begin{theorem}\label{thm:poincare}
Let $\Omega\subset \R^{n}$ be an open, bounded and connected set with Lipschitz boundary. Assume that $G$ is a linear and continuous functional on $BV(\Omega)$, such that $G(\chi_{\Omega})=1$. Then there exists a constant $C>0$ depending only on $\Omega$, such that for all $u\in BV(\Omega)$ one has the Poincar\'e-Wirtinger inequality
\begin{equation}\label{eq:poincarewirt}
\|u - G(u)\|_{L^{\frac{n}{n-1}}(\Omega)} \le C |Du|(\Omega)\,.
\end{equation} 
As a particular case of \eqref{eq:poincarewirt}, we have the following Poincar\'e-trace inequality: if $E\subset \Omega$ has finite perimeter in $\Omega$, then
\begin{equation}\label{eq:poincaretrace}
\min\left\{\|{\rm Tr}_{\partial \Omega}(\chi_{E})\|_{L^{1}(\partial\Omega; \Haus{n-1})},\ \|{\rm Tr}_{\partial \Omega}(\chi_{\Omega\setminus E})\|_{L^{1}(\partial\Omega; \Haus{n-1})}\right\} \le C'\, \Per(E;\Omega)\,,
\end{equation}
where $C' = C |\Omega|^{\frac{1-n}n} \Per(\Omega)$ is scale-invariant. When $\Omega$ is a ball, we will write $C_{PT}$ in place of $C'$.
\end{theorem}
\begin{proof}
For the proof of \eqref{eq:poincarewirt} we refer to \cite[Theorem 5.11.1]{Ziemer}. The proof of \eqref{eq:poincaretrace} easily follows from \eqref{eq:poincarewirt} and the concavity of the map $t \to t^{1 - \frac{1}{n}}$, where $u=\chi_{E}$ and 
\[
G(u) = \Per(\Omega)^{-1}\int_{\partial \Omega} {\rm Tr}_{\partial \Omega}(u)\, d\Haus{n-1}\,.
\]
\end{proof}

We say that a function $u \in L^{1}_{\rm loc}(\Omega)$ has approximate limit at $x \in \Omega$ if there exists $z \in \R$ such that
\begin{equation} \label{eq:approx_lim_def}
\lim_{r \to 0} \mean{B_r(x)} |u(y) - z| \, dy = 0\,,
\end{equation}
and we denote by $\widetilde{u}(x)$ the value, $z$, of the approximate limit of $u$ at $x$. In this case, we say that $x$ is a Lebesgue point of $u$. For a given measurable set $E \subset \R^n$, we define the measure-theoretic interior of $E$ as
\[
E^1 := \left \{ x \in \R^n : \lim_{r \to 0} \frac{|E \cap B_r(x)|}{|B_r(x)|} = 1 \right \} = \left \{ x \in \R^n : \widetilde{\chi_E}(x) = 1 \right \}\,.
\] 
 
The approximate discontinuity set $S_{u}$ is the set of points where the approximate limit does not exist. We say that $x$ belongs to the set $J_u$ of approximate jump points of $u$ if there exist $a, b \in \R$, $a \neq b$, and $\nu \in \Sph^{n - 1}$ such that
\begin{equation} \label{eq:approx_jump_def}
\lim_{r \to 0} \mean{B^{+}_{r}(x, \nu)} |u(y) - a| \, dy = 0 \quad \lim_{r \to 0} \mean{B^{-}_{r}(x, \nu)} |u(y) - b| \, dy = 0,
\end{equation}
where $B^{\pm}_{r}(x, \nu) := \{ y \in B_{r}(x) : \pm ( y - x) \cdot \nu \ge 0 \}$. The triplet $(a, b, \nu)$ is uniquely determined by \eqref{eq:approx_jump_def} up to a permutation of $(a, b)$ and a change of sign of $\nu$, and we denote it by $(u^{+}(x), u^{-}(x), \nu_{u}(x))$. For the approximate traces $u^{\pm}(x)$ at $x\in J_{u}$, we adopt the convention of having $u^{+}(x) > u^{-}(x)$. Such traces can also be extended to $x \in \Omega \setminus S_{u}$ by setting $u^{+}(x) = u^{-}(x) = \widetilde{u}(x)$.

For application purposes, it is useful to introduce a representative of $u$ which is a convex combination of $u^\pm$; that is, we fix a Borel function $\lambda:\Omega\to [0,1]$ and we define the {\em $\lambda$-representative} of $u$ as
\[
u^\lambda(x) = \lambda(x) u^+(x) + (1-\lambda(x)) u^-(x) \ \text{ for } x \in \Omega \setminus (S_u \setminus J_u).
\]
It is clear that the $\lambda$-representative is generalization of the {\em precise representative}
\begin{equation*}
u^*(x) = \begin{cases} \widetilde{u}(x) & \text{ if } x \in \Omega \setminus S_u, \\
\displaystyle \frac{u^{+}(x) + u^{-}(x)}{2} & \text{ if } x \in J_u, \end{cases}
\end{equation*}
since $u^*(x) = u^\lambda(x)$ if $\lambda(x) = \frac{1}{2}$ for all $x \in J_u$. More generally, it is possible to define $u^{\lambda}$ whenever $\lambda \in \B_{b}(\Omega)$, the space of bounded Borel functions on $\Omega$.
Now, in the case $u \in BV_{\rm loc}(\Omega)$, we know that $\Haus{n - 1}(S_u \setminus J_u) = 0$, thanks to \cite[Theorem 3.78]{AFP}, and therefore $u^\lambda(x)$ is well defined for $\Haus{n - 1}$-a.e. $x \in \Omega$.

We recall the standard decomposition of the distributional gradient of a function $u \in BV(\Omega)$: we have
\begin{equation*}
Du = \nabla u \, \Leb{n} + D^s u = \nabla u \, \Leb{n} + D^c u + D^j u,
\end{equation*}
where $\nabla u \in L^1(\Omega; \R^n)$ is the density of the absolutely continuous part of $Du$, $D^s u$ is the singular part, $D^j u$ is the jump part, that is, $$D^j u = (u^+ - u^-) \nu_u \, \Haus{n-1} \res J_u,$$ and $D^c u$ is the Cantor part.

For $N > 0$, we denote by $T_{N}$ the truncation function defined as
$T_{N}(t) = \max\{-N,\min \{t,N \}\}$. 
Given $u\in BV(\Omega)$, we have $T_{N}(u) \in BV(\Omega)$ for any $N > 0$, hence the $\lambda$-representative of $T_{N}(u)$ is well-defined, and, as shown in \cite[Proposition 2.1]{ComiLeo}, we have 
\begin{equation*}
\lim_{N \to + \infty} T_{N}(u)^{\lambda}(x) = u^{\lambda}(x) \ \text{ for } \Haus{n - 1}\text{-a.e. } x \in \Omega.
\end{equation*}
The following $\lambda$-approximation result represents a fundamental tool in the proofs of the main results of Section \ref{sec:existence}. Its proof can be found in \cite[Section 3]{ComiLeo}.

\begin{theorem} [$\lambda$-approximation] \label{thm:smooth_lambda_approx}
Let $\lambda:\Omega\to [0,1]$ be a given Borel function. Then for every $u\in BV(\Omega)$ there exists a sequence $(u^\lambda_k)_{k \in \N}\subset C^\infty(\Omega) \cap BV(\Omega)$ such that we have the following:
\begin{enumerate}
\item $u^\lambda_k \to u$ in $BV(\Omega)$-strict as $k \to + \infty$,
\item $u^\lambda_k(x) \to u^\lambda(x)$ for $\Haus{n-1}$-a.e. $x \in \Omega$ as $k\to +\infty$,
\item $\displaystyle \lim_{k \to + \infty} \sqrt{1 + |Du^{\lambda}_k|^2}(\Omega) = \sqrt{1 + |Du|^2}(\Omega)$,
\item $\|u^\lambda_k\|_{L^\infty(\Omega)} \le \left( 1 + \frac{1}{k} \right) \|u\|_{L^{\infty}(\Omega)}$ for all $k \in \N$,
\item if $\Omega$ is an open set with bounded Lipschitz boundary, then $\Tr_{\partial \Omega}(u^\lambda_k)(x) = \Tr_{\partial \Omega}(u)(x)$ for $\Haus{n-1}$-a.e. $x \in \partial\Omega$ and all $k \in \N$.
\end{enumerate}
\end{theorem}

Hereafter, we recall the notions of divergence-measure field and of pairing between an essentially bounded function of bounded variation and an essentially bounded divergence-measure field.

\begin{definition} \label{DMdef}
A vector field $F \in L^{\infty}(\Omega; \R^{n})$ is called an {\em essentially bounded divergence-measure field}, and we write $F \in \DM^{\infty}(\Omega)$,
if $\div F\in \Mm(\Omega)$.
A vector field $F \in L^{\infty}_{\rm loc}(\Omega; \R^n)$ is called a {\em locally essentially bounded divergence-measure field}, and we write $F \in \DM^{\infty}_{\rm loc}(\Omega)$, if $F \in \DM^{\infty}(U)$ for any open set $U \subset \subset \Omega$.
\end{definition}

For further details and an up-to-date list of references on divergence-measure fields, we refer the reader to \cite{ComiLeo}. Following the approach of \cite{crasta2019pairings} we give the notion of $\lambda$-pairing.

\begin{definition} \label{BVpairing_lambda_def}
Given a vector field $F \in \DM^{\infty}_{\rm loc}(\Omega)$, a scalar function $u \in BV_{\rm loc}(\Omega)$ and a Borel function $\lambda : \Omega \to [0, 1]$ such that $u^{\lambda} \in L^1_{\rm loc}(\Omega; |\div F|)$, we define the $\lambda$-{\em pairing} between $F$ and $Du$ as the distribution $(F, Du)_\lambda$ given by
\begin{equation} \label{BV_pairing_lambda_def_eq} (F, Du)_\lambda := \div(u F) - u^{\lambda} \div F. \end{equation}
\end{definition}

This definition is meant to extend the classical dot product between the vector field $F$ and the gradient measure $Du$, even beyond the case $\lambda \equiv \frac 12$, which was introduced by Anzellotti in \cite{Anzellotti_1983}, and which we denote by $(F, Du)_*$.
We collect the relevant properties of the $\lambda$-pairing in the following statement, which stems from the synthesis of \cite[Proposition 4.4]{crasta2019pairings} and \cite[Theorems 2.6 and 3.4]{ComiLeo}. In particular, the approximation result is partially based on Theorem \ref{thm:smooth_lambda_approx}.

\begin{theorem}\label{lemma:pairlambda-vs-area}
Let $F\in \DM^{\infty}(\Omega)$. Let $u\in BV(\Omega)$ and $\lambda:\Omega\to [0,1]$ be a Borel function. If $u^\lambda \in L^{1}(\Omega; |\div F|)$, then we have $\div(u F), (F, Du)_{\lambda} \in \Mm(\Omega)$, with
\begin{equation} \label{eq:Leibniz_lambda_gen}
\div(uF) = u^{\lambda} \div F + (F, Du)_{\lambda} \ \text{ on } \Omega,
\end{equation}
and 
\begin{equation} \label{eq:pairing_estimate_lambda}
|(F, Du)_\lambda| \le \|F\|_{L^\infty(\Omega; \R^n)} |Du| \ \text{ on } \Omega.
\end{equation}
In addition, there exists a sequence $(u_j)_{j \in \N} \subset C^{\infty}(\Omega) \cap BV(\Omega) \cap L^\infty(\Omega)$ such that $u_j \to u$ in $BV(\Omega)$-strict and
\begin{equation} \label{eq:pairlambda_weak_conv}
(F \cdot \nabla u_j) \, \Leb{n} \weakto (F, Du)_{\lambda} \ \text{ in } \Mm(\Omega).
\end{equation}
If $u \in BV(\Omega) \cap L^\infty(\Omega)$, then $u F \in \DM^\infty(\Omega)$ and \eqref{eq:pairlambda_weak_conv} holds true for the approximating sequence $(u^\lambda_k)_{k \in \N}$ given by Theorem \ref{thm:smooth_lambda_approx}.

Finally, for all $a \ge \|F\|_{L^{\infty}(\Omega; \R^n)}$ we have
\begin{equation} \label{eq:pairlambda-vs-area}
|(F, Du)_{\lambda}| \le a\sqrt{1+|Du|^{2}} - \sqrt{a^2 -|F|^{2}}\, \Leb{n} \ \text{ on } \Omega.
\end{equation}
\end{theorem}

\medskip

Hereafter we quote a few results concerning interior and exterior normal traces of a divergence-measure vector field and their occurrence in generalized Gauss-Green and integration by parts formulas. For these results and their proofs we refer to \cite[Lemma 3.1]{comi2017locally} and \cite[Theorems 2.9 and 2.11]{ComiLeo}.

\begin{lemma} \label{lem:div_comp_supp_no_trace}
Let $F \in \DM^{\infty}(\Omega)$ be such that ${\rm supp}(F) \Subset \Omega$. Then we have $\div F(\Omega) = 0$.
\end{lemma}

We say that $\Omega$ is {\em weakly regular} if it is an open bounded set with finite perimeter in $\R^n$ such that $\Haus{n - 1}(\partial \Omega) = \Haus{n - 1}(\redb \Omega)$, or, equivalently, $\Haus{n - 1}(\partial \Omega \setminus \redb \Omega) = 0$. 

\begin{theorem} \label{thm:GG_app}
Let $F \in \DM^{\infty}(\Omega)$ and let $E \subseteq \Omega$ be of finite perimeter in $\R^n$. Assume that either $E \Subset \Omega$ or $\Omega$ is weakly regular. Then, there exist the {\em interior and exterior normal traces} of $F$ on $\redb E$; that is, the functions 
$${\rm Tr}^i(F, \redb E), {\rm Tr}^e(F, \redb E) \in L^\infty(\redb E; \Haus{n-1}),$$ 
which satisfy
\begin{equation}
 \begin{cases} \div F(E^1 \cap \Omega)  = - \displaystyle \int_{\redb E} {\rm Tr}^i(F, \redb E) \, d \Haus{n-1} \\
\|{\rm Tr}^i(F, \redb E)\|_{L^\infty(\redb E; \Haus{n-1})}  \le \|F\|_{L^\infty(E; \R^n)},  \end{cases} \label{eq:GG_1} \\
\end{equation}
and
\begin{equation*}
\begin{cases} \div F( (E^1 \cup \redb E) \cap \Omega ) = - \displaystyle \int_{\redb E} {\rm Tr}^e(F, \redb E) \, d \Haus{n-1} \\
\|{\rm Tr}^e(F, \redb E)\|_{L^\infty(\redb E; \Haus{n-1})}  \le \begin{cases} \|F\|_{L^\infty(\Omega \setminus E; \R^n)} & \text{ if } E \subset \subset \Omega \\
\|F\|_{L^\infty(\Omega; \R^n)} & \text{ otherwise} \end{cases}. \end{cases} 
\end{equation*}
In particular, if $\Omega$ is weakly regular, then there exists ${\rm Tr}^i(F, \partial \Omega) \in L^\infty(\partial \Omega; \Haus{n-1})$ satisfying 
\begin{equation}
\div F(\Omega) = - \int_{\partial \Omega} {\rm Tr}^i(F, \partial \Omega) \, d \Haus{n-1} \text{ and } \|{\rm Tr}^i(F, \partial \Omega)\|_{L^\infty(\partial \Omega; \Haus{n-1})} \le \|F\|_{L^\infty(\Omega; \R^n)}. \label{eq:GG_boundary_domain}
\end{equation}
\end{theorem}

\begin{theorem} \label{thm:GG_boundary_domain_u}
Let $\Omega$ be an open bounded set with Lipschitz boundary.
Let $F \in \DM^{\infty}(\Omega)$ and $u \in BV(\Omega)$ be such that $u^\lambda \in L^1(\Omega; |\div F|)$ for some Borel function $\lambda : \Omega \to [0, 1]$. Then we have
\begin{equation} \label{eq:GG_boundary_domain_u}
\int_{\Omega} u^\lambda \, d \div F + (F, Du)_\lambda(\Omega) = \div(uF)(\Omega) = - \int_{\partial \Omega} {\rm Tr}_{\partial \Omega}(u) {\rm Tr}^i(F, \partial \Omega) \, d \Haus{n-1}.
\end{equation}
\end{theorem}

\section{Special classes of PMC measures}\label{sec:admissible_measures}
\label{sec:admissible}
From the weak formulation of the PMCM equation, see Definition \ref{def:PMCM}, we can immediately obtain a necessary condition for the existence of a solution. The following result is a direct consequence of \cite[Lemma 4.1]{ComiLeo}, therefore we quote it without proof. 

\begin{lemma} \label{necessary_cond} Assume that there exists $T \in L^{\infty}(\Omega; \R^n)$ such that \eqref{L_infty_bound_eq} and \eqref{divergence_eq} holds true. Then $\mu$ satisfies the following conditions:
\begin{enumerate}
\item $\mu \in \MH(\Omega)$, 
\item $\max\{|\mu(E^{1})|, |\mu(E^{1} \cup \redb E)|\} \le \Per(E)$, for all sets $E \Subset \Omega$ of finite perimeter in $\R^n$,
\item if in addition $\Omega$ is weakly regular, then $\max\{|\mu(E^{1} \cap \Omega)|, |\mu((E^{1} \cup \redb E) \cap \Omega)|\} \le \Per(E)$ for all sets $E\subset \Omega$ of finite perimeter in $\R^n$.
\end{enumerate}
\end{lemma}
In view of the variational approach to the PMCM equation, and with reference to Lemma \ref{necessary_cond}, we shall introduce some definitions and facts concerning measures in $\Mm(\Omega)$.

\begin{definition}\label{def:muLper}
A measure $\mu\in \Mm(\Omega)$ is said to satisfy the \emph{necessary condition for the PMCM equation} if 
\begin{equation} \label{eq:necess_cond}
|\mu(E^{1} \cap \Omega)| \le \Per(E) \quad \text{ for all measurable sets } \, E\subset \Omega.
\end{equation}
Then, $\mu$ is \emph{non-extremal} if there exists $0<L<1$ such that 
\begin{equation} \label{eq:subcritical_cond}
|\mu(E^{1} \cap \Omega)| \le L\, \Per(E) \quad \text{ for all measurable sets } \, E\subset \Omega.
\end{equation}
\end{definition}

These conditions can be seen as particular cases of the perimeter bound condition introduced in \cite[Section 4]{ComiLeo}.

\begin{definition}\label{def:muLper}
Given $\mu \in \Mm(\Omega)$ and $L > 0$, we say that $\mu$ belongs to $\PB_{L}(\Omega)$ if 
\begin{equation} \label{eq:PBL}
|\mu(E^{1} \cap \Omega)| \le L\, \Per(E)\,,\quad \text{ for all measurable sets } \, E\subset \Omega.
\end{equation}
We also set $\PB(\Omega) := \bigcup_{L > 0} \PB_{L}(\Omega)$ and, if $\mu \in \PB(\Omega)$, we say that $\mu$ satisfies a perimeter bound condition.
\end{definition}

In particular, $\mu$ is non-extremal if and only if $\mu \in \PB_{L}(\Omega)$ for some $L \in (0, 1)$, and therefore we directly deduce from \cite{ComiLeo} some relevant results concerning non-extremal measures.

\begin{remark} \label{necessary_cond_ball} 
If $\mu\in \Mm(\Omega)$ is non-extremal, then $\mu \in \MH(\Omega)$, by \cite[Remark 4.3]{ComiLeo}.
\end{remark}

As pointed out in \cite{ComiLeo}, measures in $\PB(\Omega)$ belong to the dual of the space $BV(\Omega)$. Hence, we recall here some relevant definitions and results concerning measures in $BV(\Omega)^*$, and for a detailed exposition, we refer the reader to \cite[Section 2]{ComiLeo}.


Given $\mu \in \MH(\Omega)$, the functional 
\begin{equation*} 
\T_{\mu}(u) = \int_\Omega u^* \, d \mu \quad \text{ for all } u \in BV(\Omega) \cap L^{\infty}(\Omega)
\end{equation*}
is well-defined and linear, although not necessarily continuous. Then, we introduce the following terminology: 

\begin{minipage}{1\linewidth}
\begin{center} \textit{given $\mu \in \MH(\Omega)$, we say that $\mu \in BV(\Omega)^*$ if there exists $\T \in BV(\Omega)^*$ such that} \end{center}
$$\T(u) = \T_{\mu}(u) \quad \textit{ for all } \ u\in BV(\Omega) \cap L^\infty(\Omega).$$
\end{minipage}

Under the additional assumption that $\Omega$ is an open bounded set with Lipschitz boundary, it was proved in \cite[Theorem 8.2]{Phuc_Torres} (see also \cite[Lemma 4.4 and Corollary 4.11]{ComiLeo}) that a necessary and sufficient condition for the continuity of $\T_{\mu}$ on $BV(\Omega)\cap L^{\infty}(\Omega)$ with respect to the $BV$-norm is that $\mu \in \PB(\Omega)$.
To study the variational problem related to the PMCM equation, we also introduce the action of the measure $\mu$ against the $\lambda$-representative of $u$. Given a pair $(\mu, \lambda) \in \MH(\Omega) \times \B_{b}(\Omega)$, we define 
\begin{equation} \label{def:T_mu_lambda}
\T_{\mu, \lambda}(u) = \int_\Omega u^\lambda \, d \mu \quad \text{ for all } u \in BV(\Omega) \cap L^{\infty}(\Omega).
\end{equation}
It is worth remarking that the functional $\T_{\mu, \lambda}$ can be extended to a continuous functional defined on the whole space $BV(\Omega)$, as soon as $\mu$ is \textit{admissible} in the sense of the following definition.

\begin{definition}\label{def:muadmissible}
We say that $\mu\in \MH(\Omega)$ is \emph{admissible} if $|\mu| \in BV(\Omega)^{*}$.
\end{definition}

To characterize the admissible measures, we state some facts concerning measures in the dual of $BV$ proved in \cite{ComiLeo}.

\begin{lemma}{\rm\cite[Lemma 4.7]{ComiLeo}} \label{lem:2a2b-bis} 
If $\mu \in \MH(\Omega)$ is admissible, then, for any $u\in BV(\Omega)$ and any Borel function $\lambda : \Omega \to [0, 1]$, we have $u^\lambda \in L^{1}(\Omega; |\mu|)$. In particular, there exists a constant $C>0$ such that 
\begin{equation} \label{eq:admissibility_def}
\int_\Omega |u^{\lambda}| \, d|\mu| \le C\, \|u\|_{BV(\Omega)}  \quad \text{ for all } u\in BV(\Omega).
\end{equation}
\end{lemma}

\begin{remark} \label{rem:div_admissible_Leibniz}
Thanks to Lemma \ref{lem:2a2b-bis}, if the measure $\div F$ is admissible in Theorems \ref{lemma:pairlambda-vs-area} and \ref{thm:GG_boundary_domain_u}, then we do not need to require $u^\lambda \in L^{1}(\Omega; |\div F|)$.
\end{remark}

\begin{remark} \label{rem:extremal_vs_admissible}
We notice that, if $\mu \in \MH(\Omega)$ is admissible, then we have also $\mu \in BV(\Omega)^{*}$. Indeed, due to Lemma \ref{lem:2a2b-bis}, we know that $u^* \in L^1(\Omega, |\mu|)$ for all $u \in BV(\Omega)$, and, by exploiting \eqref{eq:admissibility_def} with $\lambda \equiv \frac{1}{2}$, we obtain
\begin{equation*}
\left | \int_{\Omega} u^{*} \, d \mu \right | \le \int_{\Omega} |u^{*}| \, d |\mu| \le C \big\| u \big\|_{BV(\Omega)}.
\end{equation*}
We underline that, in general, the converse is not true, as shown in \cite[Proposition 5.1]{Phuc_Torres} and \cite[Remark 4.12]{ComiLeo}. 
\end{remark}

We collect in the following proposition the relations between the divergence-measure fields, the perimeter bound, and the admissibility condition.

\begin{proposition}{\rm\cite[Proposition 4.10]{ComiLeo}} \label{prop:adm_PB_div}
The following holds true:
\begin{enumerate}
\item If $\Omega$ is weakly regular and $F \in \DM^{\infty}(\Omega)$, then $\div F \in \PB_L(\Omega)$ for $L = \|F\|_{L^\infty(\Omega; \R^n)}$; if in addition $\Omega$ has Lipschitz boundary, then $\div F \in BV(\Omega)^*$.

\item If $\Omega$ is a bounded open set with Lipschitz boundary and $\mu \in \PB(\Omega)$, then there exists $F \in \DM^{\infty}(\Omega)$ such that $\div F = \mu$ on $\Omega$ and we have $\mu \in BV(\Omega)^*$.

\item If $n \ge 2$ and $|\Omega| < + \infty$ and $\mu \in \MH(\Omega)$ is an admissible measure, then $\mu \in \PB(\Omega)$.

\item If $\mu \in \MH(\Omega)$ is an admissible measure, then there exists $F, \overline{F} \in \DM^{\infty}(\Omega)$ such that $\mu = \div F$ and $|\mu| = \div \overline{F}$ on $\Omega$.

\item If $n = 1$ and $\mu \in \Mm(\Omega)$, then $\mu$ is admissible and there exists $f, \overline{f} \in BV(\Omega)$ such that $\mu = D f$ and $|\mu| = D \overline{f}$ on $\Omega$; if in addition $|\R \setminus \Omega| > 0$, then $\Mm(\Omega) = \PB(\Omega)$.
\end{enumerate}
\end{proposition}

\begin{remark}
Points (1) and (2) of Proposition \ref{prop:adm_PB_div} can be rephrased in terms of the non-extremality condition. On the one hand, if $\mu = \div F$ for some $F \in \DM^{\infty}(\Omega)$ with $\|F\|_{L^\infty(\Omega; \R^n)} < 1$, then $\mu$ is non-extremal, by  Proposition \ref{prop:adm_PB_div}(1). On the other hand, if $\Omega$ is a bounded open set with Lipschitz boundary and $\mu \in \Mm(\Omega)$ is non-extremal, then $\mu \in BV(\Omega)^*$ and there exists $F \in \DM^{\infty}(\Omega)$ such that $\div F = \mu$ on $\Omega$, thanks to Proposition \ref{prop:adm_PB_div}(2). It is also relevant to point out that non-extremal measures may be not admissible. Indeed, up to a constant factor, \cite[Remark 4.12]{ComiLeo} gives us an example of a non-extremal measure $\mu$ such that $|\mu| \notin BV^*(\Omega)$.
\end{remark}

In the following lemma, we give a sufficient condition under which non-extremality implies admissibility.
\begin{lemma}{\rm\cite[Lemma 4.16]{ComiLeo}} \label{lem:Schmidt_non_extremality}
Let $\Omega$ be an open bounded set with Lipschitz boundary and let $\mu \in \MH(\Omega)$. Assume that its positive and negative parts, $\mu^+$ and $\mu^-$, are non-extremal. Then $\mu$ is non-extremal and admissible. In particular, if $\mu \ge 0$ is non-extremal, then it is admissible. 
\end{lemma}

An immediate consequence of Lemma \ref{lem:2a2b-bis} is that, for an admissible measure $\mu$, the functional $\T_{\mu, \lambda}$ defined in \eqref{def:T_mu_lambda} can be naturally extended to all functions in $BV(\Omega)$.

\begin{lemma}{\rm\cite[Lemma 4.17]{ComiLeo}} \label{lem:lambda_int_approx}
Let $\mu \in \MH(\Omega)$ be admissible and $\lambda: \Omega \to [0, 1]$ be a Borel function. Then the functional $\T_{\mu, \lambda} : BV(\Omega) \to \R$ given by \eqref{def:T_mu_lambda}
is well-defined.
Then, given $u \in BV(\Omega)$, if we consider the truncation of $u$, $T_N(u)$ for $N > 0$, and its smooth approximation $T_{N}(u)^{\lambda}_{k}$ given by Theorem \ref{thm:smooth_lambda_approx}, we have
\begin{equation} \label{eq:conv_trunc_lambda}
\lim_{N \to + \infty} \int_{\Omega} T_N(u)^{\lambda} \, d\mu = \int_{\Omega} u^{\lambda} \, d \mu
\end{equation}
and
\begin{equation} \label{eq:conv_trunc_smooth_approx}
\lim_{N \to + \infty} \lim_{k \to + \infty} \int_{\Omega} T_N(u)_k^{\lambda} \, d\mu = \int_{\Omega} u^\lambda \, d \mu.
\end{equation}
Finally, for any $u \in BV(\Omega)$ we can find a sequence $(u_j)_{j \in \N} \subset C^\infty(\Omega) \cap BV(\Omega) \cap L^\infty(\Omega)$ such that
\begin{equation} \label{eq:density_conv_lambda}
\lim_{j \to + \infty} \T_{\mu, \lambda}(u_j) = \T_{\mu, \lambda}(u).
\end{equation}
\end{lemma}

Then, one obtains the following sharp estimate of $|\T_{\mu, \lambda}(u)|$ in terms of the total variation of $u$ and its trace on the boundary of $\Omega$.

\begin{proposition}{\rm\cite[Proposition 4.19]{ComiLeo}}\label{prop:trunc-smooth-coarea}
Let $\Omega$ be an open bounded set with Lipschitz boundary. Let $L > 0$ and $\mu \in \PB_L(\Omega)$. Let $u \in BV(\Omega)$ and $\lambda:\Omega\to [0,1]$ be a Borel function. If $u^\lambda \in L^{1}(\Omega; |\mu|)$, then
\begin{equation} \label{eq:mu_u_Du_condition}
\left | \int_{\Omega} u^{\lambda} \, d \mu \right | \le L \left(|Du|(\Omega) + \int_{\de \Omega} |{\rm Tr}_{\partial \Omega}(u)|\, d\Haus{n-1}\right).
\end{equation}
In particular, \eqref{eq:mu_u_Du_condition} holds for all $u \in BV(\Omega)$ as long as $\mu$ is admissible.
\end{proposition}


As a consequence, we see that, if a measure is in $\PB_L(\Omega)$, then it can be represented as the divergence of a field satisfying $|T(x)| \le L$ for $\Leb{n}$-a.e. $x \in \Omega$. In particular, a non-extremal measure can be written as the divergence of a sub-unitary divergence-measure field.

\begin{lemma}{\rm\cite[Lemma 4.21]{ComiLeo}} \label{lem:existence_optimal_T}
Let $\Omega$ be an open bounded set with Lipschitz boundary. Let $L > 0$ and $\mu \in \PB_L(\Omega)$ be an admissible measure. Then there exists $T \in \DM^{\infty}(\Omega)$ such that $\div T = \mu$ on $\Omega$ and $\|T\|_{L^{\infty}(\Omega; \R^n)} \le L$.
\end{lemma}

We end this section by introducing a suitable subclass of admissible measures, which we denote as $\tMm(\Omega)$ and that will be considered for the semicontinuity result of the following section.

\begin{definition}\label{def:tMomega}
Given a measure $\mu \in \MH(\Omega)$, we say that $\mu\in \tMm(\Omega)$ if
\[
\mu = \mu_1 + \mu_2 + \mu_3,
\] 
where $\mu_1, \mu_2, \mu_3$ are admissible measures in $\MH(\Omega)$ such that
\begin{itemize}
\item $\mu_1 = h \Leb{n}$ for some $h \in L^{q}(\Omega)$ and $q>n$,
\item ${\rm supp}(\mu_2) \Subset \Omega$ and $\mu_2 = \div T_0$ for some $T_0 \in \DM^{\infty}(\Omega) \cap C(\Omega \setminus S; \R^n)$, where $S \subset \Omega$ is a compact set such that $\Haus{n-1}(S) = 0$,
\item $|\mu_3|(\Omega \setminus \Gamma) = 0$, for some compact set $\Gamma\subset \Omega$ such that $|\Gamma| = 0$ and ${\rm dist}(S, \Gamma) > 0$.
\end{itemize}
\end{definition}

The following relevant examples of measures in $\tMm(\Omega)$ are provided by elaborating on \cite[Examples 4.14 and 4.15]{ComiLeo}, respectively.

\begin{example} \label{ex:tMomega} \rm
Let $\Omega$ be a bounded open set with Lipschitz boundary. 
Let $\mu\in \Mm(\Omega)$ such that
\[
\mu = h\, \Leb{n} + \gamma\, \Haus{n-1}\restrict \Gamma
\] 
where $h\in L^{q}(\Omega)$ for some $q>n$, $\gamma \in L^{\infty}(\Omega; \Haus{n-1}\restrict \Gamma)$, 
and $\Gamma\subset \Omega$ is a compact set such that there exist $\Lambda, \overline{\rho}>0$ for which
\begin{equation}\label{eq:mudensityestimate}
\Haus{n-1}(\Gamma \cap B_{\rho}(x)) \le \Lambda \rho^{n-1}\qquad \forall\, 0<\rho<\overline{\rho} \text{ and }\forall\, x\in \Gamma\,.
\end{equation}
Then, $\mu_1 = h \Leb{n}$ and $\mu_3 = \gamma\, \Haus{n-1}\restrict \Gamma$ are admissible measures, and therefore $\mu \in \tMm(\Omega)$. For the proof of this assertion, see \cite[Example 4.14]{ComiLeo}.
\end{example}

We notice that the restriction on the summability of the density of $\mu_1$ in Definition \ref{def:tMomega} can be overcome by carefully choosing the measure $\mu_2$, as shown in the following example.

\begin{example} \rm \label{ex:Virginia}
Let $n \ge 2$ and $\mu = h \Leb{n}$, where $$h(x) = \frac{(n-1)}{|x-x_0|} \ \text{ for } x \neq x_0,$$ for some $x_0 \in \R^n$. Clearly, $h \in L^q_{\rm loc}(\R^n)$ only for $q \in [1, n)$, so that this example does not fall in the same category of Example \ref{ex:tMomega}. In addition, as noticed in \cite[Example 4.15]{ComiLeo}, $\mu$ is an admissible measure, since $\mu \ge 0$ and $\mu = \div T$ for $$T(x) = \frac{(x-x_0)}{|x-x_0|} \ \text{ for } x \neq x_0.$$ It is plain to see that $T \in \DM^\infty_{\rm loc}(\R^n) \cap C(\R^n \setminus \{x_0\}; \R^n)$, but this is not enough to ensure that $\mu \in \tMm(\Omega)$ for any open bounded set $\Omega$ with Lipschitz boundary and $x_{0}\in \Omega$, given that $\mu$ does not have compact support. To prove that $\mu \in \tMm(\Omega)$, we decompose $\mu$ as a sum $\mu_1 + \mu_2$ in the following way. Given $\delta > 0$ such that $B_{2 \delta}(x_0) \Subset \Omega$, we set
\begin{equation*}
\mu = \mu_1 + \mu_2, \ \text{ with } \mu_1 = \left ( h (1 - \eta_\delta)  - T \cdot \nabla \eta_\delta \right ) \Leb{n} \ \text{ and } \ \mu_2 = \div(\eta_\delta T),
\end{equation*}
where 
\begin{equation*}
\eta_\delta(x) = \begin{cases} 1 & \text{ if } |x - x_0| \le \delta, \\
2 - \frac{|x - x_0|}{\delta} & \text{ if } \delta < |x - x_0| < 2 \delta, \\
0 & \text{ if } |x - x_0| \ge 2 \delta. \end{cases}
\end{equation*}
It is clear that $\eta_\delta \in \Lip_c(\R^n)$, ${\rm supp}(\eta_\delta) \subseteq \overline{B_{2 \delta}(x_0)}$, $0 \le \eta_\delta \le 1$ and $\eta_\delta \equiv 1$ on $B_\delta(x_0)$. Hence, $\left ( h (1 - \eta_\delta)  - T \cdot \nabla \eta_\delta \right ) \in L^\infty(\Omega)$, so that, arguing as in Example \ref{ex:tMomega}, we see that $\mu_1$ is an admissible measure. As for $\mu_2$, it is easy to check that $S = \{ x_0 \}$, and
$$\div(\eta_\delta T) = ( h \eta_\delta + T  \cdot \nabla \eta_\delta) \Leb{n},$$
so that $\div(\eta_\delta T) = h \Leb{n} = \mu$ on $B_\delta(x_0)$ and $( h \eta_\delta + T  \cdot \nabla \eta_\delta) \in L^\infty(\Omega \setminus B_\delta(x_0))$. Since $\mu \in BV(\Omega)^*$ and $\mu \ge 0$, for all $u \in BV(\Omega)$ we get
\begin{align*}
\left | \int_{\Omega} u \, d |\div(\eta_\delta T)| \right | & \le \left | \int_{\Omega} u \chi_{B_\delta(x_0)} \, d \mu \right | + \left | \int_{\Omega} u \chi_{\Omega \setminus B_\delta(x_0)} \left | h \eta_\delta + T  \cdot \nabla \eta_\delta \right | \, dx \right | \\
& \le C \|u \chi_{B_\delta(x_0)}\|_{BV(\Omega)} + \|h \eta_\delta + T  \cdot \nabla \eta_\delta\|_{L^\infty(\Omega \setminus B_\delta(x_0))} \|u \chi_{\Omega \setminus B_\delta(x_0)}\|_{L^1(\Omega)} \\
& \le C_1 \|u\|_{L^1(\Omega)} + C |Du|(\overline{B_\delta(x_0)}) + C \int_{\partial B_\delta(x_0)} |{\rm Tr}_{\partial B_\delta(x_0)}(u)| \, d \Haus{n-1} \\
& \le C_2 \|u\|_{BV(\Omega)}, 
\end{align*} 
for some constants $C, C_1, C_2 > 0$, thanks to the Leibniz rule for $BV$ functions and the continuity of the trace operator.
\end{example}

\begin{remark}
We stress the fact that, even if $\Omega$ is a bounded open set with Lipschitz boundary, it is necessary to require $\mu_2$ to be admissible in Definition \ref{def:tMomega}. Indeed, the existence of a vector field $T_0 \in \DM^{\infty}(\Omega) \cap C(\Omega \setminus S; \R^n)$, for some compact set $S \subset \Omega$ with $\Haus{n-1}(S) = 0$, satisfying $\mu_2 = \div T_0$ implies only that $\mu_2 \in BV(\Omega)^*$ (see Remark \ref{rem:extremal_vs_admissible}), while a priori does not give any information on $|\mu_2|$. Actually, the example given in \cite[Proposition 5.1]{Phuc_Torres} of a measure in $BV(\R^n)^*$ whose total variation does not belong to $BV(\R^n)^*$ is exactly the divergence-measure of a vector field $T \in \DM^\infty(\R^n) \cap C(\R^n \setminus \{0\}; \R^n)$.

However, we can drop the admissibility assumption on $\mu_2$ if we ask it to be nonnegative. Therefore, we recover the assumptions on the measure $\mu$ imposed in \cite{ziemer1995nonhomogeneous} as a particular case of the definition of $\mu_2$: indeed, if $\mu \ge 0$ and there exists $C > 0$ and $\eps \in (0, 1]$ such that $\mu(B_r(x)) \le C r^{n-1+\eps}$ for all $B_r(x) \subset \Omega$, then \cite[Theorem 4.5]{phuc2008characterizations} implies that there exists $T \in \DM^{\infty}(\Omega) \cap C(\Omega; \R^n)$ such that $\div T = \mu$.
\end{remark}

\section{Minimizing the capillarity-type functional $\Jj_{\mu}$} 
\label{sec:existence}
Taking into account Proposition \ref{prop:trunc-smooth-coarea}, we shall assume from this section onward that $\Omega$ is a bounded open set with Lipschitz boundary\footnote{In the one-dimensional case, this translates into assuming $\Omega$ to be a finite union of open bounded intervals, see Section \ref{sec:n_1}.}, and that $\mu$ is admissible in the sense of Definition \ref{def:muadmissible}, and satisfies the non-extremality condition \eqref{eq:subcritical_cond}, unless otherwise stated.  

We fix a ball $B$ such that $\Omega \Subset B$, choose a function $\phi\in W^{1,1}(B)$, and set
\[
\cS(\phi) = \{u\in BV(B):\ u=\phi \text{ on }B\setminus \Omega\}\,.
\]
Then we define the functional $\Jj_{\mu}:BV(B) \times \B_{b}(\Omega) \to [-\infty,+\infty]$ as
\begin{equation}\label{eq:Jmu}
\Jj_{\mu}[u,\lambda] = \begin{cases}
\sqrt{1+|Du|^2}(B) + \ds\int_\Omega u^{\lambda}\, d\mu & \text{if }u\in \cS(\phi)\,,\\
+\infty & \text{otherwise.}
\end{cases}
\end{equation}
We notice that $\Jj_{\mu}$ is not identically $+\infty$. To see this, observe that the function $z_{\phi}\in BV(B)$ defined by $z_{\phi}=\phi$ on $B\setminus \Omega$ and $z_{\phi}=0$ on $\Omega$ satisfies $z_{\phi}\in \cS(\phi)$ and
\[
\Jj_{\mu}[z_{\phi}, \lambda] = \sqrt{1+|D\phi|^{2}}(B\setminus \overline{\Omega}) + |\Omega| + \int_{\de\Omega} |{\rm Tr}_{\partial \Omega}(\phi)|\, d\Haus{n-1} < +\infty
\]
for all $\lambda \in \B_{b}(\Omega)$.

The first relevant result concerning this functional is the density-in-energy of smooth functions, stated in the next lemma.

\begin{lemma} \label{lem:density_energy}
Let $\mu \in \MH(\Omega)$ be an admissible measure, $\lambda : \Omega \to [0,1]$ be a Borel function and $u \in \cS(\phi)$. Then there exists a sequence $(u_j)_{j \in \N} \subset \cS(\phi) \cap C^{\infty}(\Omega)$ such that
\begin{equation} \label{eq:density_energy}
\lim_{j \to + \infty}  \Jj_{\mu}[u_{j}, \lambda] = \Jj_{\mu}[u, \lambda].
\end{equation}
\end{lemma}

\begin{proof}
Arguing as in the proof of \cite[Lemma 4.17]{ComiLeo}, we exploit Theorem \ref{thm:smooth_lambda_approx} to construct on $\Omega$ the sequence $(T_{N}(u)^\lambda_k)_{k, N \in \N} \subset BV(\Omega) \cap C^{\infty}(\Omega) \cap L^\infty(\Omega)$. 
Then we define 
\[
u_{k, N} = T_{N}(u)^\lambda_k \, \chi_\Omega  + \phi \, \chi_{B \setminus \Omega},
\]
and we point out that $u_{k, N} \in \cS(\phi) \cap C^{\infty}(\Omega)$.
Since $u_{k, N} = T_{N}(u)^\lambda_k$ on $\Omega$, then \eqref{eq:conv_trunc_smooth_approx} implies that
\begin{equation*}
\lim_{N \to + \infty} \lim_{k \to + \infty} \int_{\Omega} u_{k,N} \, d\mu = \int_{\Omega} u^\lambda \, d\mu.
\end{equation*}
Then, we notice that
\begin{align*}
\sqrt{1 + |D u_{k,N}|^2} & = \sqrt{1 + |\nabla T_{N}(u)^\lambda_k |^2} \Leb{n} \restrict \Omega + |\Tr_{\partial \Omega} (T_N(u)^\lambda_k) - \Tr_{\partial \Omega} (\phi)| \, \Haus{n-1} \restrict \partial \Omega \\
& \quad + \sqrt{1+ |\nabla \phi|^2} \Leb{n} \restrict B \setminus \overline{\Omega}
\end{align*}
Theorem \ref{thm:smooth_lambda_approx} then implies that
\[
\lim_{k \to + \infty} \int_{\Omega} \sqrt{1+|\nabla  T_{N}(u)^\lambda_k|^{2}}\, dx = \sqrt{1+|DT_N(u)|^{2}}(\Omega)
\]
and 
\begin{equation*}
|\Tr_{\partial \Omega} (T_N(u)^\lambda_k) - \Tr_{\partial \Omega} (\phi)| \, \Haus{n-1} \restrict \partial \Omega = |\Tr_{\partial \Omega} (T_N(u)) - \Tr_{\partial \Omega} (\phi)| \, \Haus{n-1} \restrict \partial \Omega.
\end{equation*}
Finally, the strict convergence of $T_N(u)$ to $u$ as $N\to + \infty$ allows us to conclude that
\begin{align*}
\lim_{N \to + \infty} \lim_{k \to + \infty} \sqrt{1 + |D u_{k,N}|^2}(B) & = \sqrt{1+|D u|^{2}}(\Omega) + \int_{\partial \Omega} |\Tr_{\partial \Omega} (u) - \Tr_{\partial \Omega} (\phi)| \, d \Haus{n-1} \\
& \quad + \int_{B \setminus \Omega} \sqrt{1+ |\nabla \phi|^2} \, dx = \sqrt{1+|D u|^{2}}(B).
\end{align*}
Thus, there exists a sequence $(k_j)_{j \in \N} \subset \N$ such that $u_j := T_{j}(u)_{k_j}^\lambda$ belongs to $\cS(\phi) \cap C^\infty(\Omega)$ and satisfies \eqref{eq:density_energy}.
\end{proof}

Since we are interested in minimizing $\Jj_{\mu}[u,\lambda]$, we start by doing such an operation with respect to $\lambda$. We have the following, elementary lower bound.

\begin{lemma} \label{lem:min_J_lambda}
Let $\mu \in \MH(\Omega)$ be an admissible measure and $u \in BV(\Omega)$. Let $\{\Omega_+,\Omega_-\}$ be a Borel partition of $\Omega$ associated with the Hahn decomposition of $\mu$, i.e., such that the positive and negative parts of $\mu$ satisfy $\mu_{+} = \mu \restrict \Omega_+$ and $\mu_{-} = -\mu \restrict \Omega_-$. We set $\lambda_{\mu} = \chi_{\Omega_-}$, for which we have
\begin{equation} \label{eq:lambda_mu_representative}
u^{\lambda_{\mu}} = u^- \chi_{\Omega_+} + u^+ \chi_{\Omega_-}\,.
\end{equation} 
Then, for all Borel functions $\lambda : \Omega \to [0,1]$ we have
\begin{equation*}
\int_{\Omega} u^{\lambda_\mu} \, d \mu\ \le\ \int_{\Omega} u^{\lambda} \, d\mu\,,
\end{equation*}
with equality if and only if $\lambda(x) = \lambda_\mu(x)$ for $|\mu|$-almost every $x\in J_{u}$.
\end{lemma}

\begin{proof}
Thanks to \eqref{eq:lambda_mu_representative}, we have
\begin{align*}
\int_{\Omega} (u^\lambda - u^{\lambda_\mu}) \, d \mu & = \int_{\Omega} \left ( u^+ (\lambda - \chi_{\Omega_{-}}) + u^{-} (1 - \lambda - \chi_{\Omega_{+}} ) \right ) \, d\mu \\
& = \int_{\Omega_+} (u^+ - u^-) \lambda \, d\mu + \int_{\Omega_{-}} (u^+ - u^-) ( \lambda - 1 ) \, d\mu \\
& = \int_{\Omega} (u^+ - u^-) \lambda \chi_{\Omega_{+}} \, d\mu_{+} + \int_{\Omega} (u^+ - u^-) (1 - \lambda ) \chi_{\Omega_{-}} \, d\mu_{-} \\
& = \int_{\Omega} (u^+ - u^-) (\lambda \chi_{\Omega_{+}} + (1 - \lambda) \chi_{\Omega_{-}}) \, d |\mu| \ge 0,
\end{align*}
with equality if and only if
\begin{equation*}
\lambda \chi_{\Omega_{+}} + (1 - \lambda) \chi_{\Omega_{-}} = 0\qquad \text{$|\mu|$-a.e. on }J_{u},
\end{equation*}
which implies $\lambda(x) = \chi_{\Omega_{-}}(x) = \lambda_\mu(x)$ for $|\mu|$-a.e. $x\in J_{u}$.
\end{proof}

The next proposition, which easily follows from Lemmas \ref{lem:density_energy} and \ref{lem:min_J_lambda}, contains an important observation related to the lower-semicontinuous relaxation of the functional $\Jj_{\mu}[u, \lambda]$. 

\begin{proposition} \label{prop:relax_lambda_mu}
Let $\mu \in \MH(\Omega)$ be an admissible measure and $\lambda:\Omega\to [0,1]$ be a Borel function. For any $u \in \cS(\phi)$ there exists a sequence $(u_{j})_{j\in \N} \subset \cS(\phi)\cap C^{\infty}(\Omega)$ such that $u_{j}\to u$ in $L^{1}(\Omega)$ and 
\begin{equation} \label{eq:equality_inf_1}
\lim_{j\to+\infty} \Jj_{\mu}[u_{j}, \lambda] = \Jj_{\mu}[u, \lambda_{\mu}] \le \Jj_{\mu}[u, \lambda],
\end{equation}
with equality if and only if $\lambda(x) = \lambda_{\mu}(x)$ for $|\mu|$-almost every $x\in J_{u}$. Moreover, we have
\begin{equation} \label{eq:equality_inf}
\inf_{u\in \cS(\phi)} \Jj_{\mu}[u, \lambda] = \inf_{u\in \cS(\phi)} \Jj_{\mu}[u, \lambda_{\mu}].
\end{equation}
In particular, $v \in \cS(\phi)$ satisfies $\Jj_{\mu}[v, \lambda] = \inf_{u\in \cS(\phi)} \Jj_{\mu}[u, \lambda]$ if and only if it satisfies $\Jj_{\mu}[v, \lambda_{\mu}] = \inf_{u\in \cS(\phi)} \Jj_{\mu}[u, \lambda_{\mu}]$ and we have $\lambda(x) = \lambda_{\mu}(x)$ for $|\mu|$-almost every $x\in J_{v}$.
\end{proposition}

\begin{proof}
It is enough to apply Lemma \ref{lem:density_energy} with $\lambda_{\mu}$ and observe that the resulting sequence $(u_{j})_{j\in \N}$ verifies the obvious property
\[
\Jj_{\mu}[u_{j}, \lambda] = \Jj_{\mu}[u_{j}, \lambda_{\mu}]\,,
\] 
since $u_{j}$ is continuous for every $j\in \N$ by construction. Then Lemma \ref{lem:min_J_lambda} allows to obtain \eqref{eq:equality_inf_1}. In turns, \eqref{eq:equality_inf} follows from \eqref{eq:equality_inf_1}, since by Lemma \ref{lem:density_energy} we can always choose a smooth minimizing sequence, on which the functionals coincide. Finally, if $v \in \cS(\phi)$ is a minimizer of $\Jj_{\mu}[\cdot, \lambda]$, then it is also a minimizer of $\Jj_{\mu}[\cdot, \lambda_\mu]$ by \eqref{eq:equality_inf}, and we get
$$ \inf_{u\in \cS(\phi)} \Jj_{\mu}[u, \lambda] = \Jj_{\mu}[v, \lambda] \ge \Jj_{\mu}[v, \lambda_{\mu}] = \inf_{u\in \cS(\phi)} \Jj_{\mu}[u, \lambda_{\mu}],$$
thanks to \eqref{eq:equality_inf_1}. Therefore, we obtain $\Jj_{\mu}[v, \lambda] = \Jj_{\mu}[v, \lambda_{\mu}]$, and, by Lemma \ref{lem:min_J_lambda}, this can hold true if and only if $\lambda(x) = \lambda_{\mu}(x)$ for $|\mu|$-almost every $x\in J_{v}$. The converse implication is trivial, since, if $\lambda(x) = \lambda_{\mu}(x)$ for $|\mu|$-almost every $x\in J_{v}$, then $\Jj_{\mu}[v, \lambda] = \Jj_{\mu}[v, \lambda_{\mu}]$ and the functionals admits the same infimum by \eqref{eq:equality_inf}.
\end{proof}

In the next subsections, we shall discuss the lower semicontinuity and the coercivity properties of the functional $\Jj_{\mu}$. In particular, we consider the functional $\Jj_{\mu}[u] = \Jj_{\mu}[u, \lambda_\mu]$ for the lower semicontinuity, since this choice of $\lambda$ is essential in the argument of the proof, and also in the light of Proposition \ref{prop:relax_lambda_mu}. We shall denote this functional simply by $\Jj_{\mu}[u]$ when no ambiguity occurs.

\subsection{Coercivity of $\Jj_{\mu}$}\label{sec:coercivity}
Let us start by proving that $\Jj_{\mu}[\cdot, \lambda]$ is coercive on $\cS(\phi)$ for any Borel function $\lambda : \Omega \to [0, 1]$. Since $\mu$ is admissible and non-extremal, i.e., it satisfies \eqref{eq:subcritical_cond}, by Proposition \ref{prop:trunc-smooth-coarea} we obtain
\[
\left|\int_\Omega u^{\lambda}\, d\mu\right| \le L\, \left(|Du|(\Omega) + \int_{\de \Omega} |{\rm Tr}_{\partial \Omega}(u)|\, d\Haus{n-1}\right)\,,
\]
hence for all $u \in \cS(\phi)$ we get
\begin{align}\nonumber
\Jj_{\mu}[u, \lambda] &\ge |Du|(B) - L\, \left(|Du|(\Omega) + \int_{\de \Omega} |{\rm Tr}_{\partial \Omega}(u)|\, d\Haus{n-1}\right)\\\nonumber
&\ge |Du|(B\setminus \overline{\Omega}) + (1-L)|Du|(\Omega) -L\int_{\de \Omega} |{\rm Tr}_{\partial \Omega}(u)|\, d\Haus{n-1} + \int_{\de \Omega}|{\rm Tr}_{\partial \Omega}(u - \phi)|\, d\Haus{n-1}\\\nonumber
&\ge (1-L)\left(|Du|(\Omega) + \int_{\de \Omega} |{\rm Tr}_{\partial \Omega}(u)|\, d\Haus{n-1}\right) - \int_{\de \Omega}|{\rm Tr}_{\partial \Omega}(\phi)|\, d\Haus{n-1}\\\label{eq:stimaJmu}
&= (1-L)|Du_{0}|(B) - \int_{\de \Omega}|{\rm Tr}_{\partial \Omega}(\phi)|\, d\Haus{n-1}\,,
\end{align}
where 
\[
u_{0}(x) = \begin{cases}
u(x) & \text{if }x\in \Omega\,,\\
0 & \text{if }x\in B\setminus \Omega\,.
\end{cases}
\]
In particular, this result holds for $\lambda = \lambda_\mu$.

\subsection{Lower semicontinuity of $\Jj_{\mu}$}\label{subsec:sci}

In this subsection, we set $\lambda = \lambda_\mu$ and prove the following result.

\begin{theorem}[Semicontinuity]\label{thm:sci}
Let $\mu\in \tMm(\Omega)$ be a non-extremal measure, and let $\mu = \mu_1 + \mu_2 + \mu_3$ be the decomposition given in Definition \ref{def:tMomega}. In addition, assume that $\|T_0\|_{L^\infty(\Omega; \R^n)} \le 1$ and that $\mu_2, \mu_3$ satisfy the non-extremality assumption \eqref{eq:subcritical_cond}. Then the functional $\Jj_{\mu}$ is lower semicontinuous in the topology of $L^{1}(\Omega)$.
\end{theorem}

Roughly speaking, the strategy of the proof consists in dealing separately with the measures $\mu_2$ and $\mu_3$ of the decomposition of $\mu \in \tMm(\Omega)$.

\begin{theorem} \label{thm:liminf_mu_2}
Let $\mu \in \MH(\Omega)$ be an admissible non-extremal measure such that ${\rm supp}(\mu) \Subset \Omega$ and $\mu = \div T$ for some $T \in \DM^{\infty}(\Omega) \cap C(\Omega \setminus S; \R^n)$, where $S \subset \Omega$ is a compact set such that $\Haus{n-1}(S) = 0$. Then for any $u \in BV(\Omega)$ we have $u^{\lambda}(x) = \tilde{u}(x)$ for $|\mu|$-a.e. $x \in \Omega$. If in addition, we assume that $\|T\|_{L^\infty(\Omega; \R^n)} \le 1$, then the functional $\Jj_{\mu}$ is lower semicontinuous in the topology of $L^{1}(\Omega)$. In particular, for every sequence $(u_j)_j \subset BV(\Omega)$ converging to $u$ in $L^1(\Omega)$ we have
\begin{equation} \label{eq:liminf_local_mu_2}
\liminf_{j \to + \infty} \sqrt{1 + |Du_j|^2}(\Omega') + \int_{\Omega} \tilde{u}_j \, d \mu \ge \sqrt{1 + |Du|^2}(\Omega') + \int_{\Omega} \tilde{u} \, d \mu
\end{equation}
for any open set $\Omega' \subset \Omega$ such that $S \Subset \Omega'$.
\end{theorem}

\begin{proof}
We start by noticing that, since $\mu$ is admissible, then we have $\mu \in \MH(\Omega)$, which yields $|\mu|(S) = 0$, while Lemma \ref{lem:2a2b-bis} implies that $u^\lambda \in L^1(\Omega; |\div T|)$ for all $u \in BV(\Omega)$. In addition, we have $\mu = \div T$ on $\Omega$ and $T \in \DM^{\infty}(\Omega) \cap C(\Omega \setminus S; \R^n)$, so that $|\div T|(\Sigma) = 0$ for any countable $\Haus{n-1}$-rectifiable set $\Sigma \subset (\Omega \setminus S)$, thanks to \cite[Proposition 3.4]{ACM} and \cite[Section 2.4]{crasta2017anzellotti}. In particular, this implies that, given $u \in BV(\Omega)$, we have $$|\div T|(J_u \setminus S) = |\div T|(J_u) = 0,$$ so that $u^{\lambda}(x) = \tilde{u}(x)$ for $|\div T|$-a.e. $x \in \Omega$. In addition, it is easy to notice that
\begin{equation*}
(T, Du)_\lambda = T \cdot Du \ \text{ on } \Omega \setminus S,
\end{equation*}
see also \cite[Remark 4.10]{crasta2019pairings}.
Indeed, we just need to exploit \eqref{eq:pairlambda_weak_conv}: for any $\psi \in C_c(\Omega \setminus S)$ we clearly have $\psi T \in C_c(\Omega \setminus S; \R^n)$, so that, given the sequence $(\hat{u}_j)_{j} \subset C^\infty(\Omega) \cap BV(\Omega) \cap L^\infty(\Omega)$ as in Theorem \ref{lemma:pairlambda-vs-area}, we get
\begin{equation*}
\int_{\Omega \setminus S} \psi \, d (T, Du)_\lambda = \lim_{j \to + \infty} \int_{\Omega} (\psi T) \cdot \nabla \hat{u}_j \, dx = \int_{\Omega \setminus S} (\psi T) \cdot d Du,
\end{equation*}
due to the fact that $\hat{u}_j \to u$ in $BV(\Omega)$-strict. Hence, \eqref{eq:Leibniz_lambda_gen} yields
\begin{equation*}
\div(uT) = \tilde{u} \div T + T \cdot Du  \ \text{ on } \Omega \setminus S.
\end{equation*}
Finally, given that $\Haus{n-1}(S) = 0$ implies $|Du|(S) = 0$, we deduce that $T$ is continuous up to a $|Du|$-negligible set, so that the product $T \cdot Du$ is well posed as a Radon measure on $\Omega$, and the pairing measure satisfies
\begin{equation*}
(T, Du)_{\lambda}(A) = \int_{A \setminus S} T \cdot d Du = \int_A T \cdot d Du  \ \text{ for all Borel sets } A \subset \Omega.
\end{equation*}
Hence, with a little abuse of notation, we shall simply write $(T, Du)_{\lambda} = T \cdot Du$. 
Since $|\div T|(S) = 0$ and, by Theorem \ref{lemma:pairlambda-vs-area}, we have
\[
|\div (uT)| \le |\tilde{u}||\div T| + \|T\|_{L^{\infty}(\Omega;\R^{n})} |Du|\ \text{ on } \Omega\,,
\]
we conclude that $|\div (uT)|(S) = 0$. Therefore, we get
\begin{equation} \label{eq:Leibniz_u_tilde_pair_prod}
\div(uT) = \tilde{u} \div T + T \cdot Du  \ \text{ on } \Omega.
\end{equation}

Thanks to the density-in-energy of smooth functions given by Lemma \ref{lem:density_energy}, it is enough to prove the lower semicontinuity of $\Jj_\mu$ on a sequence of functions $u_{j}\in \cS(\phi)\cap C^{\infty}(\Omega)$, such that $u_{j}\to u$ in $L^{1}(\Omega)$ as $j\to+\infty$. 
Due to the coercivity of $\Jj_\mu$ (Section \ref{sec:coercivity}), we may assume without loss of generality that $u\in \cS(\phi)$, that the sequence $(u_{j})_j$ is bounded in $BV(\Omega)$ and that $Du_j \weakto Du$ in $\Mm(\Omega; \R^n)$, up to a subsequence.

Now let $S_t := \left \{ x \in \Omega : {\rm dist}(x, S) < t \right \}$ for $0 < t < {\rm dist}(S, \partial \Omega)$. Hence, for all $\eps > 0$ there exists $\delta > 0$ such that 
\begin{equation} \label{eq:S_t_estimates}
\left | \int_{S_t} T \cdot d Du \right | \le |Du|(S_t) \le \sqrt{1 + |Du|^2}(S_t) < \eps \ \text{ for all } 0 < t < \delta.
\end{equation}
Let $U \Subset \Omega$ be an open set such that ${\rm supp}(\mu) \Subset U$. 
Let now $\eta \in C^{\infty}_c(\Omega)$ be such that $0 \le \eta \le 1$ and $\eta \equiv 1$ on $U$. Then, by the Leibniz rule for smooth scalar functions (see for instance \cite[Theorem 3.1]{CF1}) we have
\begin{equation*}
\div(\eta T) = \eta \div T + T \cdot \nabla \eta \, \Leb{n} = \eta \mu + T \cdot \nabla \eta \, \Leb{n} = \mu + T \cdot \nabla \eta \, \Leb{n},
\end{equation*}
so that $\div(\eta T) = \mu$ on $U$. Hence, for all $v \in BV(\Omega)$ we get
\begin{equation*}
\int_\Omega \tilde{v} \, d \mu = \int_U \tilde{v} \, d \mu =  \int_U \tilde{v} \, d \div(\eta T) = \int_\Omega \tilde{v} \, d \div(\eta T),
\end{equation*}
so that we can assume that $T \in C_c(\Omega; \R^n)$ without loss of generality. Therefore, we have ${\rm supp}(u_jT) \Subset \Omega$, and, by Lemma \ref{lem:div_comp_supp_no_trace} and \eqref{eq:Leibniz_u_tilde_pair_prod}, we get
\begin{equation*}
\int_\Omega u_j \, \div T = - \int_{\Omega} T \cdot \nabla u_j \, dx + \div(u_j T)(\Omega) = - \int_{\Omega} T \cdot \nabla u_j \, dx. 
\end{equation*}
Let now $\psi \in C^{\infty}_c(S_{2t})$ be such that $0 \le \psi \le 1$ and $\psi \equiv 1$ on $S_{3t/2}$ for some $0 < t < \frac{\delta}{3}$. Then, we get
\begin{equation*}
\int_{\Omega} T \cdot \nabla u_j \, dx = \int_{\Omega \setminus \overline{S_t}} (1 - \psi) T  \cdot \nabla u_j \, dx + \int_{S_{2t}}  \psi T \cdot \nabla u_j \, dx.
\end{equation*}
It is clear that $(1 - \psi) T \in C_c(\Omega \setminus \overline{S_t}; \R^n)$ so that 
\begin{equation} \label{eq:conv_lim_S_t}
\lim_{j \to + \infty}  \int_{\Omega \setminus \overline{S_t}} (1 - \psi) T \cdot \nabla u_j \, dx = \int_{\Omega \setminus \overline{S_t}} (1 - \psi) T \cdot d Du =  \int_{\Omega} (1 - \psi) T \cdot d Du.
\end{equation}
On the other hand, we easily obtain
\begin{equation} \label{eq:compensation_on_S_t}
\sqrt{1 + |Du_j|^2}(S_{2t}) - \int_{S_{2t}}  \psi T \cdot \nabla u_j \, dx \ge (1 - \|T\|_{L^\infty(\Omega; \R^n)}) |Du_j|(S_{2t}) \ge 0.
\end{equation}
All in all, we exploit the superadditivity of the liminf, the lower semicontinuity of the area functional, \eqref{eq:S_t_estimates}, \eqref{eq:conv_lim_S_t} and \eqref{eq:compensation_on_S_t} to obtain
\begin{align*}
\liminf_{j \to + \infty} \Jj_{\mu}[u_j] & \ge \liminf_{j \to + \infty} \sqrt{1 + |Du_j|^2}(B \setminus \overline{S_{2t}}) - \lim_{j \to + \infty} \int_{\Omega  \setminus \overline{S_t}} (1 - \psi) T  \cdot \nabla u_j \, dx \\
& \quad + \liminf_{j \to + \infty} \left ( \sqrt{1 + |Du_j|^2}(S_{2t}) - \int_{S_{2t}}  \psi T \cdot \nabla u_j \, dx \right )\\
& \ge \sqrt{1 + |Du|^2}(B \setminus \overline{S_{2t}}) - \int_{\Omega} (1 - \psi) T \cdot d Du \\
& = \sqrt{1 + |Du|^2}(B) - \int_{\Omega} T \cdot d Du - \sqrt{1 + |Du|^2}(\overline{S_{2t}}) +  \int_{S_{2t}} \psi T \cdot d Du \\
& \ge \sqrt{1 + |Du|^2}(B) - \int_{\Omega} T \cdot d Du - 2 \eps \\
& =  \sqrt{1 + |Du|^2}(B) + \int_{\Omega} \tilde{u} \, d \div T - \div(uT)(\Omega) - 2 \eps = \Jj_{\mu}[u] - 2 \eps,
\end{align*}
where the last equality follows from \eqref{eq:Leibniz_u_tilde_pair_prod}, the fact that ${\rm supp}(uT) \Subset \Omega$ and Lemma \ref{lem:div_comp_supp_no_trace}.
Thus, this proves the lower semicontinuity of $\Jj_\mu$, since $\eps$ is arbitrary. Finally, it is easy to notice that the argument above also holds when we restrict the area functional on some open set $\Omega'$ such that $S \Subset \Omega' \subset \Omega$ since it is only important that $S_{2t} \Subset \Omega'$ for some $t > 0$. Therefore, the argument above allows us to prove also \eqref{eq:liminf_local_mu_2}.
\end{proof}

Now, we deal with the last part of the measure $\mu \in \tMm(\Omega)$; that is, the one concentrated on a Lebesgue-negligible set $\Gamma$. In the next two lemmas, we prove some crucial estimates that will allow us to localize. The first is an upper bound of the Lebesgue measure of the sub-levels and the super-levels of any function $v\in \cS(\phi)$ in terms of the function $\phi$, the constant $L$, the energy value $\Jj_{\mu}[v]$, and the truncation parameter $M$.

\begin{lemma}\label{lem:apriori1}
Let $\mu \in \MH(\Omega)$ be non-extremal and admissible, and let $v\in \cS(\phi)$. Then for every $M>0$ we have
\begin{equation*}
\Big|\{x\in \Omega:\ |v(x)|>M\}\Big| \le \frac{1}{C_B(1-L)M} \left(\Jj_{\mu}[v] + \int_{\de \Omega} |{\rm Tr}_{\partial \Omega}(\phi)|\, d\Haus{n-1}\right)\,,
\end{equation*}
where $L\in (0,1)$ is the constant appearing in the definition of non-extremality, and $C_{B}$ denotes the Poincar\'e constant of $B$.
\end{lemma}
\begin{proof}
Let $v_{0}$ be the zero-extension of $v$ outside $\Omega$. By \eqref{eq:stimaJmu} we have
\[
\Jj_{\mu}[v] \ge (1-L) |Dv_{0}|(B) - \int_{\de \Omega}|{\rm Tr}_{\partial \Omega}(\phi)|\, d\Haus{n-1}\,,
\]
hence by Poincar\'e's inequality on $B$ combined with Chebichev's inequality we obtain
\begin{align*}
\Jj_{\mu}[v] &\ge C_B(1-L) \int_\Omega |v_{0}| \, dx - \int_{\de \Omega}|{\rm Tr}_{\partial \Omega}(\phi)|\, d\Haus{n-1}\\
&\ge C_B(1-L)M |\{x\in \Omega:\ |v(x)|>M\}| - \int_{\de \Omega}|{\rm Tr}_{\partial \Omega}(\phi)|\, d\Haus{n-1}\,,
\end{align*}
whence the conclusion follows.
\end{proof}

In the next lemma, we consider a local truncation $v_M$ of $v\in \cS(\phi)$ and we bound $\Jj_{\mu}[v_M]$ from above in terms of $\Jj_{\mu}[v]$, up to a suitably estimated error. 

\begin{lemma}\label{lem:apriori2}
Let $\mu \in \MH(\Omega)$ be non-extremal and admissible.
Given an open set $U\Subset \Omega$ with Lipschitz boundary, a function $v\in \cS(\phi) \cap C^\infty(\Omega)$, and $M>0$, we define the local $M$-truncation of $v$ in $U$ as
\[
v_{M}(x) = \begin{cases}
T_{M}(v)(x) & \text{if }x\in U,\\
v(x) & \text{otherwise.}
\end{cases}
\]
Then we have
\begin{align}\nonumber
\Jj_{\mu}[v_{M}]\ &\le\ \Jj_{\mu}[v] + \frac{2L^{2}}{C_B(1+L)(1-L)^{2} M}\left( \Jj_{\mu}[v] + \int_{\de\Omega} |{\rm Tr}_{\partial \Omega}(\phi)|\, d\Haus{n-1} \right)\\\label{eq:apriori3}
&\qquad\qquad\qquad + (1+L) \int_{\de U}(|v| - M)_{+}\, d\Haus{n-1},
\end{align}
where $L\in (0,1)$ is the constant appearing in the definition of non-extremality, and $C_{B}$ denotes the Poincar\'e constant of $B$.
\end{lemma}
\begin{proof}
We see that \eqref{eq:apriori3} follows directly from 
\begin{equation}\label{eq:tronca1}
\Jj_{\mu}[v_{M}] \le \Jj_{\mu}[v] + \frac{2L^{2}}{1-L^{2}} \left |\left \{x \in \Omega : |v(x)|>M \right \}\right | + (1 + L) \int_{\de U}(|v| - M)_{+}\, d\Haus{n-1},
\end{equation}
thanks to Lemma \ref{lem:apriori1}. For the proof of \eqref{eq:tronca1} we observe that 
\begin{align*}
\Jj_{\mu}[v_{M}] &= \sqrt{1+|D v_{M}|^{2}}(B) + \int_\Omega v_{M}\, d\mu\\ 
&= \sqrt{1+|Dv|^{2}}(B\setminus \overline U) + \int_{\de U} (|v| - M)_{+}\, d\Haus{n-1} + \int_{U}\sqrt{1+|\nabla v_{M}|^{2}}\, dx + \int_\Omega v_{M}\, d\mu\\
&= \sqrt{1+|Dv|^{2}}(B\setminus \overline U) + \int_{\de U} (|v| - M)_{+}\, d\Haus{n-1} + \int_{U}\sqrt{1+|\nabla v|^{2}}\, dx \\
&\qquad\qquad - \int_{\{|v|>M\}\cap U} \sqrt{1+|\nabla v|^{2}}\, dx + \Big|\{|v|>M\}\cap U\Big| + \int_\Omega v_{M}\, d\mu\\
&= \Jj_{\mu}[v] + \int_{\de U} (|v| - M)_{+}\, d\Haus{n-1} - \int_{\{|v|>M\}\cap U} \left(\sqrt{1+|\nabla v|^{2}}-1\right)\, dx \\
& \quad - \int_{\{|v|>M\}\cap U} (v - v_M)\, d\mu\\
&= \Jj_{\mu}[v] + \int_{\de U} (|v| - M)_{+}\, d\Haus{n-1} + N\,,
\end{align*}
where we have set
\begin{equation*}
N = - \int_{\{|v|>M\}\cap U} \left(\sqrt{1+|\nabla v|^{2}}-1\right)\, dx - \int_{\{|v|>M\}\cap U} (v - v_M)\, d\mu\,.
\end{equation*}
Next, we observe the following, elementary fact: 
\begin{equation*}
\forall\, a>0,\qquad \sqrt{1+a^{2}} - 1 > La\quad \Leftrightarrow\quad a> \frac{2L}{1-L^{2}}\,.
\end{equation*}
Consequently, we can set $w = v - v_M$ and, after noticing that $|w| = |v| - M$ and $|\nabla w| = |\nabla v|$ almost everywhere on $\{|v|>M\} \cap U$, we obtain
\begin{align*}
N &\le -L\int_{\{|v|>M,\ |\nabla w|> \frac{2L}{1-L^{2}}\}\cap U} |\nabla w| \, dx - \int_{\{|v|>M\}\cap U} w\, d\mu\\
&= -L\int_{\{|v|>M\}\cap U} |\nabla w| \, dx + L \int_{\{|v|>M,\ |\nabla w|\le \frac{2L}{1-L^{2}}\}\cap U}|\nabla w| \, dx - \int_{\{|v|>M\}\cap U} w\, d\mu\\
&\le \frac{2L^{2}}{1-L^{2}}\Big|\{|v|>M\}\cap U\Big| -L\int_{\{|w|>0\}\cap U} |\nabla w| \, dx - \int_{\{|w|>0\}\cap U} w\, d\mu \,.
\end{align*}
At this point we can use the coarea formula on the first integral and the Cavalieri representation of the second integral, obtaining
\begin{align*}
N &\le \frac{2L^{2}}{1-L^{2}}\Big|\{|v|>M\}\cap U\Big| -\int_{0}^{+\infty} \Big(L\, \Per(\{w>t\};U) + \mu(\{w>t\}\cap U)\Big)\, dt\\
&\qquad \qquad -\int_{0}^{+\infty} \Big(L\,\Per(\{w<-t\};U) -\mu(\{w<-t\}\cap U)\Big)\, dt \\
&= \frac{2L^{2}}{1-L^{2}}\Big|\{|v|>M\}\cap U\Big| -\int_{M}^{+\infty} \Big(L\, \Per(\{v>t\}\cap U) + \mu(\{v>t\}\cap U)\Big)\, dt\\
&\qquad \qquad -\int_{M}^{+\infty} \Big(L\,\Per(\{v<-t\}\cap U) -\mu(\{v<-t\}\cap U)\Big)\, dt  + L \int_{\de U}(|v|-M)_{+}\, d\Haus{n-1}\\
&\le \frac{2L^{2}}{1-L^{2}}\Big|\{|v|>M\}\cap U\Big| +L\int_{\de U}(|v|-M)_{+}\, d\Haus{n-1}\,,
\end{align*}
where in the last inequality we also used \eqref{eq:subcritical_cond}, so that \eqref{eq:apriori3} follows.
\end{proof}

Now we can deal with the lower semicontinuity of $\Jj_\mu$, under the assumption that $\mu$ is a non-extremal admissible measure concentrated on a Lebesgue-negligible set.

\begin{theorem} \label{thm:liminf_mu_3}
Let $\mu \in \MH(\Omega)$ be an admissible measure such that $|\mu|(\Omega \setminus \Gamma) = 0$ for some compact set $\Gamma \Subset \Omega$ with $|\Gamma| =0$. Assume in addition that $\mu$ satisfies the non-extremality assumption \eqref{eq:subcritical_cond}. Then the functional $\Jj_{\mu}$ is lower semicontinuous in the topology of $L^{1}(\Omega)$. In particular, given $u \in BV(\Omega)$, for every sequence $(u_j)_j$ bounded in $BV(\Omega)$ and converging to $u$ in $L^1(\Omega)$, we have
\begin{equation} \label{eq:liminf_local_mu_3}
\liminf_{j \to + \infty} \sqrt{1 + |Du_j|^2}(\Omega') + \int_{\Omega} u^\lambda_j \, d \mu \ge \sqrt{1 + |Du|^2}(\Omega') + \int_{\Omega} u^\lambda \, d \mu
\end{equation}
for any open set $\Omega' \subset \Omega$ such that $\Gamma \Subset \Omega'$.
\end{theorem} 

\begin{proof}
First, we note that, due to the coercivity of $\Jj_\mu$ (Section \ref{sec:coercivity}), we need to check the lower semicontinuity only on $\cS(\phi)$. Then, thanks to Lemma \ref{lem:density_energy}, any $v\in \cS(\phi)$ can be approximated by a sequence of functions $(v_{h})_{h}\subset \cS(\phi)\cap C^{\infty}(\Omega)$ in such a way that $\Jj_{\mu}[v] = \lim_{h} \Jj_{\mu}[v_{h}]$. Therefore, in order to check the lower semicontinuity of $\Jj_{\mu}$ we can fix a function $u\in \cS(\phi)$ and a sequence of functions $v_{h}\in \cS(\phi)\cap C^{\infty}(\Omega)$, such that $v_{h}\to u$ in $L^{1}(\Omega)$ as $h\to+\infty$. Then, arguing by contradiction, we assume there exists $\rho \in (0, 1)$ such that $\Jj_{\mu}[v_{h}] \le \Jj_{\mu}[u] - K\rho$ for all $h$, and for a positive constant $K$ that will be chosen at the end of the argument. In particular, it will be crucial to apply \eqref{eq:apriori3} on a suitably chosen neighbourhood $U$ of $\Gamma$ such that
\begin{equation} \label{eq:de_U_good_prop}
\Haus{n-1}(\de U \cap J_u) = 0 \ \text{ and } \ v_{h} \to \widetilde u \text{ in } L^{1}(\de U; \Haus{n-1}). 
\end{equation}
To find $U$, we recall that $v_{h}$ converges to $u$ in $L^{1}(\Omega)$ and we use the coarea formula with the distance function from $\Gamma$. Therefore, for $h$ large enough we can set 
\[
v_{h,M}(x) = \begin{cases}
T_{M}(v_{h})(x) & \text{if }x\in U,\\
v_h(x) & \text{otherwise,}
\end{cases}
\]
and deduce from \eqref{eq:apriori3} that
\begin{align*}
\Jj_{\mu}[v_{h,M}] &\le \Jj_{\mu}[v_{h}] + \frac{C}{M} + 2\int_{\de U}(|\widetilde{u}| - M)_{+}\, d\Haus{n-1}\,,
\end{align*}
for all $h\ge h_{0}$, and for some $h_{0}$ and $C$ depending only on the choice of $u, U$ and $L \in (0,1)$. 
Then, we note that by choosing $M$ large enough and only depending on $u$ and $U$ (hence, uniformly with respect to $h$) the term 
\[
\frac{C}{M} + 2\int_{\de U}(|\widetilde{u}| - M)_{+}\, d\Haus{n-1}
\]
can be made arbitrarily small. In conclusion, for every $\rho>0$ there exist $M>0$ and $h_{0}$, such that 
\[
\Jj_{\mu}[v_{h,M}] \le \Jj_{\mu}[v_{h}] + \rho\qquad \forall\, h\ge h_{0}\,.
\] 
Thus, we may directly take $v_{h}\in W^{1,1}(\Omega\setminus \de U) \cap BV(\Omega)$, such that $|v_{h}|\le M$ on $U\supset \Gamma$, and $\Jj_{\mu}[v_{h}] \le \Jj_{\mu}[u] - (K-1)\rho$ for all $h$. At the same time, we can analogously define the local $M$-truncation of the function $u$ on $U$, denoted as $u_{M}$, so that by the continuity of the area functional with respect to convergence in $BV$-norm of the $M$-truncations and \eqref{eq:conv_trunc_lambda}, we can find $M$ large enough with the property $\Jj_{\mu}[u_{M}]\ge \Jj_{\mu}[u] - \rho$. All in all, this implies that we can take from the very beginning $|u|, |v_h|\le M$ on $U$ and $\Jj_{\mu}[v_{h}] \le \Jj_{\mu}[u] - (K-2)\rho$ for all $h$. 

Since 
\[
\sqrt{1+|Du|^{2}}(U) = \int_{U}\sqrt{1+|\nabla u|^{2}}\, dx + |D^{s}u|(U)
\]
we can choose $U$ with Lebesgue measure so small that 
\[
\sqrt{1+|Du|^{2}}(U) \le |Du|(\Gamma) + \rho\,,
\]
given that $|\Gamma| = 0$.
By \eqref{eq:de_U_good_prop} we have $\sqrt{1+|Du|^{2}}(\de U) = 0$, and since 
\begin{equation} \label{eq:liminf_out_U}
\liminf_{h} \sqrt{1+|Dv_{h}|^{2}}(B\setminus \overline U) \ge \sqrt{1+|Du|^{2}}(B\setminus \overline U)
\end{equation}
by the lower semicontinuity of the area functional, we can assume that 
\[
\sqrt{1+|Dv_{h}|^{2}}(B\setminus U) \ge \sqrt{1+|Du|^{2}}(B\setminus U) - \rho
\]
for $h$ large enough. Then, up to possibly choosing a bigger $h_{0}$, we have for all $h\ge h_{0}$
\begin{align}\nonumber
|Dv_{h}|(U) - |Du|(\Gamma) &+ \int_{\Gamma}(v_{h}-u^\lambda)\, d\mu\ 
\\\nonumber
&\le\ \sqrt{1+|Dv_{h}|^{2}}(U) - \sqrt{1+|Du|^{2}}(U) + \rho + \int_{\Gamma}(v_{h}-u^\lambda)\, d\mu\\\nonumber
&\le\ \sqrt{1+|Dv_{h}|^{2}}(B) - \sqrt{1+|Du|^{2}}(B) + 2\rho + \int_{\Gamma}(v_{h}-u^\lambda)\, d\mu\\\nonumber
&\le\ \Jj_{\mu}[v_{h}] - \Jj_{\mu}[u] + 2\rho\\\label{eq:deltavartotU}
&\le\ -(K-4)\rho\,.
\end{align}

In the next step, we first recall Lusin's Theorem and find a compact set $\Gamma_{1}\subset \Gamma$ such that $u^{\lambda}$ is continuous on $\Gamma_{1}$ and, setting $\Gamma_{2} = \Gamma \setminus \Gamma_{1}$, we have
\begin{equation}\label{eq:estgamma2}
|Du|(\Gamma_{2}) + M|\mu|(\Gamma_{2}) + \int_{\Gamma_{2}} |u^{\lambda}|\, d|\mu| < \rho\,.
\end{equation}
Now we set 
\begin{equation} \label{eq:epsilon_def_rho_Gamma}
\eps = (1+|\mu|(\Gamma))^{-1}\rho.
\end{equation} 
Then, we notice that, by \eqref{eq:mu_pm_differentiation}, for $|\mu|$-a.e. $x \in \Omega^{\mp}$ we have
\begin{equation*}
\mu^{\pm}(B_r(x)) = o(|\mu|(B_r(x)) ) \ \text{ as } r \to 0^+.
\end{equation*}
In addition, by the definition of $u^{\pm}$, for $\Haus{n-1}$-a.e. $x \in \Omega$ we see that 
\begin{align*}
\lim_{r \to 0^+} \frac{|\{u > t \} \cap B_r(x)|}{|B_r(x)|} = 0 & \ \text{ for all } t \ge u^+(x) + \eps, \\
\lim_{r \to 0^+} \frac{|\{u < t \} \cap B_r(x)|}{|B_r(x)|} = 0 & \ \text{ for all } t \le u^-(x) - \eps.
\end{align*}
Therefore, by exploiting the spherical coarea formula, we deduce that there exists a sequence $(r_j)_{j \in \N}$, $r_j \to 0^+$, such that
\begin{align*}
\lim_{j \to + \infty} \frac{\Haus{n-1}(\{u^+ > t \} \cap \partial B_{r_j}(x))}{\Per(B_{r_j}(x))} = 0 & \ \text{ for } \Leb{1}\text{-a.e. } t > u^+(x) + \eps, \\
\lim_{j \to + \infty} \frac{\Haus{n-1}(\{u^- < t \} \cap \partial B_{r_j}(x))}{\Per(B_{r_j}(x))} = 0 & \ \text{ for } \Leb{1}\text{-a.e. } t < u^-(x) - \eps,
\end{align*}
since ${\rm Tr}_{\de B_r(x)}(\chi_{\{u > t\}}) = \chi_{\{u^+ > t \}}$ and ${\rm Tr}_{\de B_r(x)}(\chi_{\{u < t\}}) = \chi_{\{u^- < t \}}$ $\Haus{n-1}$-a.e. on $\de B_r(x)$ and for $\Leb{1}$-a.e. $t$. Moreover, we can additionally assume that 
\begin{align}\label{eq:convgen1}
\text{$v_{h}\to \widetilde{u}$ in $L^{1}(\de B_{r_{j}}(x))$ as $h\to+\infty$, and $\widetilde{u} = u^{+} = u^-$ $\Haus{n-1}$-a.e. on $\de B_{r_{j}}(x)$.}
\end{align}

Hence, exploiting these facts and Lemma \ref{lem:sopraupiueps}, we can apply Besicovitch covering theorem and find a finite family of pairwise disjoint balls $\{B_{i} := B(x_{i},r_{i})\}_{i=1,\dots, m}$ centered on $\Gamma_{1}$, contained in $U$, and covering $\widetilde{\Gamma}_{1}\subset \Gamma_{1}$ such that, if we set $\widetilde{\Gamma}_{2} = \Gamma \setminus \widetilde{\Gamma}_{1}$, the inequality \eqref{eq:estgamma2} is also satisfied with $\widetilde{\Gamma}_{2}$ in place of $\Gamma_{2}$, and the following properties are verified:
\begin{equation}\label{eq:upluscont}
|u^{\lambda}(x) - u^{\lambda}(x_{i})| < \eps\qquad \forall\, x\in \Gamma_{1}\cap B_{i},
\end{equation}
\begin{align}
 \Haus{n-1}( \{ u^+ > t \} \cap \partial B_i)  < \frac{\Per(B_i)}{2}  & \ \text{ for } \Leb{1}\text{-a.e. } t > u^+(x_i) + \eps, \label{eq:trace_min_sup_level_set} \\
 \Haus{n-1}( \{ u^- < t \} \cap \partial B_i) < \frac{\Per(B_i)}{2} & \ \text{ for } \Leb{1}\text{-a.e. } t < u^-(x_i) - \eps, \label{eq:trace_min_sub_level_set}
\end{align}
\begin{equation}\label{eq:cutvarBi}
\int^{u^{-}(x_{i})-\eps}_{- M}\Per(\{u<t\};B_{i})\, dt + \int_{u^{+}(x_{i})+\eps}^{M}\Per(\{u>t\};B_{i})\, dt < \rho \frac{|Du|(B_{i})}{1+|Du|(U)},
\end{equation}
and
\begin{equation} \label{eq:decay_mu_pm_x_i}
\mu_{\pm}(B_i) \le \frac{\eps}{2M + 1} |\mu|(B_i) \ \text{ whenever } x_i \in \Omega^{\mp},
\end{equation}
for $i=1,\dots,m$. For the sake of simplicity, from now on we will omit the tilde in $\widetilde{\Gamma}_{1}$ and $\widetilde{\Gamma}_{2}$. 

Therefore by \eqref{eq:deltavartotU} and \eqref{eq:estgamma2} we infer that
\begin{align}
\sum_{i=1}^{m} \Big(|Dv_{h}|(B_{i}) &- |Du|(B_{i}) + \int_{\Gamma_{1}\cap B_{i}} (v_{h}-u^{\lambda})\, d\mu\Big) \le |Dv_{h}|(U) - |Du|(\Gamma_{1}) + \int_{\Gamma_{1}} (v_{h}-u^{\lambda})\, d\mu \nonumber \\
&= |Dv_{h}|(U) - |Du|(\Gamma) + \int_{\Gamma} (v_{h}-u^{\lambda})\, d\mu + |Du|(\Gamma_{2}) -\int_{\Gamma_{2}} (v_{h}-u^{\lambda})\, d\mu \nonumber \\
&\le -(K-4)\rho + |Du|(\Gamma_{2}) + M|\mu|(\Gamma_{2}) + \int_{\Gamma_{2}}|u^{\lambda}|\, d|\mu|\nonumber\\
&\le -(K-5)\rho\,. \label{eq:K-7_rho}
\end{align}
Now we set
 \begin{equation*}
 \widehat{v}_{h}(x) = \begin{cases}
\max\Big(\min(u^{+}(x_{i})+\eps,v_h(x)),u^{-}(x_{i})-\eps\Big) & \text{if }x\in B_i \text{ for some } i,\\
v_h(x) & \text{otherwise\,,}
\end{cases}
\end{equation*}
and
 $$\widehat{u}(x) = \begin{cases}
\max\Big(\min(u^{+}(x_{i})+\eps,u(x)),u^{-}(x_{i})-\eps\Big) & \text{if }x\in B_i \text{ for some } i,\\
u(x) & \text{otherwise.}
\end{cases}$$ 
One can easily check that $\widehat{v}_{h} \to \widehat{u}$ in $L^1(\Omega)$, hence by the lower semicontinuity of the total variation, and by choosing $h$ large enough, we get
\begin{equation}\label{eq:DhatvDu}
\sum_{i=1}^{m} |D\widehat{v}_{h}|(B_{i}) \ge \sum_{i=1}^{m} |D\widehat{u}|(B_{i}) - \rho\,.
\end{equation}
On the other hand, by coarea and \eqref{eq:cutvarBi} we get
\begin{equation}\label{eq:DuDhatu}
\sum_{i=1}^{m} |Du|(B_{i}) \le \sum_{i=1}^{m} |D\widehat{u}|(B_{i}) +  \frac{\rho}{1+|Du|(U)}\sum_{i=1}^{m} |Du|(B_{i}) \le \sum_{i=1}^{m} |D\widehat{u}|(B_{i}) + \rho\,.
\end{equation}
Therefore, by combining \eqref{eq:DhatvDu} with \eqref{eq:DuDhatu}, for $h$ large enough we obtain
\begin{equation}\label{eq:Dvhat1}
\sum_{i=1}^{m} |D\widehat{v}_{h}|(B_{i}) \ge \sum_{i=1}^{m} |Du|(B_{i}) - 2\rho\,.
\end{equation}
On the other hand, again thanks to the coarea formula, we have 
\begin{align}\nonumber
\sum_{i=1}^{m}|Dv_{h}|(B_{i}) &= \sum_{i=1}^{m}\left(|D\widehat{v}_{h}|(B_{i}) + \int_{u^{+}(x_{i})+\eps}^{M} \Per(\{v_{h}>t\};B_{i})\, dt + \int_{-M}^{u^{-}(x_{i})-\eps} \Per(\{v_{h} < t\};B_{i})\, dt\right)\\\nonumber
&= \sum_{i=1}^{m}\left(|D\widehat{v}_{h}|(B_{i}) + \int_{u^{+}(x_{i})+\eps}^{M} \Per(\{v_{h}>t\}\cap B_{i})\, dt + \int_{-M}^{u^{-}(x_{i})-\eps} \Per(\{v_{h} < t\}\cap B_{i})\, dt\right. \\\nonumber
& \quad -\left. \int_{u^{+}(x_{i})+\eps}^{M} \Haus{n-1}(\{v_{h}>t\}\cap \de B_{i})\, dt - \int_{-M}^{u^{-}(x_{i})-\eps} \Haus{n-1}(\{v_{h} <t\}\cap \de B_{i})\, dt\right) \\\nonumber
&= \sum_{i=1}^{m}\Bigg(|D\widehat{v}_{h}|(B_{i}) + \int_{u^{+}(x_{i})+\eps}^{M} \Per(\{v_{h}>t\}\cap B_{i})\, dt + \int_{-M}^{u^{-}(x_{i})-\eps} \Per(\{v_{h}<t\}\cap B_{i})\, dt \\\label{eq:stimona1}
& \quad - \int_{\de B_{i}} \big(v_{h}-u^{+}(x_{i})-\eps\big)_{+}\, d\Haus{n-1} - \int_{\de B_{i}} \big(u^{-}(x_{i}) - \eps - v_h \big)_{+}\, d\Haus{n-1}\Bigg). 
\end{align}
Then, we notice that, combining the Poincar\'e-trace inequality \eqref{eq:poincaretrace} with \eqref{eq:trace_min_sup_level_set} and \eqref{eq:trace_min_sub_level_set}, we obtain
\begin{equation*}
\Haus{n-1}(\{u^{+}>t\}\cap \de B_{i}) \le C_{PT} \Per(\{u>t\};B_{i}) \ \text{ for } \Leb{1}\text{-a.e. } t > u^+(x_i) + \eps
\end{equation*}
and
\begin{equation*}
\Haus{n-1}(\{u^{-}<t\}\cap \de B_{i}) \le C_{PT} \Per(\{u<t\};B_{i}) \ \text{ for } \Leb{1}\text{-a.e. } t < u^-(x_i) - \eps.
\end{equation*}
Thanks to these last two inequalities, \eqref{eq:convgen1}, and \eqref{eq:cutvarBi}, for $h$ large enough from \eqref{eq:stimona1} we get the following estimate:
\begin{align*}
\sum_{i=1}^{m}|Dv_{h}|&(B_{i}) 
 -\sum_{i=1}^{m}\left(|D\widehat{v}_{h}|(B_{i}) + \int_{u^{+}(x_{i})+\eps}^{M} \Per(\{v_{h}>t\}\cap B_{i})\, dt + \int_{-M}^{u^{-}(x_{i})-\eps} \Per(\{v_{h}<t\}\cap B_{i})\, dt\right) \\
&\ge - \sum_{i=1}^{m}\int_{\de B_{i}} \big(u^{+}-u^{+}(x_{i})-\eps\big)_{+}\, d\Haus{n-1} - \rho -\sum_{i = 1}^m \int_{\de B_{i}} \big(u^{-}(x_{i})-\eps - u^{-} \big)_{+}\, d\Haus{n-1} - \rho \\
&= - \sum_{i=1}^{m}\left(\int_{u^{+}(x_{i})+\eps}^{M} \Haus{n-1}(\{u^{+}>t\}\cap \de B_{i})\, dt + \int_{-M}^{u^{-}(x_{i})-\eps} \Haus{n-1}(\{u^{-} < t\}\cap \de B_{i})\, dt\right) - 2\rho \\
&\ge - C_{PT}\sum_{i=1}^{m}\left(\int_{u^{+}(x_{i})+\eps}^{M} \Per(\{u>t\};B_{i}) \, dt + \int_{-M}^{u^{-}(x_{i})-\eps} \Per(\{u<t\};B_{i})\, dt\right) - 2\rho \\
&\ge - 2(C_{PT}+1)\rho\,.
\end{align*}
We can thus combine the last chain of inequalities with \eqref{eq:Dvhat1} and get
\begin{align}\label{eq:stimadiffvar}
\sum_{i=1}^{m}\Big(|Dv_{h}|(B_{i}) - |Du|(B_{i})\Big) & \ge \sum_{i=1}^{m} \left(\int_{u^{+}(x_{i})+\eps}^{M} \Per(\{v_{h}>t\}\cap B_{i})\, dt + \int_{-M}^{u^{-}(x_{i})-\eps} \Per(\{v_{h}<t\}\cap B_{i})\, dt\right) \\
& \quad - 2(C_{PT}+2)\rho \nonumber
\end{align}
We pass now to the integral with respect to $\mu$. Since $\lambda = \lambda_\mu$, we see that $u^\lambda(x_i) = u^{\pm}(x_i)$ whenever $x_i \in \Omega^\mp$, respectively, see \eqref{eq:lambda_mu_representative}. Hence, we consider separately these two cases. If $x_i \in \Omega^-$, then, using \eqref{eq:upluscont}, \eqref{eq:decay_mu_pm_x_i} and the non-extremality of $\mu$, we get
\begin{align*}
\int_{\Gamma_{1}\cap B_{i}} (v_{h} - u^{\lambda})\, d\mu & = \int_{\Gamma_{1}\cap B_{i}} (v_{h} - u^{+}(x_{i}) -  \eps )\, d\mu + \eps \mu(\Gamma_{1} \cap B_i) - \int_{\Gamma_1 \cap B_i} (u^\lambda - u^+(x_i)) \, d \mu \\
& \ge \int_{\Gamma_{1}\cap B_{i}} (v_{h} - u^{+}(x_{i}) -  \eps ) \, d\mu_{+} - \int_{\Gamma_{1}\cap B_{i}} (v_{h} - u^{+}(x_{i}) -  \eps ) \, d\mu_{-} - 2 \eps |\mu|(B_i) \\
& \ge - (2M + \eps) \mu_+(B_i) - \int_{\Gamma_{1}\cap B_{i}} (v_{h} - u^{+}(x_{i}) -  \eps )_+ \, d\mu_{-} - 2 \eps |\mu|(B_i) \\
& \ge - \int_{\Gamma_{1}\cap B_{i}} (v_{h} - u^{+}(x_{i}) -  \eps )_+ \, d\mu_{-} - 3 \eps |\mu|(B_i) \\
& \ge \int_{\Gamma_{1}\cap B_{i}} (v_{h} - u^{+}(x_{i}) -  \eps )_+ \, d\mu - \int_{\Gamma_{1}\cap B_{i}} (v_{h} - u^{+}(x_{i}) -  \eps )_+ \, d\mu_{+} - 3 \eps |\mu|(B_i) \\
& \ge \int_{\Gamma_{1}\cap B_{i}} (v_{h} - u^{+}(x_{i}) -  \eps )_+ \, d\mu - 4 \eps |\mu|(B_i) \\
& = \int_{u^{+}(x_{i}) + \eps}^{M} \mu(\{v_{h}>t\}\cap B_{i})\, dt - 4 \eps |\mu|(B_i) \\
& \ge - \int_{u^{+}(x_{i}) + \eps}^{M} L \Per(\{v_{h}>t\}\cap B_{i})\, dt - 4 \eps |\mu|(B_i).
\end{align*}
Arguing analogously, for $x_i \in \Omega^+$ we get
\begin{equation*}
\int_{\Gamma_{1}\cap B_{i}} (v_{h} - u^{\lambda})\, d\mu \ge - \int_{-M}^{u^{-}(x_{i}) - \eps} L P(\{v_{h}< t\}\cap B_{i})\, dt - 4 \eps |\mu|(B_i)
\end{equation*}
In conclusion, by combining \eqref{eq:stimadiffvar}, the above calculations and the definition of $\eps$ \eqref{eq:epsilon_def_rho_Gamma}, we obtain
\begin{align*}
-(K-5)\rho &\ge \sum_{i=1}^{m} \Big(|Dv_{h}|(B_{i}) - |Du|(B_{i}) + \int_{\Gamma_{1}\cap B_{i}} (v_{h}-u^{\lambda})\, d\mu\Big)\\ 
& \ge - \left ( \sum_{i = 1}^m 4 \eps |\mu|(B_i) \right ) - 2(C_{PT}+2)\rho \\
&\ge - 2(C_{PT}+4)\rho
\end{align*}
which gives a contradiction as soon as we fix $K> 2 C_{PT}+13$. This concludes the proof of the lower semicontinuity of $\Jj_{\mu}$.

Finally, we notice that, thanks to \eqref{eq:liminf_out_U}, the area functional outside a small neighborhood of $\Gamma$ does not play any relevant role in the proof, so we may repeat all the above steps to prove \eqref{eq:liminf_local_mu_3}.
\end{proof}

\begin{proof}[Proof of Theorem \ref{thm:sci}]
Thanks to the density-in-energy of smooth functions given by Lemma \ref{lem:density_energy}, it is enough to prove the lower semicontinuity of $\Jj_\mu$ on a sequence of functions $u_{j}\in \cS(\phi)\cap C^{\infty}(\Omega)$, such that $u_{j}\to u$ in $L^{1}(\Omega)$ as $j\to+\infty$. 
Due to the coercivity of $\Jj_\mu$ (Section \ref{sec:coercivity}), we may assume without loss of generality that $u\in \cS(\phi)$, that the sequence $(u_{j})_j$ is bounded in $BV(\Omega)$ and that $Du_j \weakto Du$ in $\Mm(\Omega; \R^n)$, up to passing to a subsequence.

We know that $\mu = \mu_1 + \mu_2 + \mu_3$, where the measures $\mu_j$ are as in Definition \ref{def:tMomega}, with the additional requirements that $\|T_0\|_{L^\infty(\Omega; \R^n)} \le 1$ and that $\mu_2, \mu_3$ are non-extremal. Since $\mu_1 = h \Leb{n}$ for some $h \in L^q(\Omega)$ with $q > n$, we see that
\begin{equation*}
\lim_{j \to + \infty} \int_\Omega u_j \, d \mu_1 = \lim_{j \to + \infty} \int_{\Omega} u_j \, h \, dx = \int_{\Omega} u \, h \, dx =  \int_\Omega u \, d \mu_1,
\end{equation*}
thanks to the Sobolev embedding of $BV(\Omega)$ into $L^{\frac{n}{n-1}}(\Omega)$ (recalling that $\Omega$ is a bounded open set with Lipschitz boundary) and elementary interpolation inequalities. Then, we choose two open sets $U, V$ with Lipschitz boundary such that 
$$U, V \Subset \Omega, \ {\rm dist}(U, V) > 0, \ S \Subset U, \ \Gamma \Subset V \text{ and } \sqrt{1 + |Du|^2}(\partial U \cup \partial V) = 0.$$ 
This is possible since $S$ and $\Gamma$ are compact sets with positive distance and $\sqrt{1 + |Du|^2}$ is a Radon measure.
We notice that, since $u_j$ is smooth, we have also $\sqrt{1 + |Du_j|^2}(\partial U \cup \partial V) = 0$.
Then, by \eqref{eq:liminf_local_mu_2} we get
\begin{equation*}
\liminf_{j \to + \infty} \sqrt{1 + |Du_j|^2}(U) + \int_{\Omega} u_j \, d \mu_2 \ge \sqrt{1 + |Du|^2}(U) + \int_{\Omega} \tilde{u} \, d \mu_2
\end{equation*}
and, by \eqref{eq:liminf_local_mu_3}, we obtain
\begin{equation*}
\liminf_{j \to + \infty} \sqrt{1 + |Du_j|^2}(V) + \int_{\Omega} u_j \, d \mu_3 \ge \sqrt{1 + |Du|^2}(V) + \int_{\Omega} u^\lambda \, d \mu_3.
\end{equation*}
All in all, the superadditivity of the liminf and the lower semicontinuity of the area functional imply
\begin{align*}
\liminf_{j \to + \infty} \Jj_\mu[u_j] & \ge \liminf_{j \to + \infty} \sqrt{1 + |Du_j|^2}(B \setminus ( \overline{U} \cup \overline{V})) + \lim_{j \to + \infty} \int_\Omega u_j \, d \mu_1 \\
& \quad + \liminf_{j \to + \infty} \left ( \sqrt{1 + |Du_j|^2}(U) + \int_{\Omega} u_j \, d \mu_2 \right ) \\
& \quad + \liminf_{j \to + \infty} \left ( \sqrt{1 + |Du_j|^2}(V) + \int_{\Omega} u_j \, d \mu_3 \right ) \\
& \ge \sqrt{1 + |Du|^2}(B \setminus ( \overline{U} \cup \overline{V})) +  \int_\Omega u \, d \mu_1 + \sqrt{1 + |Du|^2}(U) + \int_{\Omega} \tilde{u} \, d \mu_2  \\
& \quad + \sqrt{1 + |Du|^2}(V) + \int_{\Omega} u^\lambda \, d \mu_3 \\
& = \sqrt{1 + |Du|^2}(B) + \int_{\Omega} u^\lambda \, d \mu = \Jj_\mu[u].
\end{align*}
This ends the proof.
\end{proof}

We dispose now of all the necessary tools for granting the existence of a minimizer of $\Jj_{\mu}$ on $\cS(\phi)$. 

\begin{theorem}[Existence of minimizers]\label{thm:existence}
Let $\mu\in \tMm(\Omega)$ be a non-extremal measure, and let $\mu = \mu_1 + \mu_2 + \mu_3$ be the decomposition given in Definition \ref{def:tMomega}. In addition, assume that $\|T_0\|_{L^\infty(\Omega; \R^n)} \le 1$ and that $\mu_2, \mu_3$ satisfy the non-extremality assumption \eqref{eq:subcritical_cond}. Then the functional $\Jj_{\mu}$ admits a minimizer in $\cS(\phi)$.
\end{theorem}
\begin{proof}
We can apply the Direct Method of the Calculus of Variations, owing to the coercivity property shown in Section \ref{sec:coercivity} and to Theorem \ref{thm:sci}. Indeed, note that for any $0<C<+\infty$ the set 
\[
\{u\in \cS(\phi):\ \Jj_{\mu}[u] \le C\}
\] 
is closed in $L^{1}(B)$. Indeed, if $(u_{h})_{h}$ is a sequence of functions in $\cS(\phi)$ such that $\Jj_{\mu}[u_{h}]\le C<+\infty$ for all $h$, and $u_{h}\to u$ in $L^{1}(B)$ (or, equivalently, in $L^{1}(\Omega)$) as $h\to +\infty$, by \eqref{eq:stimaJmu} and the lower semicontinuity of the total variation of the gradient we obtain $u\in \cS(\phi)$ and, by Theorem \ref{thm:sci}, $\Jj_{\mu}[u] \le C$.  

Finally, by Poincar\'e's inequality on $B$ and by compactness in $BV(B)$, we infer that the set $\{u\in \cS(\phi):\ \Jj_{\mu}[u] \le C\}$ is compact, so that there exists a solution to 
\[
\min \{ \Jj_\mu[u] : u \in \cS(\phi) \}.
\]
\end{proof}

\begin{remark} \label{rem:lambda_lambda_mu_eq}
We stress the fact that the choice of $\lambda = \lambda_\mu$ is indeed essential to obtain the existence of the minimum, not only because it is needed to prove the semicontinuity of $\Jj_\mu[\cdot] = \Jj_\mu[\cdot, \lambda_\mu]$ (Theorem \ref{thm:sci}), but also because, thanks to Proposition \ref{prop:relax_lambda_mu}, the infimum of $\Jj_\mu[\cdot, \lambda]$ is attained at some $u \in \cS(\phi)$ if and only if $\lambda(x) = \lambda_\mu(x)$ for $|\mu|$-a.e. $x \in J_u$. Hence, we can select $\lambda = \lambda_\mu$ without loss of generality when dealing with the existence of minimizers from this point onwards.
\end{remark}

\section{Existence of weak solutions of \eqref{eq:PMCM} via convex duality}
\label{sec:duality}

Thanks to Theorem \ref{thm:existence}, we know that, under suitable assumptions on the measure $\mu$, a minimizer $\baru \in BV(B)$ of the functional $\Jj_{\mu}$ exists, and Proposition \ref{prop:relax_lambda_mu} and Remark \ref{rem:lambda_lambda_mu_eq} ensure that a minimizer $u \in \cS(\phi)$ for $\Jj_{\mu}[\cdot, \lambda]$ cannot exist unless $\lambda = \lambda_\mu$ $|\mu|$-a.e. on $J_u$. Therefore, in this section, we fix $\lambda = \lambda_\mu$ and we show that any minimizer $\baru$ is a weak solution to \eqref{eq:PMCM}; that is, there exists a measurable vector field $T$ such that the pair $(\baru,T)$ satisfies the equations in Definition \ref{def:PMCM}. 
Using tools of convex analysis, we will see that $T$ can be determined as an element of the dual of the space $\Mm(\Omega; \R^n)$ of vector-valued Radon measures endowed with the total variation topology and that its action on a suitable subspace of distributional gradients of $BV$ functions is given precisely by the pairing $(T, Du)_{\lambda}$.

\subsection{Preliminaries of convex optimization and duality}
To begin with, we recall some standard notation and facts from convex analysis (with reference to the monography \cite{ekeland1999convex} for most of them). Let $V$ be a locally convex topological vector space, and denote by $V^{*}$ its topological dual. The duality pairing between $v^{*}\in V^{*}$ and $v\in V$ is denoted as $\langle v, v^{*}\rangle$. We say that a function $\Hh:V\to \overline{\R}$ belongs to $\Gamma_{0}(V)$ if there exists a family $\{v^{*}_{\alpha}\}_{\alpha}$ (with $\alpha$ belonging to some index set, that we do not specify) of elements of the dual space $V^{*}$ and a corresponding family of real numbers $\{\beta_{\alpha}\}_{\alpha}$, such that 
\[
\Hh(v) = \sup_{\alpha} \langle v, v^{*}_{\alpha}\rangle + \beta_{\alpha}\,,
\] 
and moreover the image of $\Hh$ is not reduced to the set $\{\pm\infty\}$. To be more specific, any such function $\Hh$ is either constantly $-\infty$ (this is formally the case when the family of pairs $(v_{\alpha}^{*},\beta_{\alpha})$ is empty) or a generic, convex and lower semicontinuous function taking values in $(-\infty,+\infty]$. It is well-known that any $\Hh\in \Gamma_{0}(V)$ is convex and lower semicontinuous. The Legendre transform (or polar) of $\Hh$ is the function $\Hh^{*}:V^{*}\to \overline\R$ defined as
\[
\Hh^{*}(v^{*}) = \sup_{v\in V} \langle v,v^{*}\rangle - \Hh(v)\,.
\]
The subdifferential of $\Hh$ at $v$ is the (possibly empty) set 
\[ \de \Hh(v) = \{ u^{*}\in V^{*} : \Hh(w) \ge \Hh(v) + \langle w-v, u^{*}\rangle \text{ for all } w\in V \} .\]
Associated with $\Hh\in \Gamma_{0}(V)$ we consider the corresponding minimization problem $\cP$ and set
\begin{equation}
\inf \cP := \inf_{v\in V} \Hh(v)\,.
\end{equation}
This is also called the \textit{primal problem}. To obtain a dual formulation of the problem, we need to consider a family of convex and lower semicontinuous perturbations of $\Hh$ defined using an auxiliary space $Y$. More precisely, let $Y$ be a topological vector space and let $Y^{*}$ denote its dual. For more simplicity, the duality pairing between $p\in Y$ and $p^{*}\in Y^{*}$ is again denoted as $\langle p, p^{*}\rangle$. Assume that we have a function $\Phi:V\times Y\to \overline\R$ belonging to $\Gamma_{0}(V\times Y)$, such that $\Hh(v) = \Phi(v,0)$ for all $v\in V$, and consider for any fixed $p\in Y$ the perturbed problem 
\begin{equation}
\inf \cP_{p} := \inf_{v\in V} \Phi(v,p)\,.
\end{equation}
The \textit{dual problem} to $\cP$ is denoted by $\cP^{*}$, and defined as the following maximization problem:
\begin{equation}
\sup \cP^{*} := \sup_{p^{*}\in Y^{*}} -\Phi^{*}(0,p^{*})\,,
\end{equation}
where 
\[
\Phi^{*}(v^{*},p^{*}) = \sup_{(v,p)\in V\times Y} \langle v,v^{*}\rangle + \langle p,p^{*}\rangle - \Phi(v,p)\,.
\]
An immediate property relating the primal and the dual problems is the following inequality:
\[
\sup \cP^{*} \le \inf \cP\,,
\]
which follows immediately from
\[
-\Phi^{*}(0,p^{*}) = -\sup_{v,p} \langle p, p^{*}\rangle - \Phi(v,p) = \inf_{v,p} \Phi(v,p) - \langle p, p^{*}\rangle \le \inf_{v} \Phi(v,0) = \inf \cP\,.
\]
We set
\[
\kappa(p) = \inf_{v\in V} \Phi(v,p)\,.
\]
We remark that $\kappa$ is convex (see \cite{ekeland1999convex}). We now give a couple of definitions and related propositions (see, respectively, \cite[Propositions III-2.1, 2.2, 2.3 and 4.1]{ekeland1999convex}).

\begin{definition}[Normal Problem]
The primal problem $\cP$ is called normal if $\kappa(0)$ is finite and $\kappa$ is lower semicontinuous at $0$. Similarly, the dual problem $\cP^{*}$ is called normal if $\kappa^{*}(0)$ is finite and $\kappa^{*}$ is lower-semicontinuous at $0$ (here $\kappa^{*}$ denotes the Legendre transform of $\kappa$). 
\end{definition}

\begin{proposition}
The following are equivalent:
\begin{itemize}
\item[(i)] $\cP$ is normal;
\item[(ii)] $\cP^{*}$ is normal;
\item[(iii)] $\inf \cP = \sup \cP^{*} \in \R$. 
\end{itemize}
\end{proposition}

\begin{definition}[Stable problem]
The primal problem $\cP$ is stable if $\kappa(0)$ is finite and $\de\kappa(0) \neq \emptyset$.
\end{definition}
\begin{proposition}
The following are equivalent:
\begin{itemize}
\item[(i)] $\cP$ is stable;
\item[(ii)] $\cP$ is normal and $\cP^{*}$ admits a solution.
\end{itemize}
\end{proposition}

The next proposition provides a useful stability criterion for the primal problem $\cP$. 
\begin{proposition} \label{prop:ET2.3}
Let $\Phi(v,p)$ be convex, $\inf \cP = \kappa(0)$ be finite, and $\kappa(p)$ be continuous at $p=0$. Then $\cP$ is stable. 
\end{proposition}

Now we assume that $\Hh(v) = \Ii(v,\Lambda v)$, where $\Ii : V\times Y\to \overline\R$ and $\Lambda :V\to Y$ is linear and continuous. Let also $\Lambda^{*}:Y^{*}\to V^{*}$ denote the adjoint of $\Lambda$, defined via the property $\langle v, \Lambda^{*}p^{*}\rangle = \langle \Lambda v, p^{*}\rangle$ for $v\in V$ and $p^{*}\in Y^{*}$. Clearly the primal problem $\cP$ corresponds to
\[
\inf_{v\in V} \Ii(v,\Lambda v)\,.
\]
We perturb the primal problem by setting
\[
\Phi(v,p) = \Ii(v,\Lambda v - p)
\]
so that we obtain
\begin{align*}
\Phi^{*}(0,p^{*}) &= \sup_{p} \langle p,p^{*}\rangle - \inf_{v} \Ii(v,\Lambda v - p)\\
&= \sup_{q} \sup_{v} \langle \Lambda v - q,p^{*}\rangle - \Ii(v,q)\\
&= \sup_{v,q} \langle v, \Lambda^{*}p^{*}\rangle + \langle q, -p^{*}\rangle - \Ii(v,q)\\
&= \Ii^{*}(\Lambda^{*}p^{*},-p^{*})\,.
\end{align*}
In conclusion, the dual problem $\cP^{*}$ in this specific case is given by
\[
\sup_{p^{*}\in Y^{*}} - \Ii^{*}(\Lambda^{*}p^{*}, - p^{*})\,.
\]
Next, we state the key result that allows us to characterize the solutions of the primal and the dual problems.
\begin{proposition}\label{prop:ET4.1}
The following are equivalent:
\begin{itemize}
\item[(i)] $\baru$ solves $\cP$, $\barp^{*}$ solves $\cP^{*}$, and $\sup \cP^{*} = \inf \cP$;

\item[(ii)] $\baru\in V$ and $\barp^{*}\in Y^{*}$ verify the extremality relation
\begin{equation}\label{eq:extremality0}
\Ii(\baru, \Lambda \baru) + \Ii^{*}(\Lambda^{*}\barp^{*}, -\barp^{*}) = 0\,,
\end{equation}
that is to say, $(\Lambda^{*}\barp^{*},-\barp^{*}) \in \de \Ii(\baru, \Lambda\baru)$.
\end{itemize}
\end{proposition}

We now further specialise this setting to the case $\Hh(v) = \Ii(v,\Lambda v) = \Ff(v) + \Gg(\Lambda v)$, for some functions $\Ff : V \to \overline{\R}$ and $\Gg : Y \to \overline{\R}$. In this case, assuming all the necessary convexity and lower semicontinuity properties of the functions $\Ff$ and $\Gg$, we have that $\cP^{*}$ corresponds to 
\[
\sup_{p^{*}\in Y^{*}} - \Ff^{*}(\Lambda^{*}p^{*}) - \Gg^{*}(-p^{*}).
\]
Moreover the extremality relation \eqref{eq:extremality0} is equivalent to the two relations
\begin{equation}\label{eq:extremality1}
\Ff(\baru) + \Ff^{*}(\Lambda^{*}\barp^{*}) = \langle \baru, \Lambda^{*}\barp^{*}\rangle
\end{equation}
and
\begin{equation}\label{eq:extremality2}
\Gg(\Lambda \baru) + \Gg^{*}(-\barp^{*}) = -\langle \baru, \Lambda^{*}\barp^{*}\rangle\,.
\end{equation}
These last two relations correspond to $\Lambda^{*}\barp^{*}\in \de \Ff(\baru)$ and $-\barp^{*}\in \de \Gg(\Lambda\baru)$, respectively.

\subsection{Optimality relations}

To recover the weak formulation we let $\mu \in \tMm(\Omega)$ and consider a minimizer $\baru \in BV(B)$ of $\Jj_{\mu}$. Then, we recall that the jump part of the distributional gradient of $\baru$ is of the form $(\baru^{+}-\baru^{-})\nu_{\baru}\ \Haus{n-1}\restrict J_{\baru}$. We shall define the functional $\Ff$ (see \eqref{eq:F_functional}) in such a way that it is finite on functions $u\in BV(\Omega)$ such that the trace of $u$ on $\de \Omega$ equals that of $\baru$, and $\nu_{u} = \nu_{\baru}$ $\Haus{n-1}$-almost everywhere on $J_{u}\cap \Gamma$, where $\Gamma$ is a compact, Lebesgue negligible set on which the measure $\mu_3$ from the decomposition of $\mu$ is concentrated, see Definition \ref{def:tMomega}.
Then we let
\[
V = BV(\Omega)
\]
endowed with the strong $BV$ topology. We then let $Y=\Mm(\Omega; \R^n)$ endowed with the total variation topology. On noting that the standard embedding $\Phi:L^{1}(\Omega;\R^{n})\mapsto Y$ induces a surjective restriction operator $R:Y^{*}\mapsto L^{\infty}(\Omega;\R^{n})$ defined as $R(y^{*})(f) = \langle y^{*},\Phi(f)\rangle$, from now on we shall identify $Y^{*}$ with $R(Y^{*})$ with a slight abuse of notation. 

We define a particular set $\tV \subset V$ by setting
\begin{align} 
\tV := \{ u\in V :  {\rm Tr}_{\partial \Omega}(u) = {\rm Tr}_{\partial \Omega}(\baru) \text{ in } L^1(\partial \Omega; \Haus{n-1}), & \ \Haus{n-1}(J_{u}\setminus J_{\baru}) = 0, \nonumber \\
& \nu_{u} = \nu_{\baru} \ \Haus{n-1}\text{-a.e. on } J_{u}\cap \Gamma  \}. \label{def:tV}
\end{align}
We notice that this definition can be relaxed by assuming, instead of the last two conditions, that the vector field $\nu_{\baru}$ admits an extension $\mathcal{N}$ to the whole $\Gamma$, so that one requires $\nu_{u} = \mathcal{N}$ $\Haus{n-1}$-almost everywhere on $J_u \cap \Gamma$.

We show now that $\tV$ is a closed and convex subset of $V$.

\begin{lemma}
$\tV$ is a closed and convex subset of $V$.
\end{lemma}

\begin{proof}
Convexity of $\tV$ follows immediately from the definition.
Let $(u_k)_{k \in \N} \subset \tV$ be such that $u_k \to u$ with respect to the strong topology of $BV(\Omega)$. Clearly, $u \in BV(\Omega)$, so we need to check the conditions in the definition of $\tV$, \eqref{def:tV}. By the continuity of the trace operator (we refer for instance to \cite[Theorems 3.87 and 3.88]{AFP}), we have
\begin{equation*}
\|{\rm Tr}_{\partial \Omega}(u_k) - {\rm Tr}_{\partial \Omega}(u)\|_{L^1(\partial \Omega; \Haus{n-1})} \le C \|u_k - u \|_{BV(\Omega)} \to 0 \ \text{ as } \ j \to + \infty.
\end{equation*}
However, by definition of $\tV$ \eqref{def:tV}, we have ${\rm Tr}_{\partial \Omega}(u_k) = {\rm Tr}_{\partial \Omega}(\baru)$ for all $k \in \N$, and so we conclude that ${\rm Tr}_{\partial \Omega}(u) = {\rm Tr}_{\partial \Omega}(\baru)$ in $L^1(\partial \Omega; \Haus{n-1})$. Then, the strong convergence implies that 
\begin{equation} \label{eq:strong_conv_jump}
| D^j u_k - D^j u |(B) \to 0 \ \text{ as } \ k \to + \infty
\end{equation}
for all Borel sets $B \subseteq \Omega$. In particular, by choosing $B = \Omega \setminus J_{\baru}$, we get
\begin{align*}
\int_{J_u \setminus J_{\baru}} |u^+ - u^-| \, d \Haus{n-1} & = |D^j u|(\Omega \setminus J_{\baru}) = \lim_{k \to + \infty} |D^j u_k|(\Omega \setminus J_{\baru}) \\
& = \lim_{k \to + \infty} \int_{J_{u_k} \setminus J_{\baru}} |u^+_k - u^-_k| \, d \Haus{n-1} = 0 
\end{align*}
by \eqref{def:tV}. This implies that either $\Haus{n-1}(J_{u}\setminus J_{\baru}) = 0$ or $|u^+ - u^-| = 0$ $\Haus{n-1}$-a.e. on $J_u \setminus J_{\baru}$. However, if $u^+(x) = u^-(x)$, then $x \notin J_u$, so that the second case is not possible. Finally, we need to prove that $\nu_{u} = \nu_{\baru} \ \Haus{n-1}$-a.e. on $J_{u}\cap \Gamma$. Without loss of generality, we may assume $\Haus{n-1}(J_u \cap \Gamma) > 0$. By applying \eqref{eq:strong_conv_jump} to $B = J_u \cap \Gamma$ we get
\begin{align*}
| D^j u_k - D^j u |(J_u \cap \Gamma) & = \int_{J_u \cap \Gamma} \left | (u^+_k - u^-_k) \nu_{u_k} - (u^+ - u^-) \nu_u \right | \, d \Haus{n -1} \\
& = \int_{J_u \cap \Gamma} \left | (u^+_k - u^-_k) \nu_{\baru} - (u^+ - u^-) \nu_u \right | \, d \Haus{n -1} \to 0,
\end{align*}
which is possible if and only if $\nu_{u} = \nu_{\baru} \ \Haus{n-1}$-a.e. on $J_{u}\cap \Gamma$. Indeed, one can argue by contradiction assuming that there exist $\delta>0$ and a set $K\subset J_{u}\cap \Gamma$ with $\Haus{n-1}(K)>0$, such that $u^{+}(x)-u^{-}(x)> \delta$ and $|\nu_{u}(x)-\nu_{\baru}(x)|>\delta$ whenever $x\in K$. Up to extracting a subsequence and replacing $K$ with an $\Haus{n-1}$-equivalent set, we can assume that for all $x\in K$, 
\begin{equation}\label{eq:convnormali}
f_{k}(x)\nu_{\baru}(x) \to f(x)\nu_{u}(x)\qquad \text{as $k\to +\infty$,}
\end{equation}
where $f_{k}(x) = u_{k}^{+}(x) - u_{k}^{-}(x)$ and $f(x) = u^{+}(x) - u^{-}(x)$. Since $f_{k}\ge 0$ and $f>\delta$ on $K$, by \eqref{eq:convnormali} one infers that $f_{k}(x)\to f(x)$ for all $x\in K$, and consequently that $|\nu_{\baru}(x) - \nu_{u}(x)|=0$ for all $x\in K$, a contradiction. This ends the proof.
\end{proof}

We define $\Ff:V\to \overline\R$ as
\begin{equation} \label{eq:F_functional}
\Ff(u) = \begin{cases}
\ds\int_\Omega u^{\lambda}\, d\mu & \text{if }u\in \tV\\
+\infty & \text{otherwise.}
\end{cases}
\end{equation}
and $\Gg:Y\to \overline\R$ as
\begin{align*}
\Gg(p) & = \sqrt{1+|p|^{2}}(\Omega) \\
&= \sup \left \{ \int_{\Omega} g_{0} \, dx + \int_{\Omega}g\cdot dp \, : (g_{0},g)\in C^{0}_{c}(\Omega)\times C^{0}_{c}(\Omega;\R^{n}) \text{ such that } g_{0}^{2} + |g|^{2}\le 1 \text{ on } \Omega \right \}.
\end{align*}

\begin{proposition}
Under the previous assumptions, $\Ff\in \Gamma_{0}(V)$. 
\end{proposition}

\begin{proof}
We only need to prove that $\Ff$ is not identically $\pm\infty$, is convex, and is lower semicontinuous in $V$. The first property is an obvious consequence of Proposition \ref{prop:trunc-smooth-coarea}. As for the convexity, it is clearly enough to prove it on $\tV$, and we notice that by Definition \ref{def:tMomega} and Theorem \ref{thm:liminf_mu_2}, for all $u \in \tV$ we have 
\begin{equation*}
\int_\Omega u^{\lambda}\, d\mu = \int_\Omega u \, d\mu_1 + \int_\Omega \tilde u \, d\mu_2 + \int_\Omega u^{\lambda}\, d\mu_3.
\end{equation*}
Hence, we need only consider the last integral, given that the first two are linear in $u$. Let now $t_1, t_2 \in \R$ and $u_{1},u_{2}\in \tV$. Defining $u_{t} = t u_{1} + (1-t) u_{2}$ we have $u_{t}^{\pm} = t u_{1}^{\pm} + (1-t) u_{2}^{\pm}$ $|\mu_3|$-almost everywhere on $\Gamma$ because $\nu_{u_{1}} = \nu_{u_{2}}$ $|\mu_3|$-almost everywhere on $\Gamma\cap J_{u_{1}}\cap J_{u_{2}}$, thanks to \eqref{def:tV} and the fact that $\mu_3 \in \MH(\Omega)$. Therefore we obtain
\[
\int_\Omega u_{t}^{\lambda}\, d\mu_3 = t \int_\Omega u_{1}^{\lambda}\, d\mu_3 + (1-t) \int_\Omega u_{2}^{\lambda}\, d\mu_3\,,
\]
which implies that $\Ff$ is affine on $\tV$ and shows, in particular, the convexity of $\Ff$. Finally, the lower semicontinuity is a consequence of the continuity of the integral $\displaystyle \int_\Omega u^{\lambda}\, d\mu$ with respect to convergence in $BV$ norm (see Lemma \ref{lem:2a2b-bis}).
\end{proof}

Finally, we let $\Lambda = D$ be the distributional derivative operator, which is linear and continuous from $V$ to $Y$. We then have for every $u\in \widetilde{V}$
\[
\Jj_{\mu}[u] = \Gg(Du) + \Ff(u)
\]
so that $\Phi(u,p) = \Gg(Du-p) + \Ff(u)$ belongs to $\Gamma_{0}(V\times Y)$. Moreover, we have the following result.

\begin{proposition}\label{prop:kappalsc}
Let $\kappa(p) = \inf_{u\in V} \Phi(u,p)$, then $\kappa(0)$ is finite, and $\kappa$ is continuous at $0$.
\end{proposition}

\begin{proof}
The first claim is obvious, as $\kappa(0) = \Jj_{\mu}[\baru]$. The lower semicontinuity at $0$ can be proved as follows. Since $Y$ is a normed space, we can consider a sequence $(p_{k})_{k}\subset Y$ such that $|p_{k}|(\Omega)\to 0$ as $k\to +\infty$. By definition, there exists a sequence $(u_{k})_{k}\subset V$ such that 
\[
\kappa(p_{k})\ge \sqrt{1+|Du_{k}-p_{k}|^{2}}(\Omega) + \int_\Omega u_{k}^{\lambda}\, d\mu - \frac{1}{k}\,.
\]
Therefore, if we set 
\begin{equation*}
\mathcal{C}(\Omega) = \left \{ (g_{0},g)\in C^{0}_{c}(\Omega)\times C^{0}_{c}(\Omega;\R^{n}) \text{ such that } g_{0}^{2} + |g|^{2}\le 1 \text{ on } \Omega \right \},
\end{equation*}
then we have
\begin{align*}
\kappa(p_{k}) &\ge \sup_{(g_{0},g) \in \mathcal{C}(\Omega)} \int_\Omega g_{0}\, dx + \int_\Omega g\cdot d D u_{k} - \int_\Omega g\cdot dp_{k} + \int_\Omega u_{k}^{\lambda}\, d\mu - \frac{1}{k}\\
&\ge \sup_{(g_{0},g) \in \mathcal{C}(\Omega)} \int_\Omega g_{0}\, dx + \int_\Omega g\cdot d D u_{k} + \int_\Omega u_{k}^{\lambda}\, d\mu - |p_{k}|(\Omega) - \frac{1}{k}\\
&= \sqrt{1+|Du_{k}|^{2}}(\Omega) + \int_\Omega u_{k}^{\lambda}\, d\mu -  |p_{k}|(\Omega) - \frac{1}{k}\\
&\ge \kappa(0) - |p_{k}|(\Omega) - \frac{1}{k}\,,
\end{align*}
which implies that $\displaystyle \liminf_{k \to + \infty} \kappa(p_{k}) \ge \kappa(0)$. Finally, the upper semicontinuity follows from
\[
\kappa(p) \le \sup_{(g_{0},g) \in \mathcal{C}(\Omega)} \int_\Omega g_{0}\, dx + \int_\Omega g\cdot d D \baru - \int_\Omega g\cdot dp + \int_\Omega \baru^\lambda\, d\mu \le \kappa(0) + |p|(\Omega)\,,
\]
and this proves the last claim.
\end{proof}

We can now state and prove the main result of this section. 

\begin{theorem} \label{thm:mininimizer_solution_implication_1}
Let $\mu \in \tMm(\Omega)$ and $\baru$ be a minimizer of $\Jj_{\mu}$ on $\cS(\phi)$. Then, there exists a vector field $T \in \DM^{\infty}(\Omega)$ such that \eqref{L_infty_bound_eq} and \eqref{divergence_eq} hold true, and the pair $(\baru,T)$ satisfies \eqref{pairing_eq} for $\lambda = \lambda_\mu$.
\end{theorem}
\begin{proof}
By combining Proposition \ref{prop:kappalsc} with Proposition \ref{prop:ET2.3} we infer that the primal problem $\cP$ is stable. Hence the dual problem $\cP^{*}$ admits a solution $\barp^{*}$, therefore we can apply Proposition \ref{prop:ET4.1} and deduce the necessary relations \eqref{eq:extremality1} and \eqref{eq:extremality2}. To write them explicitly in the present situation, we must compute $\Ff^{*}(D^{*}q^{*})$ and $\Gg^{*}(q^{*})$ for $q^{*}\in Y^{*}$. We conveniently consider a function $\psi \in W^{1,1}(\Omega)$ such that\footnote{For instance, we can choose $\psi$ as any Anzellotti-Giaquinta approximation of $\baru$. Thanks to Theorem \cite[Theorem 3.1]{ComiLeo}, we know that $\psi \in C^{\infty}(\Omega) \cap BV(\Omega)$ and $\Tr_{\de \Omega}(\psi) = \Tr_{\de \Omega}(\baru)$.} its trace on $\de \Omega$ coincides with the trace of $\baru$, and set $V_{0} = -\psi + \tV \subset BV_{0}(\Omega)$. We have
\begin{align*}
\Ff^{*}(D^{*}q^{*}) &= \sup_{u\in V} \langle Du,q^{*}\rangle - \Ff(u)\\
&= \langle D\psi,q^{*}\rangle - \int_\Omega \tilde \psi \, d\mu + \sup_{u\in V_{0}} \langle Du,q^{*}\rangle - \int_\Omega u^{\lambda}\, d\mu\\
&= \begin{cases}
\langle D\psi,q^{*}\rangle - \int_\Omega \tilde \psi \, d\mu & \text{if }\langle Du, q^{*}\rangle \le \int_\Omega u^{\lambda}\, d\mu\quad \text{for all }u\in V_{0}, \\[5pt]
+\infty & \text{otherwise.}
\end{cases}
\end{align*}
Therefore, in order to have $\Ff^{*}(D^{*}q^{*})$ finite, we must have 
\begin{equation} \label{ineq:div_distr}
\langle Du, q^{*}\rangle \le \int_\Omega u^{\lambda}\, d\mu
\end{equation}
for all $u \in V_0$.
Then we notice that, since $C^{\infty}_{c}(\Omega)\subset V_{0}$, \eqref{ineq:div_distr} holds in particular for all $u \in C^\infty_c(\Omega)$. However, for $u \in C(\Omega)$, we have $u^{\lambda} = u$, and, by considering both $u$ and $-u$ as test functions, \eqref{ineq:div_distr} yields
\begin{equation*}
\int_{\Omega} q^{*} \cdot \nabla u \, dx \le \int_\Omega u\, d\mu \ \text{ and } \ \int_{\Omega} q^{*} \cdot \nabla u \, dx \ge \int_\Omega u\, d\mu,
\end{equation*}
which readily implies $-\div q^{*}= \mu$ on $\Omega$. On the other hand,
\begin{align*}
\Gg^{*}(q^{*}) &= \sup_{q\in Y} \langle q, q^{*}\rangle - \sqrt{1+|q|^{2}}(\Omega)\\
&= \sup_{q\in Y} \langle q_{ac}, q^{*}\rangle - |\Leb{n}\times q_{ac}|(\Omega) + \langle q_{s}, q^{*}\rangle - |q_{s}|(\Omega)\,,
\end{align*}
where $q_{ac}$ and $q_{s}$ are, respectively, the absolutely continuous and the singular parts of the Radon measure $q$, and $\Leb{n}\times q_{ac}$ is the $\R^{n+1}$-valued measure with first component $\Leb{n}$ and other components $q_{ac}$. To proceed with our calculation, we set 
$$\|q^{*}\|_{\perp} := \sup \left \{\langle \nu, q^{*}\rangle : \nu \in Y^\perp, |\nu|(\Omega) \le 1 \right \},$$
where
$$Y^\perp := \left \{ \nu \in Y:  \nu \perp \Leb{n} \right \}.$$
Then, by  orthogonality of the measures $q_{ac}$ and $q_{s}$, we obtain
\begin{align*}
\Gg^{*}(q^{*}) &= \sup_{q\in Y} \langle q_{ac}, q^{*}\rangle - |\Leb{n}\times q_{ac}|(\Omega) + \langle q_{s}, q^{*}\rangle - |q_{s}|(\Omega)\\
&= \sup_{q\in L^{1}(\Omega;\R^{n})} \left(\int_\Omega q\cdot q^{*}\, dx - \int_{\Omega}\sqrt{1+|q|^{2}} \, dx \right) + \sup_{\nu \in Y^\perp} \langle \nu, q^{*}\rangle - |\nu|(\Omega)\\
&= 
\begin{cases}
-\displaystyle \int_{\Omega} \sqrt{1-|q^{*}|^{2}}\, dx & \text{if }\|q^{*}\|_{L^{\infty}(\Omega; \R^n)} \le 1 \text{ and }\|q^{*}\|_{\perp} \le 1,\\[5pt]
+\infty & \text{otherwise.}
\end{cases}
\end{align*}
Note that the last step above follows from a pointwise almost everywhere optimization of the quantity
\[
\rho |q^{*}(x)| - \sqrt{1+\rho^{2}}
\]
as $\rho$ varies in $[0,+\infty)$ (see computations in \cite[Chapter V, Lemma 1.1]{ekeland1999convex}). 

At this point, we can write the optimality relations for the dual pair $(\baru, \barp^{*})$. In particular, \eqref{eq:extremality1} becomes
\begin{equation}\label{eq:extremality1.1}
\begin{cases}
\displaystyle \int_\Omega (\baru^{\lambda}-\tilde \psi)\, d\mu + \langle \nabla \psi, \barp^{*}\rangle = \langle D\baru, \barp^{*}\rangle, & \\[5pt]
-\div \barp^{*} = \mu \ \text{ on } \Omega, & 
\end{cases}
\end{equation}
and \eqref{eq:extremality2} becomes
\begin{equation}\label{eq:extremality2.1}
\begin{cases}
\sqrt{1+|D\baru|^{2}}(\Omega) = \displaystyle \int_{\Omega}\sqrt{1-|\barp^{*}|^{2}} \,dx - \langle D\baru, \barp^{*}\rangle, & \\[5pt]
\|\barp^{*}\|_{L^\infty(\Omega; \R^n)} \le 1 \text{ and }\|\barp^{*}\|_{\perp} \le 1. & 
\end{cases}
\end{equation}
Let now $T = -\barp^{*}$. In particular we have $\|T\|_{L^{\infty}(\Omega; \R^n)}\le 1$, $\div T = \mu$ on $\Omega$, and by \eqref{eq:extremality1.1} we deduce
\begin{align*}
\langle D\baru, T\rangle &= \int_{\Omega} T\cdot \nabla \psi \, dx + \int_\Omega \tilde \psi \, d\mu - \int_\Omega \baru^{\lambda} \, d\mu\\
&= \int_{\Omega} d \div(\psi T) - \int_\Omega \baru^{\lambda}\, d\mu\\
&= \int_{\Omega} d \div(\baru T) - \int_\Omega \baru^{\lambda}\, d\mu - \int_{\Omega} d \div((\baru - \psi)T)\\
&= \int_{\Omega} d \div(\baru T) - \int_\Omega \baru^{\lambda}\, d\mu\\
&= (T, D\baru)_{\lambda}(\Omega)
\end{align*}
because $\baru - \psi$ has zero trace on $\de\Omega$, hence $\div((\baru-\psi)T)(\Omega) = 0$, thanks to the integration by parts formula \eqref{eq:GG_boundary_domain_u}, which we can apply since $\div T = \mu$ is an admissible measure, in the light of Remark \ref{rem:div_admissible_Leibniz}. By plugging this last identity into \eqref{eq:extremality2.1} we get
\begin{equation}\label{eq:optimalpairing}
(T, D\baru)_{\lambda}(\Omega) = \sqrt{1+|D\baru|^{2}} (\Omega) - \int_{\Omega}\sqrt{1-|T|^{2}}\, dx\,.
\end{equation}
On the other hand, by \eqref{eq:pairlambda-vs-area} in the case $a = 1$ we know that 
\begin{equation}\label{eq:pairinguppest}
| (T, D\baru)_{\lambda}| \le \sqrt{1+|D\baru|^{2}} - \sqrt{1-|T|^{2}}\, \Leb{n} \ \text{ on } \Omega.
\end{equation}
Finally, by combining \eqref{eq:optimalpairing} with \eqref{eq:pairinguppest} we obtain that
\[
 (T, D\baru)_{\lambda} = \sqrt{1+|D\baru|^{2}} - \sqrt{1-|T|^{2}}\, \Leb{n} \ \text{ on } \Omega.
\]
This proves that $\baru$ and $T$ verify \eqref{pairing_eq}, as wanted.
\end{proof}

We notice that, if for a given $u \in \cS(\phi)$ there exists $T \in \DM^{\infty}(\Omega)$ such that \eqref{L_infty_bound_eq},  \eqref{divergence_eq} and \eqref{pairing_eq} hold, then we can give an alternative representation of $\Jj_\mu[u]$ in terms of the trace of $u$ and of the interior normal trace of $T$ on $\partial \Omega$.

\begin{theorem}\label{thm:Jsol}
Let $\mu \in \MH(\Omega)$ be an admissible measure.
Let $u \in \cS(\phi)$ and $T \in \DM^{\infty}(\Omega)$ be such that \eqref{L_infty_bound_eq} and \eqref{divergence_eq} hold true. Then we have
\begin{align}
\Jj_\mu[u] & \ge \int_{B \setminus \Omega} \sqrt{1 + |\nabla \phi|^2} \, dx + \int_{\partial \Omega} \left ( |{\rm Tr}_{\partial \Omega}(u - \phi)| - {\rm Tr}_{\partial \Omega}(u) {\rm Tr}^i(T, \partial \Omega) \right ) \, d \Haus{n-1} \nonumber \\
& \quad + \int_{\Omega} (1 - \theta) \, d |(T, Du)_\lambda| + \int_{\Omega} \sqrt{1 - |T|^2} \, dx, \label{eq:lower_bound_T_sol}
\end{align}
where $\theta$ is a Borel function satisfying $(T, Du)_\lambda = \theta |(T, Du)_\lambda|$ on $\Omega$.
In addition, if the couple $(u, T)$ satisfies \eqref{pairing_eq}, then we get
\begin{equation} \label{eq:min_value}
\Jj_\mu[u] = \int_{B \setminus \Omega} \sqrt{1 + |\nabla \phi|^2} \, dx + \int_{\partial \Omega} \left ( |{\rm Tr}_{\partial \Omega}(u - \phi)| - {\rm Tr}_{\partial \Omega}(u) {\rm Tr}^i(T, \partial \Omega) \right ) \, d \Haus{n-1} + \int_{\Omega} \sqrt{1 - |T|^2} \, dx.
\end{equation}
\end{theorem}

\begin{proof}
Since $\div T = \mu$ and $|T(x)| \le 1$ for $\Leb{n}$-a.e. $x \in \Omega$, by Remark \ref{rem:div_admissible_Leibniz} we can apply \eqref{eq:Leibniz_lambda_gen}, \eqref{eq:pairlambda-vs-area} with $a =1$, and \eqref{eq:GG_boundary_domain_u} to obtain
\begin{align*}
\sqrt{1 + |Du|^2}(\Omega) + \int_{\Omega} u^\lambda \, d \div T & \ge |(T, Du)_{\lambda}|(\Omega) + \int_{\Omega} \sqrt{1 - |T|^2} \, dx + \div(uT)(\Omega) - (T, Du)_{\lambda}(\Omega) \\
& = \int_{\Omega} (1 - \theta) \, d |(T, Du)_\lambda| + \int_{\Omega} \sqrt{1 - |T|^2} \, dx \\
& \quad - \int_{\partial \Omega} {\rm Tr}_{\partial \Omega}(u) {\rm Tr}^i(T, \partial \Omega) \, d \Haus{n-1}.
\end{align*}
Hence, we get \eqref{eq:lower_bound_T_sol} by the definition of $\Jj_\mu$. Then, if $(u, T)$ satisfies \eqref{pairing_eq}, we argue similarly as above to get
\begin{align*}
\sqrt{1 + |Du|^2}(\Omega) + \int_{\Omega} u^\lambda \, d \div T & = (T, Du)_{\lambda}(\Omega) + \int_{\Omega} \sqrt{1 - |T|^2} \, dx + \div(uT)(\Omega) - (T, Du)_{\lambda}(\Omega) \\
& = \int_{\Omega} \sqrt{1 - |T|^2} \, dx - \int_{\partial \Omega} {\rm Tr}_{\partial \Omega}(u) {\rm Tr}^i(T, \partial \Omega) \, d \Haus{n-1}.
\end{align*}
This ends the proof.
\end{proof}

\medskip

\section{Variational approximation}\label{sec:gamma}
Here we establish a Gamma-convergence result for the sequence of functionals
\begin{equation}\label{eq:Jmuj}
\Jj_{\mu_{j}}[u] = \begin{cases}
\sqrt{1+|Du|^2}(B) + \ds\int_\Omega u(x)\, \mu_{j}(x)\, dx & \text{if }u\in \cS(\phi)\,,\\
+\infty & \text{otherwise,}
\end{cases}
\end{equation}
to the functional $\Jj_{\mu}$ defined in \eqref{eq:Jmu}. Here $\mu$ is an admissible non-extremal measure, that is, $|\mu| \in BV(\Omega)^{*}$ and it satisfies \eqref{eq:subcritical_cond}, and $\mu_{j}\in C^{\infty}(\Omega)$ is such that $\mu_{j}\, \Leb{n}$ is non-extremal for all $j$, and converges to $\mu$ in the weak--$\ast$ sense, as $j\to+\infty$.
\medskip

The first step consists in proving that any admissible non-extremal measure $\mu$ can be approximated in the weak--$\ast$ sense by a sequence of absolutely continuous measures $\mu_j \Leb{n}$, which are admissible and uniformly non-extremal, and such that the density functions $\mu_{j}$ are smooth. This result is an immediate consequence of \cite[Proposition 4.23]{ComiLeo}, and it exploits an approach analogous to the one of \cite{dal1999renormalized}.

Since the next proposition requires regularization via mollification, we preliminarily recall that a (symmetric) mollifier is a function $\rho \in C^{\infty}_{c}(B_{1})$, such that $\rho \ge 0$, $\rho(-x) = \rho(x)$ and $\displaystyle \int_{B_{1}} \rho \, dx = 1$. We also note that, whenever $\mu\in \Mm(\Omega)$ is such that $\mathrm{supp} (\mu) \Subset \Omega$, and $\delta>0$ is small enough, then for any function $u\in L^{1}(\Omega)$ we have 
\[
\int_{\Omega}u(\rho_{\delta}\ast \mu)\, dx = \int_{\Omega} (\rho_{\delta}\ast u) \, d\mu\,,
\]
where $\rho_{\delta}(x) = \delta^{-n}\rho(x/\delta)$.

\begin{proposition}\label{prop:muapprox}
Let $\mu\in \MH(\Omega)$ be an admissible non-extremal measure. Then there exists a sequence of functions $(\mu_{j})_{j} \subset C^{\infty}(\Omega)$ such that $\mu_j \, \Leb{n}$ is admissible and uniformly non-extremal for $j$ large enough, $\mu_j \, \Leb{n} \weakto \mu$ as $j \to + \infty$. Moreover, if $\mu = h\, \Leb{n} + \mu^{s}$ with $h\in L^{q}(\Omega)$, for some $q \in [1, + \infty]$, and $\mu^{s}$ has compact support in $\Omega$, then for any open set $\Omega'\Subset \Omega$ containing the support of $\mu^{s}$ we have the following properties:
\begin{itemize}
\item there exists $(h_{j})_j \subset C^{\infty}(\Omega)$ such that $h_{j}\to h$ in $L^{p}(\Omega)$, for all $p \in [1, q)$ and $p = q$ if $q < + \infty$, and $\mu_{j} - h_{j} \to 0$ in $L^n(\Omega\setminus \Omega')$;
\item $\mu_{j} = \rho_{j}\ast \mu^s + h_j$ on $\Omega'$ for $j$ large enough and $\rho_j(x) = \delta_{j}^{-n} \rho(x/\delta_j)$, for some sequence $\delta_j \downarrow 0$ and some fixed mollifier $\rho \in C^{\infty}_c(B_1)$.
\end{itemize}
\end{proposition}

\begin{proof}
Since $\mu$ is non-extremal, there exists $0 < L < 1$ such that $\mu \in \PB_L(\Omega)$. Hence, by \cite[Proposition 4.23]{ComiLeo}, there exists a sequence of functions $(\mu_{j})_{j} \subset C^{\infty}(\Omega)$ such that $\mu_j \Leb{n}$ is an admissible measure, it belongs to $\PB_{L(1 + 1/j)}(\Omega)$ for all $j \in \N$, and satisfies all the other properties. In particular, this means that $\mu_j \Leb{n}$ is non-extremal as long as $j > \frac{L}{1-L}$, and this ends the proof.
\end{proof}

The following theorem proves the Gamma-convergence of $\Jj_{\mu_{j}}$ to $\Jj_{\mu}$ under suitable assumptions on $\mu$ and $\mu_{j}$. For the essential definitions of Gamma-convergence, we refer to \cite{braides2002gamma}.

\begin{theorem}\label{thm:gammaconv}
Let $\mu\in \tMm(\Omega)$ be a non-extremal measure such that $\mu_2 = 0$ (see Definition \ref{def:tMomega}). Let $\mu_{j}$ be as in the thesis of Proposition \ref{prop:muapprox}, and let $\Jj_{\mu_{j}}$ and $\Jj_{\mu}$ be defined as in \eqref{eq:Jmuj} and \eqref{eq:Jmu}, respectively.  Then $\Jj_{\mu_{j}} \stackrel{\Gamma}{\longrightarrow} \Jj_{\mu}$ with respect to the $L^{1}$ topology, as $j\to+\infty$.
\end{theorem}
\begin{proof}
Let us start from the $\Gamma$-$\limsup$ inequality. Fix $u\in \cS(\phi)$, then using the density-in-energy of smooth functions provided by Lemma \ref{lem:density_energy} we can directly assume $u\in \cS(\phi)\cap C^{\infty}(\Omega)$. Then we fix a cut-off function $\eta\in C^{\infty}_{c}(\Omega)$ which is constantly equal to $1$ on some $\Omega'\Subset \Omega$ such that $\Gamma\Subset \Omega'$, and write
\begin{align}
\Jj_{\mu_{j}}[u] - \Jj_{\mu}[u] = \int_{\Omega} u \, \mu_{j} \, dx - \int_\Omega u \, d\mu = \int_{\Omega} \eta \, u \, \Big(\mu_{j} \, dx - d\mu\Big) + \int_{\Omega}(1-\eta) \, u\, \Big(\mu_{j}\, dx - d\mu\Big)\,. \label{eq:energy_difference_ref}
\end{align}
The first integral tends to zero as $j\to+\infty$ by weak--$\ast$ convergence. Since the singular part of $\mu$ is supported on $\Gamma$, the second integral can be estimated as
\begin{align*}
\left|\int_{\Omega}(1-\eta) \, u \, \Big(\mu_{j} \, dx - d\mu\Big)\right| \le \int_{\Omega\setminus \Omega'} |u|\, |\mu_{j} - h|\, dx \le \int_{\Omega\setminus \Omega'} |u|\, |\mu_{j} - h_{j}|\, dx  + \int_{\Omega\setminus \Omega'} |u|\, |h_{j} - h|\, dx\,,
\end{align*}
where $h_{j}$ is as in Proposition \ref{prop:muapprox}. Given that $\mu \in \tMm(\Omega)$ and $\mu_2 =0$, we know that $h \in L^q(\Omega)$ for some $q > n$, so that Proposition \ref{prop:muapprox} yields $h_{j}\to h$ in $L^{p}(\Omega)$ for some $p \in (n, + \infty)$. In addition, we have $\mu_{j} - h_{j}\to 0$ in $L^n(\Omega\setminus \Omega')$, by Proposition \ref{prop:muapprox}, and so we conclude that also the second integral in \eqref{eq:energy_difference_ref} is infinitesimal as $j\to+\infty$.
This proves that the difference of the energies tends to $0$ as $j\to+\infty$, which gives the $\Gamma$-$\limsup$ inequality at once.

Let us prove the $\Gamma$-$\liminf$ inequality. Fix $u_{j}, u\in \cS(\phi)$ such that 
\begin{equation} \label{eq:gamma_liminf_sequence}
u_{j}\to u \text{ in } L^{1}(\Omega) \text{ as } j \to + \infty \ \text{ and } \sup_{j \in \N} \Jj_{\mu_{j}}[u_{j}]<+\infty.
\end{equation} 
Using Lemma \ref{lem:density_energy} we can assume that $u_{j}\in C^{\infty}(\Omega)$ without loss of generality. With a slight abuse of notation, we denote by $\mu_{j}$ also the measure $\mu_{j}\, \Leb{n}$. The goal is to prove that
\begin{equation} \label{eq:lim_inf_prove}
\liminf_{j \to + \infty}\ \Jj_{\mu_{j}}[u_{j}] - \Jj_{\mu}[u] \ge 0\,.
\end{equation}
By Definition \ref{def:tMomega}, given that $\mu_2 = 0$, we know that the singular part of $\mu$ is concentrated on a compact set $\Gamma \subset \subset \Omega$ such that $|\Gamma| = 0$. For any open set $\Omega'$ such that $\Gamma \subset \Omega' \Subset \Omega$ and $j$ large enough, Proposition \ref{prop:muapprox} ensures that $\mu_j = \rho_j * \mu^s + h_j$ on $\Omega'$, for some mollifier $\rho_j(x) = \delta_{j}^{-n} \rho(x/\delta_j)$ with $\rho \in C^{\infty}_c(B_1)$ and $\delta_j \downarrow 0$. Therefore, we see that
\begin{align*}
\int_\Omega u_j \, \mu_j \, dx & = \int_{\Omega'} u_j \, h_j \, dx + \int_{\Omega'} u_j \, (\rho_j * \mu^s) \, dx + \int_{\Omega \setminus \Omega'} u_j \, \mu_j \, dx \\
& = \int_\Omega u_j \, h_j \, dx + \int_{\Omega \setminus \Omega'} u_j \, (\mu_j - h_j) \, dx + \int_{\Omega} u_j \, (\rho_j * \mu^s) \, dx
\end{align*}
for $j$ large enough such that $\Gamma + B_{\delta_j} \subset \Omega'$.
Given that the coercivity of $\Jj_{\mu}$ is independent of $\mu$ (cfr. Section \ref{sec:coercivity}), by \eqref{eq:gamma_liminf_sequence} we infer that $(u_j)_j$ is a bounded sequence in $BV(\Omega)$. Therefore, by interpolation, we obtain that $u_j \to u$ in $L^r(\Omega)$ for all $r \in \left [1, \frac{n}{n-1} \right )$ and that $(u_j)_j$ is bounded in $L^{\frac{n}{n-1}}(\Omega)$. As recalled above, by Proposition \ref{prop:muapprox} we know that $h_{j}\to h$ in $L^{p}(\Omega)$ for some $p \in (n, + \infty)$ and $\mu_{j} - h_{j}\to 0$ in $L^n(\Omega\setminus \Omega')$, hence we obtain
\begin{equation*}
\lim_{j \to + \infty} \int_\Omega u_j \, h_j \, dx + \int_{\Omega \setminus \Omega'} u_j \, (\mu_j - h_j) \, dx = \int_\Omega u \, h \, dx.
\end{equation*}
Thus, we have to deal only with the singular part of the measure $\mu$, therefore we shall assume that $h = 0$ for the remainder of the proof. 

Now, consider a sequence of functions $(\eta_j)_j$ such that 
\begin{equation} \label{eq:prop_eta_j}
\eta_j \in C^{\infty}_c(\Omega), \ 0 \le \eta_j \le 1, \ \eta_j = 1 \text{ on } \Gamma, \ {\rm supp}(\eta_j) \subset \Gamma + B_{2 \eps_j} \, \text{ and } \, \|\nabla \eta_j\|_{L^\infty(\Omega; \R^n)} \le \frac{1}{\eps_j}
\end{equation}
for some monotone decreasing sequence $(\eps_j)_j \subset (0, + \infty)$ such that $\eps_j \to 0$.
For instance, one can construct $(\eta_j)_j$ by taking the sequence of smooth approximations of the Lipschitz functions 
$$x \to \max\left \{ \left (1 - \frac{{\rm dist}(x, \Gamma)}{\eps_j} \right ), 0 \right \}.$$
Hence, we see that
\begin{align*}
\Jj_{\mu_{j}}[u_{j}] - \Jj_{\mu}[u] &= \sqrt{1+|\nabla u_{j}|^{2}}(B) - \sqrt{1+|D u|^{2}}(B) + \int_\Omega (\rho_{j} \ast u_{j} - u^{\lambda_{\mu}})\, d\mu\\
&=  \sqrt{1+|\nabla v_{j}|^{2}}(B) - \sqrt{1+|D u|^{2}}(B) + \int_\Omega (v_{j} - u^{\lambda_{\mu}})\, d\mu \\
&\qquad + \sqrt{1+|\nabla u_{j}|^{2}}(B) - \sqrt{1+|\nabla v_{j}|^{2}}(B) 
\end{align*}
where $v_{j} = \eta_j (\rho_{j}\ast u_{j}) + (1-\eta_j) u_{j}$. 
Thanks to Young's inequality, we get
\begin{align*}
 \|\rho_j * u_j - u_j\|_{L^1(\Omega)} & \le  \|\rho_j * u_j - \rho_j * u\|_{L^1(\Omega)} +  \|\rho_j * u - u\|_{L^1(\Omega)} +  \|u_j - u\|_{L^1(\Omega)} \\
 & \le 2\|u_j - u\|_{L^1(\Omega)} +  \|\rho_j * u - u\|_{L^1(\Omega)}  \to 0,
\end{align*}
which in turn implies that
\begin{align*}
\|v_j - u\|_{L^1(\Omega)} & \le \|\eta_j ( \rho_j \ast u_j - u) \|_{L^1(\Omega)} + \|(1 - \eta_j) (u_j - u)\|_{L^1(\Omega)} \\
& \le \| \rho_j \ast u_j - u_j \|_{L^1(\Omega)} + \|u_j - u\|_{L^1(\Omega)} \le 3\|u_j - u\|_{L^1(\Omega)} +  \|\rho_j * u - u\|_{L^1(\Omega)} \to 0,
\end{align*}
so that $v_j \to u$ in $L^1(\Omega)$. In addition, it is easy to notice that, for all $x \in B \setminus \Omega$, we have $v_j(x) = u_j(x) = \phi(x)$, so that $v_j \in \cS(\phi)$. Therefore, the lower-semicontinuity of $\Jj_{\mu}$ proved in Theorem \ref{thm:sci} implies that
\begin{equation} \label{eq:low_sem_inf}
\liminf_{j \to + \infty} \sqrt{1+|\nabla v_{j}|^{2}}(B) - \sqrt{1+|D u|^{2}}(B) + \int_\Omega (v_{j} - u^{\lambda_{\mu}})\, d\mu \ge 0\,.
\end{equation}
It remains to prove that 
\begin{equation} \label{eq:low_sem_2}
\liminf_{j \to + \infty} \sqrt{1+|\nabla u_{j}|^{2}}(B) - \sqrt{1+|\nabla v_{j}|^{2}}(B) \ge 0.
\end{equation}
To this purpose, we consider $(\psi, \varphi) \in C_c(B) \times C^1_c(B; \R^n)$ such that $\|(\psi, \varphi)\|_{L^{\infty}(B; \R^{n+1})} \le 1$, and we notice that
\begin{align*}
\int_{B} \psi + v_j \, \div \varphi \, dx & = \int_{B} \psi + \eta_j (\rho_j * u_j) \, \div \varphi + (1 - \eta_j) u_j \, \div \varphi \, dx \\
& = \int_{B} (1 - \eta_j) \psi + (\rho_j * u_j) \, \div (\eta_j \varphi) + u_j \, \div ( (1 - \eta_j) \varphi ) \, dx + \int_{\Omega} \eta_j \psi \, dx + \\
& + \int_{\Omega} (u_j - \rho_j * u_j) \, \varphi \cdot \nabla \eta_j \, dx \\
& = \int_{B} (1 - \eta_j) \psi + u_j \, \div ( \eta_j (\rho_j * \varphi) + (1 - \eta_j) \varphi ) \, dx + \int_{\Omega} \eta_j \psi \, dx +\\
& + \int_{\Omega} (u_j - \rho_j * u_j) \, \varphi \cdot \nabla \eta_j \, dx  + \int_{\Omega} u_j \, \div ( \rho_j * (\eta_j \varphi) - \eta_j (\rho_j * \varphi)) \, dx.
\end{align*}
To estimate the first term, we notice that 
\begin{equation*}
|((1 - \eta_j) \psi,  \eta_j (\rho_j * \varphi) + (1 - \eta_j) \varphi  )| \le \eta_j |\rho_j*\varphi| + (1-\eta_j) |(\psi, \varphi)| \le \|(\psi, \varphi)\|_{L^{\infty}(B; \R^{n+1})} \le 1.
\end{equation*}
Hence, 
\begin{equation*}
\int_{B} (1 - \eta_j) \psi + u_j \, \div ( \eta_j (\rho_j * \varphi) + (1 - \eta_j) \varphi ) \, dx \le \sqrt{1+|\nabla u_{j}|^{2}}(B).
\end{equation*}
It is then clear that the second term satisfies
\begin{equation*}
\left |\int_{\Omega} \eta_j \psi \, dx  \right | \le \|\psi\|_{L^\infty(\Omega)} |\Gamma + B_{2 \eps_j}| \le |\Gamma + B_{2 \eps_j}| \to 0 \ \text{ as } j \to + \infty.
\end{equation*}
Then, the third term is shown to be infinitesimal as follows:
\begin{align*}
\left | \int_{\Omega} (u_j - \rho_j * u_j) \, \varphi \cdot \nabla \eta_j \, dx \right |  & \le \|\nabla \eta_j\|_{L^{\infty}(\Omega; \R^n)} \|\varphi\|_{L^\infty(\Omega; \R^n)} \|\rho_j * u_j - u_j\|_{L^1(\Omega)} \\
& \le \frac{1}{\eps_j} \left ( 2\|u_j - u\|_{L^1(\Omega)} +  \|\rho_j * u - u\|_{L^1(\Omega)} \right ) \to 0\,,
\end{align*}
which is true as long as we choose $\eps_j$ in such a way that
\begin{equation*}
2\|u_j - u\|_{L^1(\Omega)} +  \|\rho_j * u - u\|_{L^1(\Omega)} = o(\eps_j) \ \text{ as } j \to + \infty.
\end{equation*}
Finally, we deal with the last, fourth term. As noted above, the assumption \eqref{eq:gamma_liminf_sequence} combined with the coercivity (see Section \ref{sec:coercivity}) implies that
\begin{equation*}
S := \sup_{j \in \N} \| \nabla u_j\|_{L^1(\Omega; \R^n)} < + \infty.
\end{equation*}
Therefore, we first integrate by parts
\begin{equation*}
\int_{\Omega} u_j \, \div ( \rho_j * (\eta_j \varphi) - \eta_j (\rho_j * \varphi)) \, dx = - \int_{\Omega} \nabla u_j  \cdot \left ( \rho_j * (\eta_j \varphi) - \eta_j (\rho_j * \varphi) \right ) \, dx,
\end{equation*}
and then we estimate the absolute value of this term by exploiting the Lipschitzianity of $\eta_j$ in the penultimate inequality:
\begin{align*}
\left | \int_{\Omega} \nabla u_j  \cdot \left ( \rho_j * (\eta_j \varphi) - \eta_j (\rho_j * \varphi) \right ) dx \right | & \le \| \nabla u_j\|_{L^1(\Omega; \R^n)} \sup_{x \in \Omega} \int_{B_{\delta_j}(x)} \rho_{j}(x-y) \left |\eta_j(y) - \eta_j(x) \right | |\varphi(y)| dy  \\
& \le S \|\varphi\|_{L^\infty(\Omega; \R^n)} \|\nabla \eta_j\|_{L^{\infty}(B_1; \R^n)} \delta_j \le S \frac{\delta_j}{\eps_j} \to 0\,,
\end{align*}
which holds as long as we choose $\eps_j$ in such a way that
\begin{equation*}
\delta_j = o(\eps_j) \ \text{ as } j \to + \infty.
\end{equation*}
Thus, to satisfy both conditions on $\eps_j$ we can for instance choose
$$\eps_{j} = \max \{ \sqrt{\delta_j}, \sqrt{2\|u_j - u\|_{L^1(\Omega)} +  \|\rho_j * u - u\|_{L^1(\Omega)}} \}.$$
All in all, by taking the supremum over 
$$(\psi, \varphi) \in C_c(B) \times C^1_c(B; \R^n) \ \text{ such that } \ \|(\psi, \varphi)\|_{L^{\infty}(B; \R^{n+1})} \le 1$$ 
we get
\begin{equation*} 
\sqrt{1 + |\nabla v_j|^2}(B) \le \sqrt{1+|\nabla u_{j}|^{2}}(B) + |\Gamma + B_{2 \eps_j}| + \sqrt{2\|u_j - u\|_{L^1(\Omega)} +  \|\rho_j * u - u\|_{L^1(\Omega)}} + S \sqrt{\delta_j},
\end{equation*}
from which \eqref{eq:low_sem_2} readily follows.
Thus, by combining the inequalities \eqref{eq:low_sem_inf} and \eqref{eq:low_sem_2} with the superadditivity of the liminf we obtain the proof of \eqref{eq:lim_inf_prove}, and thus of the theorem.
\end{proof}

We conclude this section with the $\Gamma$-compactness property, which stems from the uniform coercivity of the sequence of functionals $\Jj_{\mu_{j}}$ for $j$ large.
\begin{proposition}\label{prop:gammacompact}
Let $\mu,\mu_{j},\Jj_{\mu},\Jj_{\mu_{j}}$ be as in Theorem \ref{thm:gammaconv}. Assume that a sequence $(u_{j})_{j}\subset \cS(\phi)$ satisfies 
\[
\sup_{j\in \N} \Jj_{\mu_{j}}[u_{j}] < +\infty\,.
\]
Then $(u_{j})_{j}$ is relatively compact in $L^{1}(B)$.
\end{proposition}
\begin{proof}
It is a simple consequence of \eqref{eq:stimaJmu} written for $\mu_{j}$ in place of $\mu$, and coupled with the fact that $\mu_{j}$ is uniformly non-extremal thanks to Proposition \ref{prop:muapprox}.
\end{proof}

\section{Characterization of $T(u)$}\label{sec:further1}

In this section, we show that the vector field $T$ associated with a weak solution $u$ of \eqref{eq:PMCM} is uniquely determined up to Lebesgue null sets by
\begin{equation}\label{eq:Tuformula}
T = T(u) := \frac{\nabla u}{\sqrt{1 + |\nabla u|^2}} \ \Leb{n}\text{-a.e. in } \Omega\,,
\end{equation}
where $Du = \nabla u \Leb{n} + D^s u$ is the decomposition of the gradient measure of $u$. 

\begin{theorem} \label{thm:Tuformula}
Let $\mu \in \MH(\Omega)$ be an admissible measure and let $\lambda : \Omega \to [0, 1]$ be a Borel function. Assume that there exist $u \in BV(\Omega)$ and $T \in \DM^{\infty}(\Omega)$ such that \eqref{L_infty_bound_eq}, \eqref{divergence_eq} and \eqref{pairing_eq} hold true. 
Then $T$ is unique and satisfies \eqref{eq:Tuformula}.
\end{theorem}

\begin{proof}
We notice that \eqref{pairing_eq} means that
\begin{equation*}
(T, Du)_{\lambda} = \sqrt{1 + |Du|^2} - \sqrt{1 - |T|^2} \Leb{n}  \ \text{ on } \Omega.
\end{equation*}
Thanks to Lemma \ref{lem:2a2b-bis}, we know that $u^\lambda \in L^1(\Omega; |\mu|)$. Since $\mu = \div T$, we can apply \cite[Proposition 4.7]{crasta2019pairings} to conclude that 
\begin{equation*}
(T, Du)_{\lambda} = (T \cdot \nabla u) \Leb{n} + (T, D u)_{\lambda}^s  \ \text{ on } \Omega\,.
\end{equation*}
Hence, if we combine these two equations and consider only the absolutely continuous parts of the measures involved, we get
\begin{equation} \label{eq:T_u_pair_explicit_Leb}
T \cdot \nabla u = \sqrt{1 + |\nabla u|^2} - \sqrt{1 - |T|^2} \ \Leb{n}\text{-a.e. in } \Omega.
\end{equation}
It is then clear that $$T(x) \cdot \nabla u(x) \ge 0 \text{ for } \Leb{n}\text{-a.e. } x \in \Omega,$$
so that there exist a scalar function $m \ge 0$ and a vector-valued function $V$ orthogonal to such $\nabla u$ that
$$T(x) = m(x) \frac{\nabla u(x)}{|\nabla u(x)|} + V(x)$$
for $\Leb{n}$-a.e. $x \in \Omega$ such that $\nabla u(x) \neq 0$. Let us fix one such $x \in \Omega$. From \eqref{eq:T_u_pair_explicit_Leb} we get
\begin{equation*}
m(x) |\nabla u(x)| = \sqrt{1 + |\nabla u(x)|^2} - \sqrt{1 - m^2(x)-|V(x)|^{2}}.
\end{equation*}
which, by bringing all the terms involving $|\nabla u|$ on the same side and taking a square power, gives us
\begin{equation*}
m(x)^2 |\nabla u(x)|^2 + |\nabla u(x)|^2 + 1 - 2m(x) |\nabla u(x)| \sqrt{1 + |\nabla u(x)|^2} = 1 - m^2(x) - |V(x)|^2.
\end{equation*}
This last equation easily implies 
\begin{equation*}
(m(x) \sqrt{1+|\nabla u(x)|^2} - |\nabla u(x)|)^2 = - |V(x)|^2,
\end{equation*}
from which we deduce that
\begin{equation*}
V(x) = 0 \ \text{ and } \ m(x) = \frac{|\nabla u(x)|}{\sqrt{1 + |\nabla u(x)|^2}} \text{ for } \Leb{n}\text{-a.e. } x \in \Omega \cap \{ \nabla u \neq 0\}.
\end{equation*}
This proves \eqref{eq:Tuformula} in the case $\nabla u(x) \neq 0$; if instead $\nabla u(x) = 0$, then \eqref{eq:T_u_pair_explicit_Leb} reduces to
$$0 = 1 - \sqrt{1 - |T(x)|^2},$$
which clearly implies $T(x) = 0$, again proving \eqref{eq:Tuformula} in such case.
\end{proof}

\begin{remark}
In the light of \eqref{eq:Tuformula}, we can refine the statement of Theorem \ref{thm:Jsol} by replacing $\sqrt{1-|T|^2}$ with $\frac{1}{\sqrt{1+|\nabla u|^2}}$. 
In particular, \eqref{eq:min_value} becomes
\begin{equation*}
\Jj_\mu[u] = \int_{B \setminus \Omega} \sqrt{1 + |\nabla \phi|^2} \, dx + \int_{\partial \Omega} \left ( |{\rm Tr}_{\partial \Omega}(u - \phi)| - {\rm Tr}_{\partial \Omega}(u) {\rm Tr}^i(T, \partial \Omega) \right ) \, d \Haus{n-1} + \int_{\Omega} \frac{1}{\sqrt{1 + |\nabla u|^2}} \, dx.
\end{equation*}
\end{remark}

\begin{remark} \label{rem:pairing_decomp_singular}
Under the same assumptions of Theorem \ref{thm:Tuformula}, we see that, if $\nabla u$ is constant $\Leb{n}$-a.e. on $\Omega$, then \eqref{eq:Tuformula} implies that $T = T(u)$ is constant $\Leb{n}$-a.e. on $\Omega$, and therefore $\mu = 0$. Consequently, if the admissible measure $\mu$ is nontrivial, then the absolutely continuous part $\nabla u$ of the distributional gradient $Du$ of a weak solution to \eqref{eq:PMCM} cannot be essentially constant on $\Omega$\footnote{Roughly speaking, the non triviality of the curvature implies the non-constancy of the absolutely continuous part of the gradient of $u$.}.  
Another relevant consequence of \eqref{eq:Tuformula} is that 
\begin{equation} \label{eq:T_1_negligible_level}
\big| \{ x \in \Omega : |T(x)| = 1\}\big| = 0,
\end{equation} 
due to the fact that $\nabla u \in L^1(\Omega; \R^n)$. Hence, we can invert \eqref{eq:Tuformula}, obtaining
\begin{equation} \label{eq:T_nabla_uformula}
\nabla u(x) = \frac{T(x)}{\sqrt{1 - |T(x)|^{2}}} \ \text{ for } \Leb{n}\text{-a.e. } x \in \Omega.
\end{equation}

In addition, if we consider instead the singular parts of the measures involved in \eqref{pairing_eq}, by \cite[Proposition 4.7]{crasta2019pairings} we get
\begin{equation} \label{eq:T_u_pair_explicit}
(T, D u)_{\lambda}^s = |D^s u| \text{ on } \Omega.
\end{equation}
Moreover, thanks to the decomposition of the singular part of $Du$, $$D^s u = D^c u + D^j u,$$ and thanks to \cite[Proposition 4.7]{crasta2019pairings}, we have
\begin{equation*}
(T, D u)_{\lambda}^s = (T, D u)_{*}^c + (T, D u)_{\lambda}^j \ \text{ on } \Omega,
\end{equation*}
so that \eqref{eq:T_u_pair_explicit} can be rewritten as
\begin{equation} \label{eq:pairingextreme}
\begin{cases}
(T, D u)_{*}^c = |D^c u| & \text{ on } \Omega. \\
(T, D u)_{\lambda}^j = |D^j u| & \text{ on } \Omega.
\end{cases}
\end{equation}
In addition, still by \cite[Proposition 4.7]{crasta2019pairings} we have
\begin{equation} \label{eq:1dpairingj1}
(T,Du)_{\lambda}^{j} = [\lambda {\rm Tr}^e(T,J_{u}) + (1 - \lambda) {\rm Tr}^i(T,J_{u})]\, |D^{j}u|,
\end{equation}
where ${\rm Tr}^{e}(T,J_{u})$ and ${\rm Tr}^{i}(T,J_{u})$ denote the normal traces of $T$ on the jump set of $u$ corresponding to the orientation given by $\nu_u$ (for their precise definition we refer the reader to \cite[Section 2.4]{crasta2019pairings}). In particular, the second equation in \eqref{eq:pairingextreme} combined with \eqref{eq:1dpairingj1} implies that
\begin{equation} \label{eq:convex_comb_eq_1_traces}
\lambda {\rm Tr}^e(T,J_{u}) + (1-\lambda) {\rm Tr}^i(T,J_{u}) = 1 \quad |D^j u|\text{-a.e.}
\end{equation}
Then, we notice that \eqref{eq:pairing_estimate_lambda} implies an analogous bound for the singular part of $D^s u$, simply by restricting it on any Lebesgue negligible set: hence, we have
\begin{equation} \label{eq:pairing_estimate_lambda_s}
|(T, D u)_{\lambda}^s| \le \|T\|_{L^\infty(\Omega; \R^n)} |D^s u| \ \text{ on } \Omega,
\end{equation}
so that \eqref{eq:T_u_pair_explicit} implies that $D^s u \neq 0$ if and only if $\|T\|_{L^\infty(\Omega; \R^n)} = 1$.
Moreover, we notice that \eqref{eq:pairing_estimate_lambda} and therefore also \eqref{eq:pairing_estimate_lambda_s} can be localized to any open set $\Omega' \subset \Omega$, see \cite[Proposition 3.4]{comi2022representation}.
Hence, for all $x \in {\rm supp}(|D^su|)$, we can exploit \eqref{eq:pairing_estimate_lambda_s} localized to any ball $B_r(x) \subset \Omega$ together with \eqref{eq:T_u_pair_explicit} in order to get
\begin{equation*}
0 < |D^s u|(B_r(x)) \le \|T\|_{L^\infty(B_r(x); \R^n)} |D^s u|(B_r(x)) \ \text{ for all } r > 0 \text{ such that } B_r(x) \subset \Omega,
\end{equation*}
which, by \eqref{L_infty_bound_eq}, implies $\|T\|_{L^\infty(B_r(x); \R^n)} = 1$ for all such $r > 0$, and so
\begin{equation*}
\lim_{r \to 0^+} \|T\|_{L^\infty(B_r(x); \R^n)} = 1 \ \text{ for all } x \in {\rm supp}(|D^su|).
\end{equation*}
\end{remark}

Under the assumption of Theorem \ref{thm:Tuformula}, it is clear that the absolutely continuous part of the measure $Du$ satisfies
\begin{equation*}
|\nabla u| = \frac{|T|}{\sqrt{1 - |T|^2}} \ \Leb{n}\text{-a.e. in } \Omega\,,
\end{equation*}
thanks to \eqref{eq:Tuformula} and \eqref{eq:T_1_negligible_level}. This implies that $\nabla u \in L^{\infty}(\Omega; \R^n)$ as long as $$|T(x)| \le L < 1 \text{ for } \Leb{n}\text{-a.e. } x \in \Omega.$$
However, this does not always hold, even if $\div T$ is a non-extremal measure $\mu$. 
Indeed, due to already established results recalled in Lemma \ref{lem:existence_optimal_T}, we know that, if $\Omega$ is a bounded open set with Lipschitz boundary and $\mu \in \Mm(\Omega)$ satisfies \eqref{eq:subcritical_cond} for some $L \in (0,1)$ and is admissible, there exists some $F \in \DM^\infty(\Omega)$ such that $\div F = \mu$ on $\Omega$ and $\|F\|_{L^\infty(\Omega; \R^n)} \le L$, but the latter condition is not necessarily satisfied by any solution to the equation $\div G = \mu$. 
On the other hand, one could wonder if indeed $u$ is Lipschitz continuous under the assumption $\|T\|_{L^\infty(\Omega; \R^n)} \le L < 1$: this is proved to be true in the following result.

\begin{proposition} \label{prop:bound_T_u_Lip}
Let $\mu \in \MH(\Omega)$ be an admissible measure and let $\lambda : \Omega \to [0, 1]$ be a Borel function. Assume that there exist $u \in BV(\Omega)$ and $T \in \DM^{\infty}(\Omega)$ such that \eqref{L_infty_bound_eq}, \eqref{divergence_eq} and \eqref{pairing_eq} hold true. Then for every $L \in (0, 1)$ and every open set $V \subset \Omega$ such that $|V \setminus \{ |T| \le L \}| = 0$ we have $u \in W^{1,1}(V)$ with $\nabla u \in L^\infty(V; \R^n)$. In particular, if $\|T\|_{L^\infty(\Omega; \R^n)} \le L < 1$, then $u \in \Lip(\Omega)$.
\end{proposition}

\begin{proof}
Clearly, the couple $(u, T)$ satisfies \eqref{pairing_eq} on $V$, and $\|T\|_{L^{\infty}(V; \R^n)} \le L$. Hence, by \eqref{eq:pairlambda-vs-area} we get
\begin{equation*}
\sqrt{1+|Du|^2} - \sqrt{1 - |T|^2} \Leb{n} = (T, Du)_\lambda \le L \sqrt{1 + |Du|^2} - \sqrt{L^2 - |T|^2} \Leb{n} \quad \text{ on } V.
\end{equation*}
Hence, we get
\begin{equation*}
\sqrt{1+|Du|^2} \le \frac{\sqrt{1 - |T|^2} - \sqrt{L^2 - |T|^2}}{1 - L} \Leb{n} \quad \text{ on } V,
\end{equation*}
which immediately implies $|Du| \ll \Leb{n}$ on $V$, and that the weak gradient of $u$ satisfies
\begin{equation*}
|\nabla u| \le \frac{\sqrt{1 - |T|^2} - \sqrt{L^2 - |T|^2}}{1 - L} \quad \Leb{n}\text{-a.e. on } V.
\end{equation*}
Finally, if $V = \Omega$, the conclusion follows from the fact that $\Omega$ is a bounded open set with a Lipschitz boundary.
\end{proof}

An interesting consequence of this result is that, if $u$ is a solution to the (PMCM) equation such that the density of the absolutely continuous part of $Du$ is essentially bounded, then $u$ is Lipschitz continuous. Roughly speaking, this means that the singular part of $Du$ is concentrated on a Borel subset of the set $\{ |\nabla u| = + \infty\}$ or, equivalently, of the set $\{ |T(u)| = 1\}$, coherently with Remark \ref{rem:pairing_decomp_singular}.

\begin{corollary}
Let $\mu \in \MH(\Omega)$ be an admissible measure and let $\lambda : \Omega \to [0, 1]$ be a Borel function. Assume that there exist $u \in BV(\Omega)$ and $T \in \DM^{\infty}(\Omega)$ such that \eqref{L_infty_bound_eq}, \eqref{divergence_eq} and \eqref{pairing_eq} hold true. If $\nabla u \in L^{\infty}(\Omega; \R^n)$, then $u \in \Lip(\Omega)$.
\end{corollary}
\begin{proof}
Thanks to \eqref{eq:Tuformula}, we get
\begin{equation*}
\|T\|_{L^\infty(\Omega; \R^n)} \le \frac{\|\nabla u\|_{L^\infty(\Omega; \R^n)}}{\sqrt{1 + \|\nabla u\|_{L^\infty(\Omega; \R^n)}^2}} < 1,
\end{equation*}
since the function $[0, + \infty) \ni \xi \to \frac{\xi}{\sqrt{1 + \xi^2}}$ is increasing. Hence, Proposition \ref{prop:bound_T_u_Lip} yields $u \in \Lip(\Omega)$.
\end{proof}


\section{Maximum principle for continuous solutions}
Next, we prove a maximum principle in the particular case of solutions to the PMCM equation which are continuous inside the domain $\Omega$. A restriction of this kind is indeed necessary, as the example of non-uniqueness constructed in Section \ref{sec:example} shows. Instead, the continuity assumption on $u$ implies that
$$(T, Du)_{\lambda} = (T, Du)_* = \div(u T) - u \div T$$
for any choice of the Borel map $\lambda$, see \cite[Remark 4.10]{crasta2019pairings}. We stress that this maximum principle holds without a-priori restrictions on the admissible measure $\mu$. 

\begin{theorem}\label{thm:maxprinc}
Let $\mu_1, \mu_2 \in \MH(\Omega)$ be admissible measures such that $\mu_1 \le \mu_2$. Let $u_i \in BV(\Omega) \cap C^0(\Omega)$ and $T_i \in \DM^{\infty}(\Omega)$ such that \eqref{L_infty_bound_eq}, \eqref{divergence_eq} and \eqref{pairing_eq} are satisfied for $i =1, 2$. Assume 
that ${\rm Tr}_{\de\Omega}(u_1) \ge  {\rm Tr}_{\de\Omega}(u_2)$ $\Haus{n-1}$-a.e. on $\partial \Omega$, then we have $u_1 \ge u_2$ in $\Omega$.
\end{theorem}

\begin{proof}
By the assumptions on $u_i$ and $T_i$, we have
\begin{equation} \label{divergence_eq2_1_2}
(T_i, Du_i)_\lambda = \sqrt{1 + |D u_i|^{2}} - \sqrt{1 - |T_i|^{2}} \Leb{n} \ \text{ on } \Omega
\end{equation}
for $i = 1, 2$ and $\lambda = \lambda_{\mu}$.
Let $\eps > 0$, $A_\eps := \{ x \in \Omega : u_2(x) - u_1(x) > \eps \}$ and $\varphi_\eps := \max\{ u_2 - u_1, \eps \}$. It is clear that $\varphi_\eps \in BV(\Omega) \cap C^0(\Omega)$ with
\begin{equation*}
D \varphi_\eps = \begin{cases} D u_2 - Du_1 & \text{ on } A_\eps, \\
0 & \text{ on } \Omega \setminus A_\eps, \end{cases}
\end{equation*}
and that the open sets $A_\eps$ have finite perimeter in $\Omega$ for $\Leb{1}$-a.e. $\eps > 0$. In the rest of the proof, we consider only such good values of $\eps$. In addition, the assumption ${\rm Tr}_{\de\Omega}(u_1) \ge  {\rm Tr}_{\de\Omega}(u_2)$ $\Haus{n-1}$-a.e. on $\partial \Omega$ implies ${\rm Tr}_{\de \Omega}(\varphi_\eps) = \eps$ $\Haus{n-1}$-a.e. on $\partial \Omega$, and therefore, by continuity, we deduce that $\varphi_\eps(x) \le 2 \eps$ for all $x \in \Omega$ with sufficiently small distance from $\partial \Omega$. Since the functions $u_1, u_2$ are continuous, then $D u_i$ has no jump part, and so we get $(T_i, Du_i)_\lambda = (T_i, Du_i)_*$ for any Borel function $\lambda : \Omega \to [0, 1]$ and $i =1,2$, and we set $(T_i, Du_i) := (T_i, Du_i)_*$ for more simplicity. We claim that 
\begin{equation} \label{eq:sign_diff_pairing}
(T_1 - T_2, D(u_2 - u_1)) \le 0
\end{equation}
Indeed, thanks to the bilinearity of the standard pairing, \eqref{divergence_eq2_1_2} and \eqref{eq:pairlambda-vs-area} with $a = 1$, we conclude that
\begin{align*}
(T_1 - T_2, D(u_2 - u_1)) & = - (T_1, Du_1) - (T_2, Du_2) + (T_1, Du_2) + (T_2, Du_1) \\
& = (T_1, Du_2) - \left ( \sqrt{1 + |Du_2|^2} - \sqrt{1-|T_1|^2} \right )  + \\
& + (T_2, Du_1) - \left (\sqrt{1+|Du_1|^2} - \sqrt{1-|T_2|^2} \right ) \le 0.
\end{align*}
In addition, the continuity of $u_1$ and $u_2$ implies that $|D u_i|(\partial^* A_\eps) =0$ for $i = 1, 2$. Given that
\begin{equation*}
\varphi_\eps = \chi_{A_\eps} (u_2 - u_1) + \eps \chi_{\Omega \setminus A_\eps},
\end{equation*}
we exploit the linearity of the standard pairing and \cite[Proposition 4.11]{crasta2017anzellotti} to see that
\begin{align*}
(T_1 - T_2, D\varphi_\eps) & = (T_1 - T_2, D ( \chi_{A_\eps} (u_2 - u_1) ) ) + (T_1 - T_2, D (\eps \chi_{\Omega \setminus A_\eps}) ) \\ 
& = \chi_{A_\eps}^* (T_1 - T_2, D(u_2 - u_1)) + (u_2 - u_1)^* (T_1 - T_2, D \chi_{A_\eps}) - \eps (T_1 - T_2, D \chi_{A_\eps}) \\
& = \chi_{A_\eps^1} (T_1 - T_2, D(u_2 - u_1)) + (u_2 - u_1) (T_1 - T_2, D \chi_{A_\eps}) - \eps (T_1 - T_2, D \chi_{A_\eps}) \\
& = (T_1 - T_2, D(u_2 - u_1)) \restrict A_\eps^1,
\end{align*}
since, on the one hand, the pairing $(T_1 - T_2, D(u_2 - u_1))$ does not see $\partial^* A_\eps$, due to \eqref{eq:pairing_estimate_lambda} and the fact $|D (u_1 - u_2)|(\partial^* A_\eps) =0$, and, on the other hand, $(u_2 - u_1)$ is continuous and $u_2 - u_1 = \eps$ on $\partial^* A_\eps$, which is where the pairing $(T_1 - T_2, D \chi_{A_\eps})$ is concentrated, again by \eqref{eq:pairing_estimate_lambda}.
Therefore, given that $A_\eps \subset A_\eps^1$, we deduce that
\begin{equation} \label{eq:pairing_max_princ_1}
0 \ge (T_1 - T_2, D(u_2 - u_1)) \restrict A_\eps \ge (T_1 - T_2, D(u_2 - u_1)) \restrict A_\eps^1 = (T_1 - T_2, D\varphi_\eps).
\end{equation}
Now, since $\Omega$ is an open bounded set with Lipschitz boundary, there exists a family of bounded open sets with smooth boundary $(\Omega_k)_{k \in \N}$ such that 
\begin{equation*}
\Omega_k \subset \Omega_{k+1}, \  \bigcup_{k = 0}^{+\infty} \Omega_k = \Omega \ \text{ and } \Per(\Omega_k) \to \Per(\Omega)
\end{equation*}
(see \cite[Theorem 1.1]{MR3314116} and the subsequent discussion).
In particular, there exists $k_0 \in \N$ such that, for all $k \ge k_0$, we have 
\begin{equation*}
\eps \le \varphi_\eps \le 2 \eps \ \text{ on } \partial \Omega_k.
\end{equation*}
Hence, for all $k \ge k_0$ we exploit \eqref{eq:Leibniz_lambda_gen} and \eqref{eq:GG_1} to obtain
\begin{align*}
(T_1 - T_2, D\varphi_\eps)(\Omega_k) & = \div( (T_1 - T_2) \varphi_\eps)(\Omega_k) - \int_{\Omega_k} \varphi_\eps \, d \div (T_1 - T_2) \\
& = - \int_{\partial \Omega_k} \varphi_\eps {\rm Tr}^i(T_1 - T_2, \partial \Omega_k) \, d \Haus{n-1} + \int_{\Omega_k} \varphi_\eps \, d \div(T_2 - T_1) \\
& \ge - 4 \eps \Per(\Omega_k) + \int_{\Omega_k} \varphi_\eps \, d (\mu_2 - \mu_1) \ge - 4 \eps \Per(\Omega_k),
\end{align*}
thanks to the fact that $\mu_2 \ge \mu_1$ and that 
$$\|{\rm Tr}^i(T_1 - T_2, \partial^* \Omega_k)\|_{L^\infty(\partial \Omega_k; \Haus{n-1})} \le \|T_1 - T_2\|_{L^{\infty}(\Omega_k; \R^n)} \le \|T_1\|_{L^\infty(\Omega; \R^n)} + \|T_2\|_{L^\infty(\Omega; \R^n)} \le 2.$$ 
All in all, we obtain
\begin{equation*}
0 \ge (T_{1}-T_{2}, D\varphi_\eps)(\Omega_k) \ge - 4 \eps \Per(\Omega_k),
\end{equation*}
so that, by taking the limit as $k \to + \infty$, we conclude that 
\begin{equation*}
0 \ge (T_{1}-T_{2}, D\varphi_\eps)(\Omega) \ge - 4 \eps \Per(\Omega),
\end{equation*}
which, thanks to \eqref{eq:pairing_max_princ_1}, implies
\begin{equation*}
0 \ge (T_1 - T_2, D(u_2 - u_1))(A_\eps) \ge - 4 \eps \Per(\Omega)
\end{equation*}
Finally, by taking the limit as $\eps \to 0$ (more precisely, along a suitable sequence $(\eps_j)_j$ for which the sets $A_{\eps_j}$ have finite perimeter), we conclude that
\begin{equation*}
(T_1 - T_2, D(u_2 - u_1))(A_0) = 0, 
\end{equation*}
which is well posed since $A_0$ is open, and, combined with \eqref{eq:sign_diff_pairing}, implies
\begin{equation*}
(T_1 - T_2, D(u_2 - u_1)) = 0 \ \text{ on } A_0.
\end{equation*}
The bilinearity of the pairing yields
\begin{equation*}
(T_2, Du_1) - (T_2, Du_2) =  (T_1, Du_1) - (T_1, Du_2) \ \text{ on } A_0,
\end{equation*}
and then we exploit \eqref{divergence_eq2_1_2} and Theorem \ref{lemma:pairlambda-vs-area} to obtain
\begin{equation*}
0 \ge (T_2, Du_1) - \sqrt{1+|Du_1|^2} + \sqrt{1-|T_2|^2} \Leb{n} = \sqrt{1+|Du_2|^2} - \sqrt{1-|T_1|^2} \Leb{n} - (T_1, Du_2) \ge 0 \ \text{ on } A_0.
\end{equation*}
In particular, this means that both sides of the equation are identically zero on the open set $A_0$. Hence, if we add the right hand side to \eqref{divergence_eq2_1_2} in the case $i=1$ we get
\begin{equation*}
\sqrt{1+|Du_1|^2} + \sqrt{1+|Du_2|^2} - (T_1, D(u_1 + u_2))  - 2\sqrt{1-|T_1|^2} \Leb{n} = 0 \ \text{ on } A_0.
\end{equation*}
On the other hand, Theorem \ref{lemma:pairlambda-vs-area} implies that 
\begin{equation*}
\frac{1}{2} (T_1, D(u_1 + u_2)) = \left ( T_1, D \left ( \frac{u_1 + u_2}{2} \right ) \right ) \le \sqrt{1 + \left | D \left (\frac{u_1 + u_2}{2} \right ) \right |^2} - \sqrt{1 - |T_1|^2} \Leb{n} \ \text{ on } \Omega.
\end{equation*}
All in all, we conclude that
\begin{equation*}
\frac{\sqrt{1+|Du_1|^2} + \sqrt{1+|Du_2|^2}}{2} \le \sqrt{1 + \left | D \left (\frac{u_1 + u_2}{2} \right ) \right |^2} \ \text{ on } A_0.
\end{equation*}
Therefore, by convexity, this implies that $$D u_1 = D u_2 \ \text{ on } A_0,$$ which means $$D \varphi_0 = 0 \ \text{ on } A_0,$$ where $\varphi_0 := \max\{u_2 - u_1, 0\} \in BV(\Omega) \cap C^0(\Omega)$. Thus, $\varphi_0$ is constant on $\Omega$; however, since ${\rm Tr}_{\de\Omega}(\varphi_0)(x) = 0$ for $\Haus{n-1}$-a.e. $x \in \partial \Omega$ by assumption, we conclude that $\varphi_0 = 0$ in $\Omega$, and so $u_1 \ge u_2$ in $\Omega$.
\end{proof}

As a consequence, we obtain the following uniqueness result.

\begin{corollary}[Uniqueness for continuous weak solutions]\label{corol:uniqueness}
Let $\mu\in \MH(\Omega)$ be an admissible measure. Let $u_{1},u_{2} \in BV(\Omega) \cap C^0(\Omega)$ be weak solutions of \eqref{eq:PMCM} such that ${\rm Tr}_{\de\Omega}(u_1) =  {\rm Tr}_{\de\Omega}(u_2)$ $\Haus{n-1}$-a.e. on $\partial \Omega$. Then $u_1 = u_2$ in $\Omega$.
\end{corollary}

\begin{remark}
We notice that Theorem \ref{thm:maxprinc} and Corollary \ref{corol:uniqueness} hold true also under the assumption $u_1, u_2 \in W^{1, 1}(\Omega)$. Indeed, for Sobolev functions we have
\begin{equation*}
(T_j, Du_j)_\lambda = (T_j \cdot \nabla u_j) \Leb{n},
\end{equation*}
for $j = 1, 2$ and any Borel function $\lambda : \Omega \to [0,1]$ (see \cite[Remark 4.8]{crasta2019pairings}), so that the pairings are absolutely continuous. Hence, it is enough that the sets $A_\eps = \{ u_2 - u_1 > \eps\}$ are measurable for all $\eps \ge 0$, as well as of finite perimeter for $\Leb{1}$-a.e. $\eps > 0$. In addition, the functions $u_1, u_2$ have no jumps, so that $u_2^* - u_1^* = \tilde{u}_2 - \tilde{u}_1 = \eps$ $\Haus{n-1}$-a.e. on $\partial^* A_\eps$, for all such good values of $\eps > 0$. Thus, the proof of Theorem \ref{thm:maxprinc} can be easily adapted to the case $u_1, u_2 \in W^{1, 1}(\Omega)$, immediately yielding the analogous uniqueness result for Sobolev solutions\footnote{Note that this follows also from testing directly with $\max\{u_{2}-u_{1},0\}$, as done in the classical (smooth) case.}.
\end{remark}

\medskip

\section{Examples}\label{sec:example}

\subsection{The PMCM equation in dimension 1} \label{sec:n_1}

Let $n = 1$. In the one-dimensional case, we assume $\Omega$ to be an open bounded interval.
It is not difficult to notice that the equation $$\div T = \mu \text{ on } \Omega$$ can be explicitly solved: indeed, the general solution is 
\begin{equation}\label{eq:gensol1d}
T(x) = \mu((- \infty, x) \cap \Omega) + c, \text{ for } c \in \R\,.
\end{equation}
In addition, we clearly have $T \in BV(\Omega)$, since $\DM^\infty(\Omega) = BV(\Omega)$ for $n =1$. Given $u \in BV(\Omega)$, we denote by $u'$ the density of the absolutely continuous part of $Du$; that is,
$$Du = u' \Leb{1} + D^s u.$$
Then, Theorem \ref{thm:Tuformula} implies that any couple $(u, T) \in BV(\Omega) \times BV(\Omega)$ which satisfies \eqref{L_infty_bound_eq}, \eqref{divergence_eq} and \eqref{pairing_eq} must also satisfy the following identity:
\begin{equation} \label{eq:Tu_dim1}
T(x) = \frac{u'(x)}{\sqrt{1 + (u'(x))^2}} \ \text{ for } \Leb{1}\text{-a.e. } x \in \Omega.
\end{equation}
In addition, \eqref{eq:T_nabla_uformula} yields
\begin{equation}\label{eq:uprimoT}
u'(x) = \frac{T(x)}{\sqrt{1 - (T(x))^2}} \ \text{ for } \Leb{1}\text{-a.e. } x \in \Omega\,.
\end{equation}
As for the singular parts of $Du$, they satisfy \eqref{eq:pairingextreme}, as noticed in Remark \ref{rem:pairing_decomp_singular}. In particular, since $T \in BV(\Omega)$, the first equation in \eqref{eq:pairingextreme} reduces to
\begin{equation*} 
\widetilde{T} \, D^c u = |D^c u| \ \text{ on } \Omega,
\end{equation*}
thanks to \cite[Theorem 3.3 and Remark 3.4]{crasta2017anzellotti}, which easily implies 
\begin{equation} \label{eq:Cantor_BV_one_dim}
|\widetilde{T}(x)| = 1 \ \text{ for } |D^c u|\text{-a.e. } x \in \Omega.
\end{equation} 
As for the second equation in \eqref{eq:pairingextreme}, we exploit \eqref{eq:1dpairingj1} and \cite[Remark 4.9]{crasta2019pairings} to conclude that \eqref{eq:convex_comb_eq_1_traces} becomes
\begin{equation} \label{eq:convex_comb_eq_1_traces_1_dim}
(\lambda T^e_{J_u} + (1-\lambda) T^i_{J_u} ) \nu_u = 1 \quad |D^j u|\text{-a.e.,}
\end{equation}
where $T^i_{J_u}$ and $T^e_{J_u}$ are the interior and exterior traces of $T$ on $J_u$ in the $BV$ sense. Moreover, since we are on $\R$, at every $x_0 \in J_u$ these traces coincide with either $T(x_0^-) = \lim_{\eps \to 0^+} T(x_0 - \eps)$ or $T(x_0^+) = \lim_{\eps \to 0^+} T(x_0 + \eps)$, according to the orientation given by $\nu_u(x_0) \in \{\pm 1\}$.

A striking consequence of this fact is the following convexity result when $\mu$ is nonnegative (by a simple sign change, an analogous concavity result holds when $\mu$ is non-positive).

\begin{theorem} \label{thm:continuity_one_dim}
Let $\Omega = (a,b)$ with $a<b\in \R$. Let $\mu \in \Mm(\Omega)$ be nonnegative and satisfying \eqref{eq:necess_cond}. Assume that $u\in BV(\Omega)$ and $T\in \DM^{\infty}(\Omega)$ satisfy \eqref{L_infty_bound_eq}, \eqref{divergence_eq} and \eqref{pairing_eq} for $\lambda = \lambda_\mu$. Then $u$ is convex on $\Omega$, and therefore locally Lipschitz. In addition, if $u_1, u_2$ are weak solutions of \eqref{eq:PMCM} satisfying $u_1(a^{+}) = u_2(a^{+})$ and $u_2(b^{-}) = u_2(b^{-})$, then $u_1 = u_2$ on $\Omega$.
\end{theorem}
\begin{proof}
We notice that the necessary condition \eqref{eq:necess_cond} implies that $\mu \in \PB_1(\Omega)$ and therefore allows us to exploit Proposition \ref{prop:adm_PB_div}(2) to conclude that $\mu \in BV(\Omega)^*$, so that $\mu$ is admissible, since $\mu \ge 0$. By \eqref{eq:gensol1d}, we have $T(x) = \mu((a, x)) + c$, for some constant $c \in \R$, so that $T$ is left-continuous at every $x \in \Omega$, continuous at $\Leb{1}$-a.e. $x \in \Omega$ (due to the non-concentration property of Radon measures), and non-decreasing. In particular, this means that \eqref{L_infty_bound_eq} implies $|T(x)| \le 1$ for all $x \in \Omega$.
Now assume by contradiction that there exists $x_{0}\in J_{u}$. We notice that $\mu \ge 0$ implies $\lambda_\mu = 0$, so that \eqref{eq:convex_comb_eq_1_traces} yields
\[
{\rm Tr}^i(T,J_{u})(x_{0}) =1\,,
\]
which means that either $T(x_{0}^{+}) = 1$, if $\nu_u(x_0) = 1$, or $T(x_{0}^{-}) = - 1$, if $\nu_u(x_0) = - 1$, see \eqref{eq:convex_comb_eq_1_traces_1_dim} and the subsequent remarks. However, $T$ is non-decreasing and $|T(x)|\le 1$ for all $x \in \Omega$. Hence, if we had $T(x_{0}^{+}) = 1$ we would get $T(x) = 1$ for all $x \in (x_{0},b)$, which is impossible because it would mean $u' = +\infty$ $\Leb{1}$-a.e. on $(x_{0},b)$ according to \eqref{eq:uprimoT}. If instead $T(x_{0}^{-}) = - 1$, then $T(x) = -1$ for all $x \in (a, x_{0})$, yielding a similar contradiction. This proves that $J_{u}$ is empty, hence $u$ is continuous on $\Omega$. Analogously, let us assume there exists a Borel set $S$ such $|D^{c}u|(S) > 0$, then by \eqref{eq:Cantor_BV_one_dim} for $|D^{c}u|$-a.e. $y \in S$ we have $|\widetilde{T}(y)| = 1$. Arguing as above, we are led to the same contradiction. Therefore, we conclude that also $D^c u = 0$. This means that $Du = u' \Leb{1}$, where $u'$ satisfies \eqref{eq:uprimoT}, and, given that $T$ is non-decreasing, it is easy to see that also $u'$ is non-decreasing, since the map $$(-1, 1) \ni \xi \to \frac{\xi}{\sqrt{1 - \xi^2}}$$
is non-decreasing. This fact, in turn, implies that $u$ is convex and so locally Lipschitz. Finally, if $u_1, u_2$ are weak solutions of \eqref{eq:PMCM} with the same trace on $\partial \Omega = \{a, b\}$, then Corollary \ref{corol:uniqueness} implies $u_1 = u_2$ on $\Omega$.
\end{proof}

We now provide two examples of discontinuous and non-unique $1$-dimensional solutions.

\begin{example} \label{eq:one_dim_discont}
\rm
Let us set $\Omega = (-1,1)$ and $\mu = \delta_{0} - \frac{1}{2\sqrt{x}} \Leb{1}\restrict (0,1)$. Clearly, $\mu$ is admissible, since $\mu^+ = D \chi_{(0, 1)}$ and $\mu^- = f' \Leb{1}$, for $f(x) = \sqrt{x} \, \chi_{(0,1)}(x)$, so that $\mu^+$ and $\mu^-$ satisfy \eqref{eq:necess_cond}, and therefore Proposition \ref{prop:adm_PB_div}(5) implies $\mu^+, \mu^- \in BV(\Omega)^*$. Then, we notice that, if we define 
\[
T(x) = \mu\big((-1,x)\big) = 
\begin{cases}
0 & \text{if }x\in (-1,0]\\
1-\sqrt{x} & \text{if }x\in (0,1)\,,
\end{cases}
\]
we have $\div T = \mu$ and $\|T\|_{L^{\infty}(\Omega)} =1$. We also fix the Dirichlet conditions $u(x)\equiv 0$ if $x\le -1$ and $u(x) \equiv h$ if $x\ge 1$, for some given $h>0$ to be specified later. If $u\in BV(\Omega)$ is such that $(u,T)$ satisfies \eqref{pairing_eq}, then the absolutely continuous part of $Du$, denoted as $u'$, satisfies
\[
u'(x) = \frac{T(x)}{\sqrt{1-|T(x)|^{2}}} = 
\begin{cases}
0 & \text{if }x\in (-1,0)\\
\ds\frac{1-\sqrt{x}}{\sqrt{2\sqrt{x} - x}} & \text{if }x\in (0,1)\,.
\end{cases}
\]
In principle, the previous formula holds for a.e. $x\in (-1,1)$, however, $u$ is a classical - and thus smooth - solution on $(-1,0)\cup (0,1)$, therefore it is a pointwise identity on this subdomain. Let us set $u_{-1} = \lim_{x\to -1^{+}}u(x)$ and $u_{1} = \lim_{x\to 1^{-}}u(x)$, then we have
\[
u(x) = \begin{cases}
u_{-1} & \text{if }x\in (-1,0)\\
u_{1} + \ds\int_{1}^{x}\frac{1-\sqrt{t}}{\sqrt{2\sqrt{t}-t}}\, dt & \text{if }x\in (0,1)\,.
\end{cases}
\]
Consequently, we obtain
\[
u(0^{-}) := \lim_{x\to 0^{-}}u(x) = u_{-1}
\]
and
\[
u(0^{+}) := \lim_{x\to 0^{+}}u(x) = u_{1} - \int_{0}^{1}\frac{1-\sqrt{t}}{\sqrt{2\sqrt{t}-t}}\, dt = u_{1} - 2+\pi/2\,.
\]
By elementary computations, one can check that, for $\lambda = \lambda_\mu = \chi_{(0, 1)}$, \eqref{pairing_eq} holds if and only if $u(0^{-})\le u(0^{+})$, which amounts to assume that
\[
u_{1} \ge 2-\pi/2 + u_{-1}.
\] 
Indeed, if $u(0^{+}) - u(0^{-}) > 0$, then $\nu_u(0) = 1$ so that we get ${\rm Tr}^i(T, \nu_u)(0) = T(0+) =1$, and this satisfies \eqref{eq:convex_comb_eq_1_traces_1_dim}; while, if $u(0^{+}) = u(0^{-})$, then $0$ is not a jump point and \eqref{eq:Cantor_BV_one_dim} holds trivially, since the set $\{0\}$ is negligible for the Cantor part of $Du$.
A posteriori, this will be satisfied if we choose $h$ large enough and select the optimal values $u_{-1}$ and $u_{1}$ to minimize $\Jj_{\mu}$. We can now fix any bounded interval $B$ such that $\Omega \Subset B$, set 
\begin{equation*}
\phi(x) := \begin{cases} 0  & \text{ if } x \le 0 \\
h x  & \text{ if } x \in (0, 1) \\
h & \text{ if } x \ge 1
\end{cases}
\end{equation*}
and compute the value of $\Jj_{\mu}[u]$ for $u \in \cS(\phi)$ thanks to Theorem \ref{thm:Jsol}:
\begin{align*}
\Jj_{\mu}[u] - |B \setminus \Omega| &= |u_{-1}| + |u_{1}-h| + \sqrt{1+|Du|^{2}}(\Omega) + \int_{\Omega} u^{\lambda}\, d\mu\\
&= |u_{-1}| + |u_{1}-h| + (Du,T)_{\lambda}(\Omega) + \int_{\Omega}\sqrt{1-|T(x)|^{2}}\, dx + \int_{\Omega} u^{\lambda}\, d\mu\\
&= |u_{-1}| + |u_{1}-h| + u_{1}T(1^-) - u_{-1}T(-1^+) + \int_{\Omega}\sqrt{1-|T(x)|^{2}}\, dx\\
&= |u_{-1}| + |u_{1}-h| + 1 + \int_{0}^{1}\sqrt{2\sqrt{t}-t}\, dt\\
&= |u_{-1}| + |u_{1}-h| + \pi/2 + 1/3\,.
\end{align*}
The functional is minimized by choosing $u_{-1} = 0$ and $u_{1}=h$, hence the solution necessarily attains the prescribed boundary values. But then one has
\[
u(0^{-}) = 0\qquad \text{and}\qquad u(0^{+}) = h - 2+\pi/2\,.
\]
It is therefore enough to choose $h>2-\pi/2$ to obtain a solution $u(x)$ which is discontinuous at $x=0$.
\end{example}

\begin{example} \label{ex:one_dim_non_uniq_non_Schmidt}
\rm
Let us consider $\Omega = (0,3)$ and 
\[
\mu = \delta_{1} + \delta_{2} - \frac{1}{2\sqrt{x-1}} \Leb{1}\restrict (1,2)- \frac{1}{2\sqrt{x-2}} \Leb{1}\restrict (2,3)\,.
\]
By arguments and calculations completely analogous to the ones exploited in the previous example, one defines
\[
T(x) = \begin{cases}
0 & \text{if }x\in (0,1)\\
1-\sqrt{x-1} & \text{if }x\in (1,2)\\
1-\sqrt{x-2} & \text{if }x\in (2,3)
\end{cases}
\]
and
\[
u(x) = \begin{cases}
a & \text{if }x\in (0,1)\\
b+f(x-1) & \text{if }x\in (1,2)\\
c+f(x-2) & \text{if }x\in (2,3)\,,
\end{cases}
\]
where 
\[
f(u) = \int_{0}^{u}\frac{1-\sqrt{t}}{\sqrt{2\sqrt{t}-t}}\, dt
\]
and $a,b,c\ge 0$ satisfy $a<b<c-f(1) = c-2+\pi/2$. Given the Dirichlet boundary condition $u(x)=0$ for $x<0$ and $u(x)=h$ for $x>3$, with $h>0$ sufficiently large, we find that:
\begin{itemize}
\item the pair $(u,T)$ satisfies \eqref{L_infty_bound_eq}, \eqref{divergence_eq} and \eqref{pairing_eq};
\item assuming $h>4-\pi$, the functional $\Jj_{\mu}$ is minimized when $a=0$ and $c = h-2+\pi/2$, hence the Dirichlet condition is attained in the classical sense; moreover, the value of the functional does not depend on the parameter $b$;
\item letting $b$ vary in the interval $(0,h-4+\pi)$ we obtain a one-parameter family of solutions, all with the same trace on the boundary of $\Omega$.
\end{itemize}

We stress that this non-uniqueness phenomenon, among solutions which attain the boundary data, can arise in dimension $1$ only when the measure $\mu$ has a varying sign, due to Theorem \ref{thm:continuity_one_dim}. The situation changes in dimension $2$ and higher, as demonstrated by Example \ref{ex:two_balls}.

Finally, we point out that a small variation of this example shows that a converse of Lemma \ref{lem:Schmidt_non_extremality} cannot hold; that is, that there exists an admissible non-extremal measure such that its positive or negative part is not non-extremal. Indeed, let us consider
\[
\mu = \frac{3}{2} \delta_{1} + \delta_{2} - \frac{1}{2\sqrt{x-1}} \Leb{1}\restrict (1,2)- \frac{1}{2\sqrt{x-2}} \Leb{1}\restrict (2,3)\,.
\]
Arguing as at the beginning of Example \ref{eq:one_dim_discont}, it is easy to see that $\mu$ is admissible. In addition, we have $\mu = \div T$ for 
\[
T(x) = \begin{cases}
- \frac{3}{4} & \text{if }x\in (0,1)\\
\frac{3}{4}-\sqrt{x-1} & \text{if }x\in (1,2)\\
\frac{3}{4}-\sqrt{x-2} & \text{if }x\in (2,3)
\end{cases},
\] 
so that, thanks to the Gauss-Green formula \eqref{eq:GG_1}, we deduce that $\mu$ is non-extremal, with $L = \frac{3}{4} = \|T\|_{L^\infty(\Omega)}$. On the other hand, $\mu^+ = \frac{3}{2} \delta_{1} + \delta_{2}$, and so, if we choose $E = \left (\frac{1}{2}, \frac{5}{2} \right )$, we get
\begin{equation*}
\mu^+(E^1 \cap \Omega) = \frac{3}{2} + 1 \le L P(E) = 2 L
\end{equation*}
if and only if $L > \frac{5}{4} > 1$, so that $\mu^+$ does not satisfy the non-extremality condition \eqref{eq:subcritical_cond}.
\end{example}

\subsection{Weak solutions with radial symmetry on annuli}
Considering that the one-dimensional case has been fully treated in Section \ref{sec:n_1}, we let $n\ge 2$, and construct/characterize weak solutions with radial symmetry on annuli. 
\medskip

Before proceeding it is worth observing that the extra assumptions of rotational symmetry of $\Omega$, of the prescribed measure $\mu$, and of the weak Dirichlet datum $\phi$, imply that the auxiliary functional $\Jj_{\mu}$ admits rotationally invariant solutions in $\cS_{\phi}$. 

\begin{proposition}\label{prop:rotinv}
Let $\Omega\subset \R^{n}$ be an open bounded set with Lipschitz boundary compactly contained in a ball $B$, let $\mu\in \tMm(\Omega)$ be non-extremal and such that $\mu_2 = 0$, and let $\phi\in W^{1,1}(B)$ be fixed. Assume further that $\Omega$, $\mu$, and $\phi$ are rotationally invariant. Then the functional $\Jj_{\mu}$ admits rotationally invariant minimizers on $\cS_{\phi}$.
\end{proposition}
\begin{proof}
This is a consequence of the results of Section \ref{sec:gamma}. Indeed by inspecting the proof of Proposition \ref{prop:muapprox} we easily see that $\mu$ can be approximated by a sequence of radial functions $\{\mu_{j}\}_{j\in \N}$ for which Theorem \ref{thm:gammaconv} can be applied, giving a sequence $\{u_{j}\}_{j\in \N}$ of minimizers of $\Jj_{\mu}$ on $\cS_{\phi}$ which are smooth in $\Omega$, and thus necessarily radial (this easily follows from the convexity of $\Jj_{\mu}$ when restricted to continuous $BV$-functions). Then, by Proposition \ref{prop:gammacompact} we conclude the existence of a subsequence $\{u_{j_{k}}\}_{k\in\N}$ such that $u_{j_{k}}\to u \in \cS_{\phi}$ radial, and thus $\Gamma$-convergence ensures that $u$ is a radial minimizer of $\Jj_{\mu}$ on $\cS_{\phi}$. 
\end{proof}

\begin{remark}
It is not clear how to exclude the existence of non-radial minimizers of $\Jj_{\mu}$ on $\cS_{\phi}$ under the rotational-invariance assumptions of Proposition \ref{prop:rotinv} (owing to Theorem \ref{thm:maxprinc}, we know that such non-radial minimizers must have jump discontinuities in $\Omega$). Another related problem is to show that \emph{weak solutions are variational solutions} in the following sense. Assume that $u, v$ are a weak solutions of \eqref{eq:PMCM} with the same trace on $\de\Omega$, then is it true that $\Jj_{\mu}[u] = \Jj_{\mu}[v]$? These problems seem difficult even in the $2$-dimensional case with rotationally invariant data.
\end{remark}

\medskip

\begin{example}
\rm
We start by characterising radial solutions of \eqref{eq:PMCM} on annuli of the form $\Omega = B_{r_{3}}\setminus \overline{B_{r_{1}}}$ for given $0<r_{1}<r_{3}$, and where $\mu = \mu_{2}\Haus{n-1}\restrict \de B_{r_{2}}$ for some $r_{2} \in (r_{1},r_{3})$ and $\mu_{2}\in \R$. We thus assume that $u\in BV(\Omega)$ is a radial solution. We conveniently write $r=|x|$ for $x\in \Omega$, and denote by $u(r)$ the one-dimensional profile of the solution $u(x)$ (with a slight abuse of notation). Moreover, we denote by $\nu_{i}$ the outer unit normal to $\de B_{r_{i}}$.

Since $\mu = 0$ on $\Omega\setminus \de B_{r_{2}}$, the function $u$ is smooth in $\Omega\setminus \de B_{r_{2}}$ and therefore the vector field $T \in \DM^{\infty}(\Omega)$, such that the pair $(u,T)$ satisfies \eqref{L_infty_bound_eq}, \eqref{divergence_eq} and \eqref{pairing_eq}, 
satisfies \eqref{eq:Tuformula} for all $x\in \Omega\setminus \de B_{r_{2}}$, as well as \eqref{eq:T_nabla_uformula}. 
Moreover, by \eqref{divergence_eq} coupled with the radial symmetry of $u$, we infer that 
\begin{equation}\label{eq:TDu}
T(x) = \gamma_{i} \frac{x}{|x|^{n}}\qquad \forall\, x\in B_{r_{i+1}}\setminus \overline{B}_{r_{i}},\quad i=1,2\,,
\end{equation}
for suitable constants $\gamma_{1},\gamma_{2}\in \R$. Since $|T(x)|<1$ for all $x\in \Omega\setminus \de B_{r_{2}}$, due to \eqref{eq:Tuformula}, we infer that 
\begin{equation}\label{eq:stimalai}
|\gamma_{i}| \le r_{i}^{n-1},\qquad i=1,2\,.
\end{equation}
Then we notice that
\begin{align*}
\div T &= \big({\rm Tr}^i(T, \de B_{r_{2}}) - {\rm Tr}^e(T, \de B_{r_{2}})\big)\, \Haus{n-1}\restrict \de B_{r_{2}}\\
&= \frac{\gamma_{2} - \gamma_{1}}{r_{2}^{n-1}}\, \Haus{n-1}\restrict \de B_{r_{2}}\,,
\end{align*}
where ${\rm Tr}^{i,e}(T, \de B_{r_{2}})$ denote the interior and exterior normal traces of $T$ on $\de B_{r_{2}}$, defined as in Theorem \ref{thm:GG_app}.
Thus by combining the above formula with \eqref{divergence_eq}, we obtain
\begin{align}\label{eq:divTlai}
\gamma_{2} = \gamma_{1} + \mu_{2} r_{2}^{n-1}\,.
\end{align}
Hence, by \eqref{eq:TDu} and \eqref{eq:T_nabla_uformula} on $\Omega\setminus \de B_{r_{2}}$, we obtain 
\[
u'(r) = \frac{\gamma_{i}}{\sqrt{r^{2n-2} - \gamma_{i}^{2}}} \ \ \text{ for } r_{i}<r<r_{i+1}\,.
\]
In addition, we see that $D^c u = 0$ and $J_u \subset \de B_{r_2}$, since $u$ is smooth in $\Omega\setminus \de B_{r_{2}}$. Then, we notice that $\lambda_{\mu} \equiv 0$ if $\mu_{2}>0$, and $\lambda_\mu \equiv 1$ if $\mu_{2}<0$, so that, by exploiting \eqref{eq:convex_comb_eq_1_traces} in the case $\lambda = \lambda_\mu$, we get
\begin{align*}
\mu_{2}>0\ &\Longrightarrow\ {\rm Tr}^i(T, J_u) = 1 \ |D^j u|\text{-a.e. on } J_u \\
\mu_{2}<0\ &\Longrightarrow\ {\rm Tr}^e(T, J_u) = 1 \ |D^j u|\text{-a.e. on } J_u \,,
\end{align*}
Hence, if we assume that $u(r)$ jumps across $r=r_{2}$, then we have that
\begin{align*}
\mu_{2}>0\ \Longrightarrow\ \begin{cases}
\gamma_{2} = r_{2}^{n-1} & \text{if }\nu_{u} = \nu_{2}\\
\gamma_{1} = -r_{2}^{n-1} & \text{if }\nu_{u} = -\nu_{2}
\end{cases}
\ \Longrightarrow\ \gamma_{2} = r_{2}^{n-1} \text{ and } \nu_{u} = \nu_{2}
\end{align*}
because the second case is in contrast with the bound $|\gamma_{1}| \le r_{1}^{n-1} < r_{2}^{n-1}$. Similarly, we have
\begin{align*}
\mu_{2}<0\ \Longrightarrow\ \begin{cases}
\gamma_{1} = r_{2}^{n-1} & \text{if }\nu_{u} = \nu_{2}\\
\gamma_{2} = -r_{2}^{n-1} & \text{if }\nu_{u} = -\nu_{2}
\end{cases}
\ \Longrightarrow\ \gamma_{2} = -r_{2}^{n-1} \text{ and } \nu_{u} = -\nu_{2}\,.
\end{align*}
Summing up, if $u(r)$ is discontinuous at $r=r_{2}$ then we obtain
\[
\gamma_{2} = \sign(\mu_{2}) r_{2}^{n-1}
\]
and thus by \eqref{eq:divTlai} we find
\[
\gamma_{1} = (\sign(\mu_{2}) - \mu_{2}) r_{2}^{n-1}\,.
\]
This allows us to write the derivative $u'(r)$ on $(r_{1},r_{2})\cup (r_{2},r_{3})$ as
\[
u'(r) = \begin{cases}
\ds\frac{\sign(\mu_{2})(1-|\mu_{2}|) r_{2}^{n-1}}{\sqrt{r^{2n-2} - (1-|\mu_{2}|)^{2}r_{2}^{2n-2}}} & \text{if }r_{1}< r < r_{2}\\[10pt]
\ds\frac{\sign(\mu_{2}) r_{2}^{n-1}}{\sqrt{r^{2n-2} - r_{2}^{2n-2}}} & \text{if }r_{2}<r<r_{3}\,,
\end{cases}
\]
so that, since $u'(r)$ is defined for all $r\neq r_{2}$, the condition 
\[
1 - \left(\frac{r_{1}}{r_{2}}\right)^{n-1} \le |\mu_{2}| \le 1 + \left(\frac{r_{1}}{r_{2}}\right)^{n-1}
\]
must hold. We remark that the second inequality is also a necessary condition for the existence of solutions, while the first inequality says that, so that $u(r)$ jumps across $r_{2}$, the value of $|\mu_{2}|$ must be not too small. Moreover, the sign of $\mu$ determines whether the jump is ``up'' or ``down'', and one has 
\begin{equation*}
u'(r_{2}^{+}) = \begin{cases} + \infty & \text{ if } \mu_2 > 0, \\
- \infty & \text{ if } \mu_2 < 0,
\end{cases}
\end{equation*} 
in the right-limit sense.  We also note that the problem can be reduced to the discussion of the case $\mu_{2}>0$, because the other case can be obtained by passing from $u$ to $-u$. In particular, if we take
\begin{equation}\label{eq:mu2choice}
1-\left(\frac{r_{1}}{r_{2}}\right)^{n-1} < \mu_{2}< 1\,,
\end{equation}
we conclude that $u(r)$ can only jump ``up'' at $r=r_{2}$, and in this case $u(r)$ is monotone increasing on the whole $(r_{1},r_{3})$. 
\end{example}
\medskip

Now we add a further component to the measure $\mu$ to build a more sophisticated example which allows us to show the non-uniqueness of discontinuous solutions to \eqref{eq:PMCM} attaining the same Dirichlet boundary datum on $\de\Omega$.
 
\begin{example} \label{ex:two_balls}
\rm 
We redefine the annulus by fixing $0<r_{1}<r_{2}<r_{3}<r_{4}$, so that $\Omega = B_{r_{4}}\setminus\overline{B_{r_{1}}}$, then choose $\mu_{2}$ as in \eqref{eq:mu2choice} and $\mu_{3} = 1 - (r_{2}/r_{3})^{n-1}$. Then, we define
\[
\mu = \mu_{2}\Haus{n-1}\restrict \de B_{r_{2}} + \mu_{3} \Haus{n-1}\restrict \de B_{r_{3}}.
\] 
At the same time, we add a weak Dirichlet boundary condition at $r=r_{1},r_{4}$, that is, we prescribe boundary values $0$ at $r=r_{1}$ and $h$ at $r=r_{4}$, with $h>0$ to be chosen later. Then we consider a radial solution which additionally minimizes the functional $\Jj_{\mu}$ among radial $BV$ functions that vanish on $B_{r_{1}}$ and take the constant value $h$ outside $B_{r_{4}}$ (of course we implicitly assume that a ball $B$ of radius larger than $r_{4}$ has been fixed to properly define the functional $\Jj_{\mu}$). It is then clear that the values of the solution must be within the interval $[0,h]$. Indeed, the fact that $0 < \mu_2, \mu_3 < 1$ ensures that by truncating a function $u$ which goes above $h$ or below $0$ we get a better competitor.

We have
\begin{align*}
\Jj_{\mu}[u] &= |B\setminus \Omega| + \sqrt{1+|D u|^{2}}(\Omega \setminus (\de B_{r_{2}}\cup \de B_{r_{3}}))\\
&\qquad + n\omega_{n} \Big(r_{1}^{n-1} u(r_{1}^{+}) + r_{2}^{n-1}\big(u(r_{2}^{-})\mu_{2} + u(r_{2}^{+}) - u(r_{2}^{-})\big)\\
&\qquad\qquad + r_{3}^{n-1}\big(u(r_{3}^{-})\mu_{3} + u(r_{3}^{+}) - u(r_{3}^{-})\big)\\
&\qquad\qquad\qquad + r_{4}^{n-1}\big(h-u(r_{4}^{-})\big)\Big)\,.
\end{align*}
Thus, if we assume that $h - u(r_{4}^{-})>0$ we obtain for $0<\delta< h - u(r_{4}^{-})$ (also thanks to the choices of $\mu_{2}$ and $\mu_{3}$)
\begin{align*}
\Jj_{\mu}[u+\delta] - \Jj_{\mu}[u] &= n\omega_{n}\delta \big(r_{1}^{n-1} + r_{2}^{n-1}\mu_{2} + r_{3}^{n-1}\mu_{3} - r_{4}^{n-1}\big)\\
&< n\omega_{n}\delta \big(r_{1}^{n-1} + r_{3}^{n-1} - r_{4}^{n-1}\big)
\end{align*} 
which gives a contradiction with the minimization property of $u$ as soon as 
\begin{equation}\label{eq:rquattrocond}
r_{4}^{n-1} \ge r_{1}^{n-1} + r_{3}^{n-1}\,.
\end{equation}
From now on, we shall assume \eqref{eq:rquattrocond} as an extra condition involving the given radii, which guarantees that the boundary datum at $r=r_{4}$ is attained in the classical sense, that is $$u(r_{4}^{-}) = h.$$

With a similar computation, we claim that the boundary datum is attained also at $r=r_{1}$. Indeed, if we assume that $u(r_{1}^{+})>0$, we can fix $0<\delta<u(r_{1}^{+})$ and build the competitor $v_{\delta} := u - \delta \chi_{(r_{1},r_{2})}$, so that by easy computations we obtain
\begin{align*}
\Jj_{\mu}[v_{\delta}] - \Jj_{\mu}[u] &= n\omega_{n}\delta \big(r_{2}^{n-1}(1-\mu_{2}) - r_{1}^{n-1}\big) < 0
\end{align*}
because $\mu_{2}> 1 - (r_{1}/r_{2})^{n-1}$. This is a contradiction with the minimality of $u$, hence our claim is proved. Now, if $u(r)$ were continuous at both $r=r_{2},r_{3}$, we would deduce that 
\[
h = u(r_{4}^{-}) - u(r_{1}^{+}) = u(r_{4}^{-}) - u(r_{3}) + u(r_{3}) - u(r_{2}) + u(r_{2}) - u(r_{1}^{+}) \le C = C(r_{1},r_{2},r_{3},r_{4})
\]
where the constant $C$ results from a catenoidal bound of the global oscillation of the function on each interval $(r_{i}, r_{i+1})$. This means that if we choose $h$ larger than $C$, we obtain a contradiction with the simultaneous continuity of $u$ at $r=r_{2},r_{3}$. Thus we have that $u(r)$ jumps either at $r=r_{2}$, or at $r=r_{3}$ (and possibly at both radii). Let us assume that $u(r)$ jumps at $r=r_{2}$. In this case, by the above characterization of solutions in $B_{r_{3}}\setminus \overline{B_{r_{1}}}$ we deduce that 
\[
\gamma_{2} = r_{2}^{n-1}\quad\text{and}\quad \gamma_{3} = \gamma_{2} + \mu_{3}r_{3}^{n-1} = r_{3}^{n-1}\,,
\]
hence the restriction of $u(r)$ to the interval $(r_{2},r_{3})$ is uniquely determined up to suitable vertical translations, while the restriction of $u(r)$ to the interval $(r_{3},r_{4})$ is uniquely determined because we have proved that $u(r_{4}^{+}) = h$. It is then immediate to check that there exists a one-parameter family of solutions obtained by vertically translating the profile $u(r)$ restricted to $(r_{2},r_{3})$ until it preserves the global monotonicity of $u(r)$ on $(r_{1},r_{4})$. Moreover, the functional $\Jj_{\mu}$ attains the same value on this family of solutions. A similar conclusion is achieved if we assume that $u(r)$ jumps at $r=r_{3}$.

Summing up: under the conditions
\begin{itemize}
\item $1-(r_{1}/r_{2})^{n-1} < \mu_{2} < 1$ and $\mu_{3} = 1 - (r_{2}/r_{3})^{n-1}$,
\item $r_{4}^{n-1} \ge r_{1}^{n-1} + r_{3}^{n-1}$,
\item $h> C(r_{1},r_{2},r_{3},r_{4})$,
\end{itemize}
we can produce a one-parameter family of weak solutions to \eqref{eq:PMCM} attaining the same Dirichlet boundary datum on $\de\Omega$, see Figure \ref{fig:nonunique}. This non-uniqueness phenomenon confirms the optimality of the continuity assumption in Theorem \ref{thm:maxprinc}.
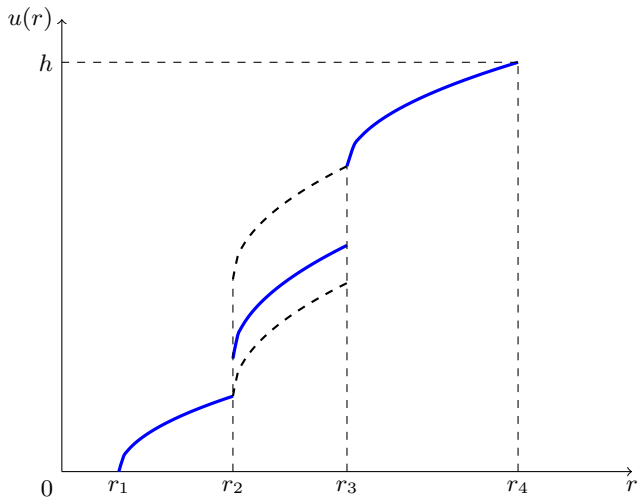
\begin{figure}[h]	
	\centering\footnotesize
	\begin{tikzpicture}[scale=1.5]
	\draw[->] (0,0) -- (5,0) node[below] {$r$};
	\draw[->] (0,0) -- (0,4) node[left] {$u(r)$};
	\node[below left] at (0,0) {$0$};
	\node[below] at (0.5,0) {$r_{1}$}; 
	\node[below] at (1.5,0) {$r_{2}$};
	\node[below] at (2.5,0) {$r_{3}$};
	\node[below] at (4,0) {$r_{4}$};
	\draw[scale=1,domain=0.5:1.5,smooth,variable=\x,blue,very thick] plot ({\x},{2*sqrt(\x-0.5)/3});
	\draw[scale=1,domain=1.5:2.5,smooth,variable=\x,blue,very thick] plot ({\x},{1+sqrt(\x-1.5)});
	\draw[scale=1,domain=1.5:2.5,smooth,variable=\x,thick,dashed] plot ({\x},{2/3+sqrt(\x-1.5)});
	\draw[scale=1,domain=1.5:2.5,smooth,variable=\x,thick,dashed] plot ({\x},{1.7+sqrt(\x-1.5)});
	\draw[dashed] (1.5,0) -- (1.5,1.7);
	\draw[scale=1,domain=2.5:4,smooth,variable=\x,blue,very thick] plot ({\x},{2.7+3*sqrt(\x-2.5)/4});
	\draw[dashed] (2.5,0) -- (2.5,2.7);
	\draw[dashed] (4,0) -- (4,{2.7+3*sqrt(1.5)/4});
	\draw[dashed] (4,{2.7+3*sqrt(1.5)/4}) -- (0,{2.7+3*sqrt(1.5)/4}) node[left] {$h$};
	\end{tikzpicture}
	\caption{The $1$-dimensional profile of a non-unique, radial solution of \eqref{eq:PMCM} with Dirichlet boundary conditions on $\Omega=B_{r_{4}}\setminus \overline{B_{r_{1}}}$. The solution attains the prescribed boundary values $0$ on $\de B_{r_{1}}$ and $h>0$ on $\de B_{r_{4}}$ in classical sense.}
		\label{fig:nonunique}
\end{figure}
\end{example}

\appendix

\section{A technical result on $BV$ functions}

We prove here a concentration property of the total variation of the gradient of a function $u\in BV_{\rm loc}(\Omega)$, which is needed in a covering argument within the proof of the lower semicontinuity of the functional $\Jj_\mu$, see Theorem \ref{thm:sci}. 

\begin{lemma}\label{lem:sopraupiueps}
Let $u\in BV_{\rm loc}(\Omega)$ and let $\eps>0$ be fixed, then for $|Du|$-a.e. $x\in \Omega$ one has 
\[
\liminf_{r\to 0^{+}}  \frac{\ds \int_{- \infty}^{u^{-}(x)-\eps} \Per\big(\{u>s\};B_{r}(x)\big)\, ds + \int_{u^{+}(x)+\eps}^{+\infty} \Per\big(\{u>s\};B_{r}(x)\big)\, ds }{|Du|(B_{r}(x))} = 0\,.
\]
\end{lemma}
\begin{proof}
We distinguish two cases. 

First, we consider the case $x\in \Omega\setminus S_{u}$, i.e., we assume $x$ in the support of the diffuse part of $Du$. In this case we have $u^+(x) = u^-(x) = \widetilde{u}(x)$. Here we follow the proof of \cite[Proposition 2.42]{AFP} coupled with \cite[Theorem 3.95]{AFP}. We introduce the rescaled function 
\[
u^{x,r}(y) = \frac{u(x+ry) - u_{B_{r}(x)}}{r^{1-n} |Du|(B_{r}(x))}\,,
\]
where $u_{B_{r}(x)}$ denotes the average of $u$ on $B_{r}(x)$. We observe that $u^{x,r}$ has zero average on $B_1$ and $|Du^{x,r}|(B_{1}) = 1$ for any $r>0$. Therefore the family of rescaled functions $\{u^{x,r}:\ r>0\}$ is sequentially relatively compact, as a consequence of Poincar\'e-Wirtinger inequality in $BV(B_{1})$. Then, we denote by $T(u,x)$ the collection of all sequential limits in the $L^{1}_{\rm loc}$-topology of $u^{x,r}$ as $r\to 0$. 
By \cite[Corollary 2.43, Theorem 2.44 and Theorem 3.95]{AFP} combined with the lower-semicontinuity of the total variation, we can assume w.l.o.g. that
\begin{itemize}
\item[(a)] $x$ is a Lebesgue point for $u$, i.e., $u_{B_{r}(x)} \to \widetilde{u}(x)$ as $r\to 0$;
\item[(b)] for any $t\in (0,1)$ there exists $v_{t}\in T(u,x)$ and $r_{i}\to 0$ as $i\to+\infty$, such that $u^{x,r_{i}}\to v_{t}$ in $L^{1}(B_{1})$, $Du^{x,r_{i}} \weakto Dv_{t}$ and $|Du^{x,r_{i}}| \weakto |Dv_{t}|$ in $\Mm(B_{1}; \R^n)$ and $\Mm(B_{1})$, respectively, and moreover 
\[
t^{n}\le |Dv_{t}|(\overline{B_{t}}) \le |Dv_{t}|(B_{1}) \le 1.
\]
\end{itemize}
We now argue by contradiction and assume there exist $\eta>0$ and $\overline{r}>0$, such that 
\begin{equation}\label{eq:stimacoareacontra}
\int_{\widetilde{u}(x)+\eps}^{+\infty} \Per\big(\{u>s\};B_{r}(x)\big)\, ds + \int_{-\infty}^{\widetilde{u}(x) - \eps} \Per\big(\{u>s\};B_{r}(x)\big)\, ds \ge \eta |Du|(B_{r}(x))
\end{equation} 
for all $0<r<\overline{r}$. Fix $t = (1-\eta/3)^{1/n}$ and consider $v=v_{t}\in T(u,x)$ and the infinitesimal sequence $r_{i}\to 0$ as in (b). We can thus find $M=M(\eta)\in \R$ such that considering the truncation $T_M(v)$ we have
\[
|DT_M(v)|(B_{1}) \ge 1-\eta/2\,.
\]
By the coarea formula, Fatou's lemma, and the lower-semicontinuity of the perimeter, we have that
\begin{align*}
1 &= \lim_{i\to+\infty} |Du^{x,r_{i}}|(B_{1})\\
&= \lim_{i\to+\infty}\int_{-\infty}^{+\infty} \Per(\{u^{x,r_{i}}>\tau\};B_{1})\, d\tau\\
&= \lim_{i\to+\infty}\int_{-\infty}^{-M} \Per(\{u^{x,r_{i}}>\tau\};B_{1})\, d\tau + \int_{-M}^{M} \Per(\{u^{x,r_{i}}>\tau\};B_{1})\, d\tau + \int_{M}^{+\infty} \Per(\{u^{x,r_{i}}>\tau\};B_{1})\, d\tau\\
&\ge \int_{-M}^{M} \Per(\{T_M(v)>\tau\};B_{1})\, d\tau + \liminf_{i\to+\infty} \int_{-\infty}^{-M} \Per(\{u^{x,r_{i}}>\tau\};B_{1})\, d\tau + \int_{M}^{+\infty} \Per(\{u^{x,r_{i}}>\tau\};B_{1})\, d\tau\\
&= |D T_M(v)|(B_{1}) + \liminf_{i\to+\infty} \int_{-\infty}^{-M} \Per(\{u^{x,r_{i}}>\tau\};B_{1})\, d\tau + \int_{M}^{+\infty} \Per(\{u^{x,r_{i}}>\tau\};B_{1})\, d\tau\\
&\ge 1-\eta/2 + \liminf_{i\to+\infty} \int_{-\infty}^{-M} \Per(\{u^{x,r_{i}}>\tau\};B_{1})\, d\tau + \int_{M}^{+\infty} \Per(\{u^{x,r_{i}}>\tau\};B_{1})\, d\tau\,.
\end{align*} 
We thus infer that
\[
\liminf_{i\to+\infty} \int_{-\infty}^{-M} \Per(\{u^{x,r_{i}}>\tau\};B_{1})\, d\tau + \int_{M}^{+\infty} \Per(\{u^{x,r_{i}}>\tau\};B_{1})\, d\tau \le \eta/2
\]
which, in terms of the function $u$, can be rewritten as
\begin{align}\nonumber
\liminf_{i\to+\infty} \frac{1}{|Du|(B_{r_{i}}(x))} & \Bigg( \int_{-\infty}^{u_{B_{r_{i}}(x)} - Mr_{i}^{1-n}|Du|(B_{r_{i}}(x))} \Per(\{u>s\};B_{r_{i}}(x))\, ds \\\label{eq:stimacoareaeta}
&\qquad\qquad + \int_{u_{B_{r_{i}}(x)} + Mr_{i}^{1-n}|Du|(B_{r_{i}}(x))}^{+\infty} \Per(\{u>s\};B_{r_{i}}(x))\, ds \Bigg) \le \eta/2\,.
\end{align}
On the other hand, we have 
\[
u_{B_{r_{i}}(x)} \pm Mr_{i}^{1-n}|Du|(B_{r_{i}}(x)) \to \widetilde{u}(x)\qquad \text{as $i\to+\infty$}\,,
\]
hence by \eqref{eq:stimacoareaeta} we finally deduce
\[
\liminf_{i\to+\infty} \frac{1}{|Du|(B_{r_{i}}(x))} \left ( \int_{-\infty}^{\widetilde{u}(x)-\eps} \Per(\{u>s\};B_{r_{i}}(x))\, ds + \int_{\widetilde{u}(x)+\eps}^{+\infty} \Per(\{u>s\};B_{r_{i}}(x))\, ds \right ) \le \eta/2\,,
\]
which is in contrast with \eqref{eq:stimacoareacontra}. This proves the thesis in this first case.

Then we consider the second case, i.e., $x\in J_{u}$. In this case we exploit the characterisation of $Du\restrict J_{u}$ and of the limit of the rescaled function $u_{x,r}(y) := u(x+ry)$ as $r\to 0$ (see \cite[Theorem 3.77]{AFP}). Indeed we have that $u_{x,r}\to v$ in $L^{1}_{\rm loc}(\R^{n})$ as $r\to 0$, with $v(y)$ defined for a.e. $y\in \R^{n}$ as
\[
v(y) = \begin{cases}
u^{+}(x) & \text{if }\langle y,\nu_{u}(x)\rangle \ge 0\\
u^{-}(x) & \text{otherwise}\,,
\end{cases}
\]
where $\nu_{u}$ is the Borel vector field defined for $|Du|$-almost every $x\in \Omega$ by 
\[
\nu_{u}(x) = \lim_{r\to 0}\frac{Du(B_{r}(x))}{|Du|(B_{r}(x))}\,.
\]
Since $|Du_{x,r}| \weakto |Dv|$ as $r\to 0$ in $\mathcal{M}_{\rm loc}(\R^n)$, we conclude that for almost all $R>0$ we have that $|Du_{x,r}|(B_{R}) \to |Dv|(B_{R})$ as $r\to 0$. Arguing as in the first case, this implies that
\begin{align*}
|Dv|(B_{R}) &= \int_{u^{-}(x)}^{u^{+}(x)} \Per(\{v>t\};B_{R})\, dt \le \int_{u^{-}(x)}^{u^{+}(x)} \liminf_{r\to 0}\Per(\{u_{x,r}>t\};B_{R})\, dt\\
&\le \liminf_{r\to 0} \int_{u^{-}(x)}^{u^{+}(x)} \Per(\{u_{x,r}>t\};B_{R})\, dt \le \liminf_{r\to 0} \int_{-\infty}^{+\infty} \Per(\{u_{x,r}>t\};B_{R})\, dt\\
&= \liminf_{r\to 0} |Du_{x,r}|(B_{R}) = |Dv|(B_{R})\,.
\end{align*}
Therefore all previous inequalities are equalities, and in particular, we infer that 
\[
\liminf_{r\to 0} \int_{-\infty}^{u^{-}(x)-\eps} \Per(\{u_{x,r}>t\};B_{R})\, dt + \int_{u^{+}(x)+\eps}^{+\infty} \Per(\{u_{x,r}>t\};B_{R})\, dt = 0\,,
\]
that is, 
\[
\liminf_{r\to 0} \int_{-\infty}^{u^{-}(x)-\eps} r^{1-n}\Per(\{u>t\};B_{Rr}(x))\, dt + \int_{u^{+}(x)+\eps}^{+\infty} r^{1-n}\Per(\{u>t\};B_{Rr}(x))\, dt = 0\,.
\]
But since $|Du|(B_{Rr})$ is asymptotic to $\omega_{n-1}(u^{+}(x)-u^{-}(x)) r^{n-1}$ as $r\to 0$, we deduce that
\[
\liminf_{r\to 0} \frac{1}{|Du|(B_{Rr}(x))} \left ( \int_{-\infty}^{u^{-}(x)-\eps} \Per(\{u>t\};B_{Rr}(x))\, dt + \int_{u^{+}(x)+\eps}^{+\infty} \Per(\{u>t\};B_{Rr}(x))\, dt \right ) = 0\,.
\]
This concludes the proof of the lemma, given that $|Du|(S_u \setminus J_u) = 0$ thanks to \cite[Lemma 3.76 and Theorem 3.78]{AFP}.
\end{proof}

\end{document}